\documentclass[11pt, leqno]{amsart}
\usepackage{amsmath}
\usepackage{amsthm}
\usepackage{amsrefs}
\usepackage{qsymbols}
\usepackage{latexsym}
\usepackage{chngcntr}
\usepackage{color}
\usepackage[noadjust]{cite}
\usepackage{paralist}
\usepackage{graphicx}
\usepackage{esint}

\newtheorem{theorem}{Theorem}[section]

\newtheorem{lemma}[theorem]{Lemma}
\newtheorem{proposition}[theorem]{Proposition}
\newtheorem{corollary}[theorem]{Corollary}
\theoremstyle{definition}
\newtheorem{definition}[theorem]{Definition}
\newtheorem{remark}[theorem]{Remark}

\newcommand{\IR}{\mathbb{R}}

\newcommand{\IC}{\mathbb{C}}
\newcommand{\IN}{\mathbb{N}}
\newcommand{\IZ}{\mathbb{Z}}

\newcommand{\Lop}{{\mathcal{L}}}
\newcommand{\Isom}{{\mathrm{Isom}}}

\newcommand{\Sec}{\mathrm{S}}

\newcommand{\cX}{\mathcal{X}}
\newcommand{\cA}{\mathcal{A}}

\newcommand{\cC}{\mathcal{C}}
\newcommand{\cE}{\mathcal{E}}
\newcommand{\cL}{\mathcal{L}}
\newcommand{\fa}{\mathfrak{a}}

\newcommand{\cS}{\mathcal{S}}
\newcommand{\cR}{\mathcal{R}}
\newcommand{\fu}{\mathfrak{u}}
\newcommand{\ff}{\mathfrak{f}}
\newcommand{\ba}{{\bar a}}
\newcommand{\bb}{{\bar b}}
\newcommand{\bbf}{{\bar f}}
\newcommand{\bg}{{\bar g}}
\newcommand{\bh}{{\bar h}}
\newcommand{\bu}{{\bar u}}
\newcommand{\bv}{{\bar v}}
\def\da{\delta\!a}
\def\db{\delta\!b}
\def\df{\delta\!f}
\def\dg{\delta\!g}
\def\dh{\delta\!h}
\def\du{\delta\!u}
\def\dv{\delta\!v}

\renewcommand{\L}{\mathrm{L}}

\newcommand{\B}{\mathrm{B}}
\newcommand{\E}{\mathrm{E}}
\newcommand{\F}{\mathrm{F}}
\newcommand{\C}{\mathrm{C}}
\newcommand{\W}{\mathrm{W}}

\renewcommand{\d}{\mathrm{d}}

\newcommand{\e}{\mathrm{e}}
\newcommand{\eps}{\varepsilon}

\newcommand{\wt}{\widetilde}

\DeclareMathOperator{\divergence}{div}
\DeclareMathOperator{\diva}{div_A}
\DeclareMathOperator{\curl}{curl}

\DeclareMathOperator{\adj}{adj}
\renewcommand{\Re}{\operatorname{Re}}

\DeclareMathOperator{\Id}{Id}

\newcommand{\dom}{\mathcal{D}}

\newcommand{\with}{\quad\hbox{with}\quad}
\newcommand{\andf}{\quad\hbox{and}\quad}

\numberwithin{equation}{section}

\title[The compressible Navier--Stokes system in critical spaces]
{Critical regularity issues for  the compressible Navier--Stokes system in 
bounded domains}
\author{Rapha\"el Danchin}
\author{Patrick Tolksdorf}
\address{Universit\'e Paris-Est, LAMA, UMR 8050, 61 Avenue du G\'en\'eral de Gaulle, 94010 Cr\'eteil, Cedex, France}
\address{Institut f\"ur Mathematik, Johannes Gutenberg-Universit\"at Mainz, Staudingerweg 9, 55099 Mainz, Germany}
\email{danchin@univ-paris12.fr}
\email{tolksdorf@uni-mainz.de}
\thanks{}

\date{\today}

\begin{document}
\begin{abstract}
We are concerned with the barotropic compressible Navier--Stokes system
in a bounded domain of $\IR^d$ (with $d\geq2$).  In a  \emph{critical regularity setting}, 
we establish local well-posedness 
for large data with no vacuum and global well-posedness for small 
perturbations of a stable constant equilibrium state.

Our results rely on new maximal regularity estimates - of independent interest - for the 
semigroup of  the Lam\'e operator, and of 
the linearized compressible Navier--Stokes equations. 
\end{abstract}
\maketitle

\section{Introduction}

We are concerned with the following \emph{barotropic compressible Navier--Stokes system}
in a $\cC^\infty$ 
 bounded domain $\Omega$ of $\IR^d$, $d \geq 2$:
\begin{equation}\label{eq:CNS}
\left\{\begin{aligned}
&\partial_t\rho+ \divergence(\rho u)=0\quad&\hbox{in }\ \IR_+\times\Omega,\\
&\partial_t(\rho u) +\divergence(\rho u\otimes u) -2\divergence(\mu D(u))-\nabla(\lambda\divergence u)
+\nabla P=0&\quad\hbox{in }\ \IR_+\times\Omega,\\
&u=0&\quad\hbox{on }\ \IR_+\times\partial\Omega,\\
&(\rho,u)|_{t=0}=(\rho_0,u_0)&\quad\hbox{in }\ \Omega.
\end{aligned}\right.
\end{equation}
The unknowns are the (scalar nonnegative) density $\rho=\rho(t,x)$
and the vector-field $u=u(t,x).$ 
The notation $D(u)$ stands for the symmetric part of the Jacobian matrix of $u.$
The viscosity coefficients $\lambda$ and $\mu$ are  smooth functions 
of $\rho$ satisfying $\mu>0$ and  $\lambda+2\mu>0.$
We shall often assume  (with no loss of generality)
that  the average value of the density on $\Omega,$ a conserved quantity,  is  equal to  $1.$ 
\medbreak
The mathematical study of the Cauchy problem (or initial boundary value problem)
for the compressible Navier--Stokes system  has been initiated sixty years
ago with the pioneering works by J.\@ Serrin~\cite{Serrin} and J.\@ Nash~\cite{Nash}
who established the local-in-time existence and uniqueness of classical solutions. 
In the case $\Omega=\IR^3,$ the global existence  of strong solutions 
with Sobolev regularity has been first proved by A.\@ Matsumura and T.\@ Nishida 
\cite{Matsumura_Nishida}, for small perturbations 
of a constant state $(\rho,u)=(\bar\rho,0)$ under the stability condition $P'(\bar\rho)>0.$
The proof was based on subtle energy estimates that enabled the authors
to pinpoint some $\L^2$-in-time integrability for both the density and the velocity, 
as well as algebraic time decay estimates. 

With completely different methods based  on parabolic maximal regularity in the framework of Lebesgue spaces, local existence has been established by V.\@ Solonnikov~\cite{Solonnikov}
for general data with no vacuum (see also the more recent work by  
the first author~\cite {Danchin} where critical regularity is almost achieved)
as well as global existence for small perturbations of $(\bar\rho,0)$
(see
~\cite{MZ02}, and~\cite{Kotschote}). 
\medbreak
In the present paper, we want to  recover the classical results of strong solutions 
 for~\eqref{eq:CNS} in the bounded domain case  \emph{within a critical regularity setting}, that is,
 in functional spaces that are invariant by the following rescaling for all $\ell>0$:
 \begin{equation}\label{eq:scaling} \bigl(\rho_0(x),u_0(x)\bigr)
\leadsto \bigl(\rho_0(\ell x),\ell u_0(\ell x)\bigr)\andf  \bigl(\rho(t,x),u(t,x)\bigr)\leadsto 
\bigl(\rho(\ell^2t,\ell x),\ell u(\ell^2t,\ell x)\bigr)\cdotp
\end{equation}
Observe that the above rescaling  leaves the whole system 
invariant, up to a change of the pressure term (provided the fluid domain is dilated accordingly,
of course). As first noticed by H.\@ Fujita and T.\@ Kato in~\cite{FK} for the incompressible
Navier--Stokes equations, working in scaling invariant spaces is the key to getting
optimal well-posedness results.  
\medbreak
Our main goal here  is to prove the following type of statements:
\begin{itemize}
\item local well-posedness for  general  data having critical regularity
and such that $\rho_0>0$; 
\item if, in addition, $P'(1)>0,$ global well-posedness 
for data $(\rho_0,u_0)$ that are small perturbations of $(1,0)$ (for some  norm
having the invariance of the first part of~\eqref{eq:scaling}). 
\end{itemize}
When the fluid domain is  the whole space,
a number of results in that spirit have been 
established, and the critical norms are  always
      built upon homogeneous Besov spaces \emph{with last index equal to~$1$.} 
 More precisely,  it has been first observed in~\cite{Danchin0}
      that one can take any data such that  $\rho_0-1$ is small in
        $\dot \B^{d/2-1}_{2,1}(\IR^d)\cap \dot \B^{d/2}_{2,1}(\IR^d),$
        and $u_0$ is small in $\dot \B^{d/2-1}_{2,1}(\IR^d).$
        Later works (see, e.g.,~\cite{Charve_Danchin},~\cite{Chen_Miao_Zhang})
        pointed out that it is actually enough  to  assume  the high frequencies of 
     the data to be in the larger space $\dot \B^{d/p}_{p,1}(\IR^d)\times \dot \B^{d/p-1}_{p,1}(\IR^d)$
     for some $p$ in the range $\bigl(2,\min(4,\frac{2d}{d-2})\bigr)\cdotp$ 
          \medbreak
         Here we aim at extending those results to the
     physically relevant case where the fluid domain is bounded and the velocity
     vanishes at the boundary. Compared to works in the whole space, 
     the expected  difficulty  is that one can no longer  
      use  techniques based on the Fourier transform  to investigate~\eqref{eq:CNS} 
     (in particular, global results of~\cite{Danchin0} were based  on a decomposition 
     into low and high frequencies of the solution). 
      Whether one can adapt those techniques to the domain case
      is unclear.      
     In the present paper, we focus on the bounded domain case which is expected
     to be easier than the unbounded domain case since, somehow, 
     low frequencies do not  exist 
    (therefore,  prescribing different regularity for low and high frequencies is irrelevant). 
      
 Since the linearized compressible Navier--Stokes system may be associated to 
 an analytic semigroup in suitable functional spaces, using maximal $\L^q$-regularity 
 seems to be an acceptable substitute to Fourier analysis. However, as already pointed out
 in previous works (see, e.g.,~\cite{Danchin}), reaching critical regularity 
 within the classical theory would require \emph{maximal $\L^1$-regularity}, 
 which is false in the setting of Lebesgue or
 Sobolev spaces for instance.
\smallbreak 
 For the reader's convenience, let us  briefly recall what maximal regularity is. 
Let $X$ be a Banach space and  $- A : \dom(A) \subset X \to X,$  the generator of a bounded analytic semigroup $(T(t))_{t\geq0}$ on  $X.$ Consider for $f \in \L^q (\IR_+ ; X)$, $1 \leq q \leq \infty$, the abstract Cauchy problem
\begin{align*}
\left\{ \begin{aligned}
 u^{\prime} (t) + A u(t) &= f (t) \qquad (t > 0), \\
 u(0) &= 0.
\end{aligned} \right.
\end{align*} 
By virtue of~\cite[Prop.~3.1.16]{Arendt_Batty_Hieber_Neubrander} the unique mild solution to this problem is given by the variation of constants formula
\begin{align*}
 u(t) = \int_0^t T(t-\tau) f(\tau) \; \d\tau \qquad (t > 0).
\end{align*}
We say that $A$ has \emph{maximal $\L^q$-regularity} if, for every $f \in \L^q (\IR_+ ; X),$ it holds for almost every  $t > 0$ that $u(t) \in \dom(A),$ and $A u \in \L^q (\IR_+ ; X)$. Notice that in this case also $u^{\prime} \in \L^q (\IR_+ ; X)$ and that the closed graph theorem implies the existence of a constant $C > 0$ such that for all $f \in \L^q (\IR_+ ; X)$ it holds
\begin{align*}
 \| u^{\prime} , A u \|_{\L^q (\IR_+ ; X)} \leq C \| f \|_{\L^q (\IR_+ ; X)}.
\end{align*}
See the monographs of Denk, Hieber, and Pr\"uss~\cite{Denk_Hieber_Pruess} and of Kunstmann and Weis~\cite{Kunstmann_Weis} for further information. Our aim here is to adapt  an argument of real interpolation that originates from  
Da Prato-Grisvard's work in~\cite{DaPrato-Grisvard} so as to
reach the endpoint $q=1$ that turns out to be the key to proving global-in-time
results in critical regularity framework (in this respect, see also our recent paper \cite{DHMT}). 

We perform the analysis first  for the semigroup associated to the Lam\'e operator 
(namely the linearization of the velocity equation if neglecting the pressure term), 
so as to get a local well-posedness result for general data
with critical regularity, then  for the linearization of the whole system~\eqref{eq:CNS} about $(\rho,u)=(1,0)$ to obtain a global result. 
 
Back to the nonlinear system, one cannot just push all nonlinear terms 
to the right-hand side and  bound them according to Duhamel's formula, though. 
The troublemaker is the convection 
term in the density equation, namely $u\cdot\nabla\rho,$ that causes
a loss of one derivative (this reflects the fact that the system under
consideration is partly hyperbolic). The way to overcome the difficulty is well-known:
it is called Lagrangian coordinates. Indeed, if rewriting 
\eqref{eq:CNS} in Lagrangian coordinates, then one just has to consider
the evolution equation for the velocity which is of parabolic type. 
Therefore, not only the loss of derivative may be avoided, 
but also the solution may be obtained (either locally for large data, or globally for small data) 
by means of the contraction mapping argument in Banach spaces. 
\medbreak
Let us now come to  the main results of the paper. 
\begin{theorem}\label{Thm:local}
Assume that $\Omega$ is a smooth bounded domain of $\IR^d$
($d\geq2$) and let $p$ be in $(d-1,2d).$
Then,  for all initial densities $\rho_0\in\B^{d/p}_{p,1}(\Omega),$
 positive and bounded away from zero,  and all $u_0\in\B^{d/p-1}_{p,1}(\Omega;\IR^d),$ 
System~\eqref{eq:CNS} admits a unique  solution $(\rho,u)$
on some nontrivial time interval $[0,T],$  such that 
$(\rho,u)\in\cC_b([0,T];\B^{d/p}_{p,1} (\Omega)\times \B^{d/p-1}_{p,1} (\Omega ; \IR^d))$ and
$$
 (\rho , u) \in \W^{1 , 1} (\IR_+ ; \B^{{d}/{p}}_{p,1} (\Omega) \times\B^{{d}/{p} - 1}_{p,1} (\Omega ; \IR^d)) \cap \L^1 (\IR_+ ; \B^{{d}/{p}}_{p,1} (\Omega)\times\B^{{d}/{p} + 1}_{p,1} (\Omega ; \IR^d))\cdotp $$
Furthermore, $\inf_{(t,x)\in[0,T]\times\Omega} \:\rho(t,x)>0$
and the average of $\rho$ is time independent. 
\end{theorem}
Proving a  global result for small perturbations of a stable constant state 
is based on  maximal regularity estimates for the linearized compressible Navier--Stokes system
(where $\mu'=\lambda+\mu$):
\begin{equation}\label{eq:LCNS}
\left\{\begin{aligned}
&\partial_ta+ \divergence u=f\quad&\hbox{in }\ \IR_+\times\Omega,\\
&\partial_t u- \mu \Delta u - \mu' \nabla\divergence u
+\nabla a=g&\quad\hbox{in }\ \IR_+\times\Omega,\\
&u=0&\quad\hbox{on }\ \IR_+\times\partial\Omega,\\
&(a,u)|_{t=0}=(a_0,u_0)&\quad\hbox{in }\ \Omega.
\end{aligned}\right.
\end{equation}
The following statement extends the work by 
P.B.\@ Mucha and W.\@ Zaj\c aczkowski~\cite{MZ02} to  the endpoint case where 
the time Lebesgue exponent is equal to $1,$
and also provides exponential decay for the solutions of the system.  

\begin{theorem}\label{thm:linear}
Take initial data  $(a_0,u_0)$ in $\B^{s+1}_{p,1} (\Omega)\times\B^{s}_{p,1}(\Omega ; \IR^d)$ and source terms $(f,g)$ in $\L^1(\IR_+;\B^{s+1}_{p,1} (\Omega)\times\B^{s}_{p,1} (\Omega ; \IR^d))$
with $(s,p)$ satisfying 
$$1<p<\infty\andf\max\biggl(\frac1p,\frac dp-\frac d2\biggr)-1<s<\frac1p\cdotp$$
Assume also that the average of $a_0$ and of $f(t)$ (for a.e.\@ $t>0$) is zero.
Then, System~\eqref{eq:LCNS} has a unique global 
solution $(a,u)$ in the space 
\begin{align}\label{eq:Ep}
\E_p^s:= \W^{1 , 1} (\IR_+ ; \B^{s+1}_{p,1} (\Omega) \times\B^{s}_{p,1} (\Omega ; \IR^d)) \cap \L^1 (\IR_+ ; \B^{s+1}_{p,1} (\Omega)\times\B^{s + 2}_{p,1} (\Omega ; \IR^d))\cdotp
\end{align}
Additionally,  there exist two positive constants $c$ and $C$ depending only on  $p,$ $\Omega,$
$\mu$, and $\mu'$ such that if $\e^{ct}(f,g) \in \L^1(\IR_+;\B^{s+1}_{p,1}(\Omega)\times\B^{s}_{p,1}(\Omega ; \IR^d))$, then
\begin{equation}\label{eq:decay}
\|\e^{ct}(a,u)\|_{\E^s_p}\leq C\biggl(\|(a_0,u_0)\|_{\B^{s+1}_{p,1}\times\B^{s}_{p,1}} 
+ \|\e^{ct}(f,g)\|_{\L^1(\IR_+;\B^{s+1}_{p,1}\times\B^{s}_{p,1})}\biggr)\cdotp\end{equation} 
\end{theorem}

After recasting System \eqref{eq:CNS} in Lagrangian coordinates, combining the above
result with  nonlinear estimates allows to get the following  global well-posedness result
for critical regularity:
\begin{theorem}\label{Thm:global}
Let $\Omega,$ $p$, and $d$ be as in Theorem~\ref{Thm:local}
and assume in addition that $P'(1)>0.$  Let $\rho_0\in \B^{d/p}_{p,1}(\Omega)$
with average $1$ and   $u_0\in\B^{d/p-1}_{p,1}(\Omega;\IR^d).$
There exists a constant 
$\alpha=\alpha(\lambda,\mu,p,d,P,\Omega) > 0$ such that if $a_0:=\rho_0-1$ and $u_0$
satisfy 
\begin{equation}\label{eq:smalldata}
\|a_0\|_{\B^{d/p}_{p,1}}+ \|u_0\|_{\B^{d/p-1}_{p,1}}\leq \alpha,
\end{equation}
then System~\eqref{eq:CNS} admits a unique global solution $(\rho,u)$
with $(a,u):=(\rho-1,u)$ 
in the maximal regularity space $\E_p:=\E_p^{d/p-1}.$
In addition, there exists  $c>0$ depending only on the parameters of the system, 
on $p$, and on $\Omega$ such that 
  $(a,u)$ fulfills: 
$$\| \e^{c t} (a,u)\|_{\E_p}\leq C \bigl(\|a_0\|_{\B^{d/p}_{p,1}}
+ \|u_0\|_{\B^{d/p-1}_{p,1}}\bigr)\cdotp$$
\end{theorem}
The rest of the paper unfolds as follows. 
The next two sections are dedicated to the ``linear study" namely 
proving maximal regularity results first for the Lam\'e operator, 
and next  for the linearized compressible Navier--Stokes system. 
In Section~\ref{s:global}, we prove our main global existence result. 
In Section~\ref{s:local}, we establish 
local existence results with no smallness condition on the velocity,
first in the easy case where the initial density is close to a constant
and, next, assuming only that the density is bounded away from zero. 
Some technical results are recalled/proved in Appendix.

\subsection*{Acknowledgement}
The authors have been partially supported by ANR-15-CE40-0011.


\section{Some background from semigroup theory}
\label{Some background from semigroup theory}
We use this section to introduce the basic functional analytic notions and arguments that are crucial for the theory that is developed afterwards. \par
Let $X$ denote a Banach space over the complex field. For $\theta \in (0 , \pi)$ define the sector $\Sec_{\theta}$ in the complex plane
\begin{align*}
 \Sec_{\theta} := \{ z \in \IC \setminus \{ 0 \} : \lvert \arg(z) \rvert < \theta \},
\end{align*}
and set  $\Sec_0 := (0 , \infty)$. 
\medbreak
The  standard definition of (bounded) analytic semigroups reads:
\begin{definition}
\label{Def: Analytic semigroup}
A family $(T(z))_{z \in \Sec_{\theta} \cup \{ 0 \}} \subset \Lop(X)$, $\theta \in (0 , \pi / 2]$, is called an \textit{analytic semigroup of angle $\theta$} if
\begin{enumerate}
 \item $T(0) = \Id$ and $T(z + w) = T(z) T(w)$ for all $z , w \in \Sec_{\theta}$; \label{Item: Semigroup law}
 \item the map $z \mapsto T(z)$ is analytic in $\Sec_{\theta}$; \label{Item: Analyticity}
 \item $\lim_{\Sec_{\vartheta} \ni z \to 0} T(z) x = x$ for all $x \in X$ and all $0 < \vartheta < \theta$.
\end{enumerate}
If in addition
\begin{enumerate}
 \item[(4)] $\| T(z) \|_{\Lop(X)}$ is bounded in $\Sec_{\vartheta}$ for all $0 < \vartheta < \theta$,
\end{enumerate}
the family $(T(z))_{z \in \Sec_{\theta} \cup \{ 0 \}}$ is called a \textit{bounded analytic semigroup}.
\end{definition}

To any analytic semigroup of some angle $\theta \in (0 , \pi / 2],$  one can  attach a unique operator $A : \dom(A) \subset X \to X$ defined by
\begin{align*}
 \dom(A) := \Big\{ x \in X : \lim_{t \to 0} \frac{1}{t} (T(t) x - x) \text{ exists} \Big\}
\end{align*}
and, for $x \in \dom(A),$ 
\begin{align*}
 A x := - \lim_{t \to 0} \frac{1}{t} (T(t) x - x).
\end{align*}
The operator  $-A$ is called the \textit{generator} of  $(T(z))_{z \in \Sec_{\theta} \cup \{ 0 \}}.$ 
\medbreak
Combining~\eqref{Item: Semigroup law} and~\eqref{Item: Analyticity} one readily sees that the range of $T(z)$ is contained in $\dom(A)$ for any $z \in \Sec_{\theta}$ and that the function $u : [0 , \infty) \to X$ given by $u (t) := T(t) x$ solves the \textit{abstract Cauchy problem}
\begin{align}
\label{Eq: Abstract IVP}
\left\{ \begin{aligned}
 u^{\prime} (t) + A u (t) &= 0 && t > 0, \\
 u(0) &= x.
\end{aligned} \right.
\end{align}

{}From the PDE perspective,  one can wonder if, whenever  $A : \dom(A) \subset X \to X$ is 
a given   linear operator,  $-A$ is the generator of an analytic semigroup. 
At this point, we need to recall the notion of a \emph{sectorial operator.}
\begin{definition}
A linear operator $B : \dom(B) \subset X \to X$ is called \textit{sectorial of angle $\omega$} for some $\omega \in [0 , \pi)$ if its spectrum satisfies $\sigma (B) \subset \overline{\Sec_{\omega}}$ and if for all $\omega < \omega^{\prime} < \pi$ there exists $C > 0$ such that
\begin{align*}
 \| \lambda (\lambda - B)^{-1} \|_{\Lop(X)} \leq C \qquad (\lambda \in \IC \setminus \overline{\Sec_{\omega^{\prime}}}).
\end{align*}
\end{definition}
The following characterization theorem for analytic semigroups is classical~\cite[Thm.~II.4.6]{Engel_Nagel}.

\begin{theorem} Let $A : \dom(A) \subset X \to X$  be a linear operator. 
Then $-A$ is the generator of an analytic semigroup if and only if $A$ is densely defined and  there exists $z \in \IC$ such that $z + A$ is sectorial of some angle $\omega \in [0 , \pi / 2)$. Moreover, $-A$ generates a bounded analytic semigroup if and only if additionally one can choose $z = 0$, i.e., $A$ itself is sectorial of angle $\omega \in [0 , \pi / 2)$.
\end{theorem}

\begin{remark}
\label{Rem: Analyticity}
The condition that $z + A$ is sectorial of angle $\omega \in [0 , \pi / 2)$ is equivalent to the fact that there exists $R > 0$ such that $\sigma (-A) \subset \overline{\Sec_{\omega}} \cup B(0 , R)$ and such that
\begin{align*}
 \| \lambda (\lambda +A)^{-1} \|_{\Lop(X)} \leq C \qquad (\lambda \in \IC \setminus [ \overline{\Sec_{\omega^{\prime}}} \cup B(0 , R)]).
\end{align*}
\end{remark}

\begin{remark}
\label{Rem: Exponential decay}
If $-A$ generates a bounded analytic semigroup and if $0 \in \rho(A)$, then the corresponding semigroup is exponentially decaying. Indeed, as $A$ is sectorial of angle $\omega \in [0 , \pi / 2)$ and as the resolvent set is open, one finds that
\begin{align*}
 \inf_{\lambda \in \sigma (A)} \Re(\lambda) > 0.
\end{align*}
Thus, there exists $\eps > 0$ and $\omega^{\prime} \in [0 , \pi / 2)$ such that $A-\eps$ is sectorial of angle $\omega^{\prime}$ which implies that the semigroup generated by $\eps- A$ is bounded. This in turn implies the exponential decay of the semigroup generated by $-A$.
\end{remark}

To solve nonlinear equations, it is helpful to consider~\eqref{Eq: Abstract IVP} for a homogeneous initial value but for an inhomogeneous right-hand side of the first equation, i.e.,
\begin{align}
\label{Eq: Inhomogeneous abstract Cauchy problem}
\left\{ \begin{aligned}
 u^{\prime} (t) + A u (t) &= f(t) && t \in (0 , T), \\
 u(0) &= 0,
\end{aligned} \right.
\end{align}
where $0 < T \leq \infty$ and $f \in \L^q (0 , T ; X)$, $1 \leq q \leq \infty$. As recalled in the introduction,  a densely defined operator $A : \dom(A) \subset X \to X$ is said to have \textit{maximal $\L^q$-regularity} if there exists a constant $C > 0$ such that for all $f \in \L^q (0 , T ; X),$  System~\eqref{Eq: Inhomogeneous abstract Cauchy problem} has a unique solution $u$ that satisfies  $u (t) \in \dom(A)$ for almost all $t \in (0 , T),$  is almost everywhere differentiable and such that
\begin{align*}
 \| u^{\prime} , A u \|_{\L^q (0 , T ; X)} \leq C \| f \|_{\L^q (0 , T ; X)}.
\end{align*}
It is classical, see, e.g., Dore~\cite[Cor.~4.4]{Dore}, that the maximal $\L^q$-regularity of $A$ implies that $- A$ generates an analytic semigroup. 
Characterizing when a given operator admits maximal $\L^q$-regularity is often a difficult issue, which involves questions on the geometry of Banach spaces and operator-valued multiplier theorems, see~\cite{Denk_Hieber_Pruess, Kunstmann_Weis}. However, if one is willing to \textit{change} the underlying Banach space into a real interpolation space between $X$ and $\dom(A)$, then the question of maximal $\L^q$-regularity simplifies tremendously. It is a classical result of Da Prato and Grisvard~\cite{DaPrato-Grisvard}, that is described below. \smallbreak
To state the result, we need to introduce the definition of \emph{a part of an operator onto another space.}
\begin{definition}
Let $X$ and $Y$ be Banach spaces and $C : \dom(C) \subset X \to X$ be a linear operator. The \textit{part} of $C$ in $Y$ is the operator given by
\begin{align*}
 \dom({\bf C}) := \{ y \in \dom(C) \cap Y : C y \in Y \}, \quad {\bf C} y := C y.
\end{align*}
\end{definition}
Let in the following $B$ denote the time derivative operator on $(0 , T)$, with $0 < T \leq \infty$, i.e.,
\begin{align*}
 B : \{u \in \W^{1 , q} ((0 , T) ; X) : u (0) = 0 \} \subset \L^q (0 , T ; X) \to \L^q (0 , T ; X), \quad u \mapsto u^{\prime}.
\end{align*}
It is well-known, see, e.g.,~\cite[Sec.~8.4-8.6]{Haase}, that $B$ is sectorial of angle $\pi / 2$. \par
Furthermore, let $A$ be a densely defined and sectorial operator of angle $\omega \in [0 , \pi / 2)$, i.e., $- A$ is the generator of a bounded analytic semigroup. We lift the operator $A$ to the time-dependent space by defining
\begin{align*}
 A^{\uparrow} : \dom(A^{\uparrow}):= \L^q (0 , T ; \dom(A)) \subset \L^q (0 , T ; X) \to \L^q (0 , T ; X), \quad [A^{\uparrow} u] (t) := A u(t).
\end{align*}
As the operator $A$ does not explicitly depend on time, the resolvents of $A^{\uparrow}$ and $B$ commute, i.e., it holds
\begin{align*}
 (\lambda - A^{\uparrow})^{-1} (\mu - B)^{-1} = (\mu - B)^{-1} (\lambda - A^{\uparrow})^{-1} \qquad (\lambda \in \rho(A) , \, \mu \in \rho(B)).
\end{align*}

In this situation, the theorem of Da Prato and Grisvard may be formulated as follows, see~\cite[Thm.~3.11, Lem.~3.5]{DaPrato-Grisvard}:

\begin{theorem}
\label{Thm: Da Prato-Grisvard}
Let $\theta \in (0 , 1)$ and $1 \leq q \leq \infty$. With the notation above, the part of the operator
\begin{align*}
 C := A^{\uparrow} + B \quad \textit{with domain} \quad \dom(C) := \dom(A^{\uparrow}) \cap \dom(B)
\end{align*}
in the real interpolation space 
\begin{align*}
 \big( \L^q(0 , T ; X) , \dom(A^{\uparrow}) \big)_{\theta , q} = \L^q (0 , T ; (X , \dom(A))_{\theta , q})
\end{align*}
is sectorial of angle $\pi / 2$. Furthermore, there exists $M > 0$ such that for all $\lambda > 0$ and $y \in \L^q (0 , T ; (X , \dom(A))_{\theta , q})$ it holds that $A^{\uparrow} (\lambda + A^{\uparrow} + B)^{-1} y \in \L^q (0 , T ; (X , \dom(A))_{\theta , q})$ and $B (\lambda + A^{\uparrow} + B)^{-1} y \in \L^q (0 , T ; (X , \dom(A))_{\theta , q})$ and that
$$\displaylines{\quad
 \| A^{\uparrow} (\lambda + A^{\uparrow} + B)^{-1} y \|_{\L^q(0 , T ; (X , \dom(A))_{\theta , q})} + \| B (\lambda + A^{\uparrow} + B)^{-1} y \|_{\L^q ( 0 , T ; (X , \dom(A))_{\theta , q})} \hfill\cr\hfill
  \leq M \| y \|_{\L^q(0 , T ; (X , \dom(A))_{\theta , q})}.\quad}$$
\end{theorem}

The application of this theorem to the situation of maximal regularity is as follows. By construction, the solution operator to~\eqref{Eq: Inhomogeneous abstract Cauchy problem} is given by
\begin{align*}
 (A^{\uparrow} + B)^{-1},
\end{align*}
so that the question of whether $A$ has maximal $\L^q$-regularity is about whether $A^{\uparrow} + B$ is invertible and whether
\begin{align*}
 A^{\uparrow} (A^{\uparrow} + B)^{-1} \quad \text{and} \quad B (A^{\uparrow} + B)^{-1}
\end{align*}
are bounded. We present how to derive these properties by means of the theorem of Da Prato and Grisvard, for operators $A$ that are additionally invertible.  Notice that the invertibility of $A$ implies the invertibility of $A^{\uparrow}$. By the argument in Remark~\ref{Rem: Exponential decay}, there exists $\eps > 0$ such that $A^{\uparrow} - \eps$ is sectorial of angle less than $\pi / 2$ as well. Applying Theorem~\ref{Thm: Da Prato-Grisvard} to $C := (A^{\uparrow} - \eps) + B$ shows that there exists $K > 0$ such that
\begin{align*}
 \| B (A^{\uparrow} + B)^{-1} \|_{\Lop(\L^q(0 , T ; (X , \dom(A))_{\theta , q}))} = \| B (\eps + C)^{-1} \|_{\Lop(\L^q(0 , T ; (X , \dom(A))_{\theta , q}))} \leq K
\end{align*}
and
\begin{align*}
 \| A^{\uparrow} (A^{\uparrow} &+ B)^{-1} \|_{\Lop(\L^q(0 , T ; (X , \dom(A))_{\theta , q}))} + \| (A^{\uparrow} + B)^{-1} \|_{\Lop(\L^q(0 , T ; (X , \dom(A))_{\theta , q}))} \\
 &\leq \| (A^{\uparrow} - \eps) (\eps + C)^{-1} \|_{\Lop(\L^q(0 , T ; (X , \dom(A))_{\theta , q}))} + (1 + \eps) \| (\eps + C)^{-1} \|_{\Lop(\L^q(0 , T ; (X , \dom(A))_{\theta , q}))} \\
 &\leq K.
\end{align*}
This shows that there exists a constant $K > 0$ such that whenever $f \in \L^q (0 , T ; (X , \dom(A))_{\theta , q})$, the equation~\eqref{Eq: Inhomogeneous abstract Cauchy problem} has a unique solution $u$ satisfying
\begin{align}
\label{Eq: Da Prato-Grisvard estimate}
 \| u , u^{\prime} , A u \|_{\L^q(0 , T ; (X , \dom(A))_{\theta , q})} \leq K \| f \|_{\L^q(0 , T ; (X , \dom(A))_{\theta , q})}.
\end{align}
 In later sections, we will in particular be interested in the case $q = 1$. 
 \medbreak
 We conclude this section, by shortly discussing how to extend this theory to include inhomogeneous initial values in~\eqref{Eq: Inhomogeneous abstract Cauchy problem} if $q = 1$. 
We have to investigate under which conditions on $x$ the function $t \mapsto A T(t) x$ lies in $\L^1 (0 , T ; (X , \dom(A))_{\theta , 1})$. Now, we use that the real interpolation space $(X , \dom(A))_{\theta , q}$ can be characterized by means of the semigroup $(T(t))_{t \geq 0}$. Indeed, e.g., by~\cite[Thm.~6.2.9]{Haase} it holds (in the special case $1 \leq q < \infty$)
\begin{align}
\label{Eq: Characterization real interpolation spaces}
 \big( X , \dom(A) \big)_{\theta , q} = \Big\{ x \in X : [x]_{\theta , q}^q := \int_0^{\infty} \| t^{1 - \theta} A T(t) x \|_X^q \, \frac{\d t}{t} < \infty \Big\} =: \dom_A (\theta , q)
\end{align}
and the norms
\begin{align*}
 \| x \|_{(X , \dom(A))_{\theta , q}} \quad \text{and} \quad \| x \|_X + [ x ]_{\theta , q} =: \| x \|_{\dom_A (\theta , q)}
\end{align*}
are equivalent. A similar result holds for $q = \infty$ with the obvious changes in the definition of $[x]_{\theta , q}$. In our case $q = 1$, we directly find by the exponential decay and the analyticity of the semigroup (i.e., we use that $\| \e^{\varepsilon s} s A T(s) \|_{\Lop(X)}$ is uniformly bounded with respect to $s > 0$ for some $\varepsilon>0$) that
\begin{align*}
 \int_0^T \| A T(s) x \|_X \, \d s &\leq \int_0^1 s^{\theta} \| s^{1 - \theta} A T(s) x \|_X \, \frac{\d s}{s}
 +\int_1^{\infty}  \| sA T(s) x \|_X \, \frac{\d s}{s} 
  \\
 &\leq \int_0^1 \| s^{1 - \theta} A T(s) x \|_X \, \frac{\d s}s + M \int_1^{\infty} \e^{- \eps s} \, \d s \; \| x \|_X \\
 &\leq M' \| x \|_{\dom_A (\theta , 1)}.
\end{align*}
Moreover, using the analyticity of the semigroup again, followed by occasional applications of Fubini's theorem and the linear substitution rule yields for some constant $M > 0$ that
\begin{align*}
 \int_0^T \int_0^{\infty} \| t^{1 - \theta} A T(t) A T(s) x \|_X \, \frac{\d t}{t} \, \d s &\leq M \int_0^{\infty} \int_0^{\infty} \frac{t^{1 - \theta}}{s + t} \| A T(\tfrac{1}{2}(s + t)) x \|_X \, \frac{\d t}{t} \, \d s \\
 &= M \int_0^{\infty} \int_t^{\infty} \frac{t^{1 - \theta}}{\tau} \| A T(\tfrac{1}{2} \tau) x \|_X \, \d \tau \, \frac{\d t}{t} \\
 &= M \int_0^{\infty} \int_0^{\tau} t^{- \theta} \, \d t \, \| A T(\tfrac{1}{2} \tau) x \|_X  \, \frac{\d \tau}{\tau} \\
 &= \frac{M}{1 - \theta} \int_0^{\infty} \| \tau^{1 - \theta} A T(\tfrac{1}{2} \tau) x \|_X \, \frac{\d \tau}{\tau} \\
 &= \frac{M 2^{1 - \theta}}{1 - \theta} \| x \|_{\dom_A (\theta , 1)}.
\end{align*}
Thus, for all $x \in (X , \dom(A))_{\theta , 1}$, we find that
\begin{align*}
 \| s \mapsto A T(s) x \|_{\L^1 (0 , T ; \dom_A (\theta , 1))} \leq M \| x \|_{\dom_A (\theta , 1)}.
\end{align*}

We formulate the results of this discussion as a corollary of the theorem of Da Prato and Grisvard.

\begin{corollary}
\label{Cor: Initial values}
Let $X$ be a Banach space and let $- A$ be the generator of a bounded analytic semigroup on $X$ with $0 \in \rho(A)$. Let $\theta \in (0 , 1)$ and $0 < T \leq \infty$. Then for all $f \in \L^1 (0 , T ; (X , \dom(A))_{\theta , 1})$ and for all $x \in (X , \dom(A))_{\theta , 1}$ the equation
\begin{align*}
 \left\{ \begin{aligned}
  u^{\prime} (t) + A u(t) &= f(t) && t \in (0 , T), \\
  u(0) &= x
 \end{aligned} \right.
\end{align*}
has a unique solution in the space
\begin{align*}
 \W^{1 , 1} ((0 , T) ; (X , \dom(A))_{\theta , 1}) \cap \L^1 (0 , T ; \dom({\bf A}))
\end{align*}
satisfying
\begin{align*}
  \| u , u^{\prime} , A u \|_{\L^1(0 , T ; (X , \dom(A))_{\theta , 1})} \leq K \big( \| x \|_{(X , \dom(A))_{\theta , 1}} + \| f \|_{\L^1(0 , T ; (X , \dom(A))_{\theta , 1})} \big).
\end{align*}
Here, ${\bf A}$ denotes the part of $A$ on $(X , \dom(A))_{\theta , 1}$.
\end{corollary}

\section{Study of the Lam\'e operator}\label{s:lame}

This section is dedicated to the study of the linearization of the velocity equation
of  System~\eqref{eq:CNS}, when neglecting the pressure. 
We shall first establish  various
regularity results for the \emph{Lam\'e operator} $L$ given by  \begin{equation}
\label{Eq: Informal Lame}
 L = - \mu \Delta - z \nabla \divergence,
\end{equation}
then look at the properties of the associated semigroup, with particular attention to the maximal $\L^q$-regularity
on Besov spaces  $\B^s_{p , q} (\Omega ; \IC^d)$  \emph{up to the limit value $q=1.$} This is done by employing Amann's technique of inter- and extrapolation spaces.
Throughout the section, 
 $\Omega \subset \IR^d$, $d \geq 1$, is  a smooth bounded domain. 
 The Lebesgue exponent $p$ is supposed to satisfy $1 < p < \infty$, the microlocal parameter $q$ satisfies $1 \leq q \leq \infty$, and we assume that the real number $s$ is such that
\begin{align}
\label{Eq: Space sees no traces}
 - 1 + \frac{1}{p} < s < \frac{1}{p}\cdotp
\end{align}
Recall that~\eqref{Eq: Space sees no traces} ensures that elements of 
\emph{$\B^s_{p , q} (\Omega ; \IC^d)$ have no trace at the boundary.} 
\medbreak

As a start, let us record  the standard $\L^2$-theory of the Lam\'e operator, following the exposition in~\cite{Mitrea_Monniaux}. Let $D u$ denote the Jacobian matrix of a vector field $u,$ and let $\nabla u$ denote its transpose. Define the $\curl$ of $u$ by
\begin{align*}
 \curl u := \frac{1}{\sqrt{2}} (\nabla u - D u).
\end{align*}
Let $\mu > 0$ and $z \in \IC$  and define the sesquilinear form
\begin{align}
\label{Eq: Sesquilinear form Lame}
 \fa : \left\{\begin{aligned}
\W^{1 , 2}_0 (\Omega ; \IC^d) \times \W^{1 , 2}_0 (\Omega ; \IC^d) &\longrightarrow \IC, \\
 (u , v)\qquad &\longmapsto \mu \int_{\Omega} \curl u \cdot \overline{\curl v} \; \d x + (\mu + z) \int_{\Omega} \divergence u \, \overline{\divergence v} \; \d x,
\end{aligned}\right.
\end{align}
where the matrix product is understood component-wise. As the complex parameter $z$ is not standard in usual considerations of the Lam\'e system, we give more details in the subsequent discussion. Under the supplementary condition that   $\mu + \Re(z) > 0,$ the sesquilinear form $\fa$ is bounded and coercive, cf.~\cite[Lem.~3.1]{Mitrea_Monniaux}.  Then, define \emph{the Lam\'e operator on $\L^2$} by
\begin{align*}
 \dom(L_2) &:= \bigl\{ u \in \W^{1 , 2}_0 (\Omega ; \IC^d) : \exists f \in \L^2 (\Omega ; \IC^d) \text{ s.t.\@ } \fa(u , v) = \langle f , v \rangle _{\L^2} \text{ for all } v \in \W^{1 , 2}_0 (\Omega ; \IC^d) \bigr\} \\
 L_2 u &:= f \qquad (u \in \dom(L_2)).
\end{align*}
With this definition, $L_2$ embodies~\eqref{Eq: Informal Lame} in the sense of distributions. Notice that 
$$\C_c^{\infty} (\Omega ; \IC^d) \subset \dom(L_2)\subset \W^{1,2}_0(\Omega;\IC^d),$$  hence
$L_2$ is densely defined. Moreover, $L_2$ is closed and, according to the Lax-Milgram 
theorem, invertible.

Following~\cite[Thm.~4.16 and Thm.~4.18]{McLean} and using a  covering argument, 
it is easy to obtain the following regularity result for $L_2$ 
(with the convention $\W^{0 , 2} (\Omega ; \IC^d) = \L^2 (\Omega ; \IC^d)$).
\begin{proposition}
\label{Prop: Higher regularity for reduced operator on L2}
Let $\mu > 0$ and $z \in \IC$ with $\mu + \Re(z) > 0$. Let $k \in \IN_0$ and $\Omega$ be a bounded domain with smooth boundary. Then, there exists a constant $C > 0$ such that for all $f \in \W^{k , 2} (\Omega ; \IC^d)$ and $u$ given by $u = L_2^{-1} f,$ it holds
\begin{align*}
 \| u \|_{\W^{k + 2 , 2} (\Omega ; \IC^d)} \leq C \| f \|_{\W^{k , 2} (\Omega ; \IC^d)}.
\end{align*}
\end{proposition}

Having some $\L^2$-mapping properties of the Lam\'e operator at our disposal, we focus now on the $\L^p$-theory. If $2 < p < \infty$, then we define the Lam\'e operator on $\L^p (\Omega ; \IC^d)$, denoted by $L_p$, to be the part of $L_2$ in $\L^p (\Omega ; \IC^d)$. Note that $L_p$ is a closed operator and that $\C_c^{\infty} (\Omega ; \IC^d)$ is included in $\dom(L_p)$. \par
For $1 < p < 2$, define $L_p$ to be the closure of $L_2$ in $\L^p (\Omega ; \IC^d)$ whenever $L_2$ is closable in this space. That $L_2$ is indeed closable in $\L^p (\Omega ; \IC^d)$ is deduced by the following argument: Since $L_2$ is closed and densely defined, its $\L^2$-adjoint $L_2^*$ is well-defined, densely defined, and closed. Clearly, this operator is the realization of~\eqref{Eq: Informal Lame} with $z$ replaced by its complex conjugate $\overline{z}$. Now, the fact that $L_2$ is closable in $\L^p (\Omega ; \IC^d)$ stems from the following lemma\footnote{We use the following notation and convention:
the antidual space of a Banach space $X$ (i.e., the space of all antilinear mappings $X \to \IC$) is denoted by $X^{\prime}$. The adjoint of a densely defined operator $A$ is denoted by $A^{\prime}$. In the particular situation where $X = \L^p (\Omega ; \IC^d)$ and $A : \dom(A) \subset \L^p (\Omega ; \IC^d) \to \L^p (\Omega ; \IC^d)$ is densely defined, the adjoint operator $A^{\prime}$ is an operator $A^{\prime} : \dom(A^{\prime}) \subset \L^p (\Omega ; \IC^d)^{\prime} \to \L^p (\Omega ; \IC^d)^{\prime}$. The corresponding adjoint operator on $\L^{p^{\prime}} (\Omega ; \IC^d)$ (where $p^{\prime}$ stands for  the H\"older conjugate exponent of $p$) is denoted by $A^*$. Thus, if $\Phi$ denotes the canonical isomorphism $\L^{p^{\prime}} (\Omega ; \IC^d) \to \L^{p^{\prime}} (\Omega ; \IC^d)^{\prime}$, then $A^*$ is given by
\begin{align}
\label{Eq: Lp adjoint}
 A^* := \Phi^{-1} A^{\prime} \Phi \quad \text{with domain} \quad \dom(A^*) := \{ u \in \L^{p^{\prime}} (\Omega ; \IC^d) : \Phi u \in \dom(A^{\prime}) \}.
\end{align}}
that can be proved by basic annihilator relations and is partly presented in~\cite[Lem.~2.8]{Tolksdorf}.
\begin{lemma}
\label{Lem: Duality}
Let $1 < p < 2$. Then $\dom(L_2)$ is dense in $\L^p (\Omega ; \IC^d)$. Moreover, $L_2$ is closable in $\L^p (\Omega ; \IC^d)$ if and only if 
the part $(L_2^*)_{p^{\prime}}$ of $L_2^*$ in $\L^{p^\prime} (\Omega ; \IC^d)$ is densely defined. In this case, it holds $L_p^* = (L_2^*)_{p^{\prime}}$ and $(L_2^*)_{p^{\prime}}^* = L_p$.
\end{lemma}

Having the $\L^p$-realization of $L_2$ at hand, we turn to the regularity theory of $L_p$ for $1 < p < \infty$. 
The counterpart of Proposition~\ref{Prop: Higher regularity for reduced operator on L2}
(that is proved in Appendix) reads:
\begin{proposition}
\label{Prop: Regularity for reduced operator on Lp}
Let $\mu > 0$ and $z \in \IC$ with $\mu + \Re(z) > 0$. Let $k \in \IN_0$ and $\Omega$ be a bounded domain with smooth boundary. For all $1 < p < \infty$ it holds $0 \in \rho(L_p)$ and $\dom(L_p^k)$ is continuously embedded into $\W^{2 k , p} (\Omega ; \IC^d)$. Moreover, in the case $2 \leq p < \infty$ there exists a constant $C > 0$ such that for all $f \in \W^{k , p} (\Omega ; \IC^d)$ and $u$ given by $u = L_p^{-1} f$ it holds
\begin{align}
\label{Eq: Elliptic regularity}
 \| u \|_{\W^{k + 2 , p} (\Omega ; \IC^d)} \leq C \| f \|_{\W^{k , p} (\Omega ; \IC^d)}.
\end{align}
In the case $1 < p < 2$ there exists a constant $C > 0$ such that for all $f \in \dom(L_p^k)$ it holds
\begin{align}
\label{Eq: Elliptic regularity 2}
 \| u \|_{\W^{2 k + 2 , p} (\Omega ; \IC^d)} \leq C \| f \|_{\W^{2 k , p} (\Omega ; \IC^d)}.
\end{align}
In particular, for any $1<p<\infty,$ we have
\begin{equation}\label{eq:domaineLp}
\dom(L_p)=\W^{2,p}(\Omega;\IC^d)\cap\W^{1,p}_0(\Omega;\IC^d).
\end{equation}
\end{proposition}

We aim at proving  that $-L_p$ generates a bounded analytic semigroup on a wide family 
of Besov spaces.  Our starting point is the following proposition, which is a consequence of~\cite[Thm.~1.3]{Mitrea_Monniaux} and~\cite[App.~A]{Danchin}.
\begin{proposition}
\label{Prop: Analyticity on Lp}
Let $\mu , \mu^{\prime} \in \IR$ with $\mu > 0$ and $\mu + \mu^{\prime} > 0$, $1 < p < \infty$, and $L_p$ be the Lam\'e operator with coefficients $\mu$ and $z = \mu^{\prime}$. Then, $- L_p$ generates a bounded analytic semigroup on $\L^p (\Omega ; \IC^d)$.
\end{proposition}
We want to  prove a similar result but at the scale of  a `negative' regularity space that may be regarded
as $\W^{-2,p}.$ 
To proceed, we need to introduce  the following canonical isomorphism
(where the dependency on $r$ is omitted for notational simplicity):
\begin{equation}\label{eq:Fi}
 \Phi : \L^{r^{\prime}} (\Omega ; \IC^d) \to \L^r (\Omega ; \IC^d)^{\prime}, \quad \Phi f := \bigg[ g \mapsto \int_{\Omega} f \cdot \overline{g} \; \d x \bigg]\cdotp
\end{equation}
Recall that $(L^*_2)_{p^{\prime}}$ is the Lam\'e operator with $z$ replaced by $\overline{z}$ on $\L^{p^{\prime}} (\Omega ; \IC^d)$. Since $\dom((L^*_2)_{p^{\prime}})$ is a closed subspace of $\W^{2 , p^{\prime}} (\Omega ; \IC^d)$, the domain $\dom((L^*_2)_{p^{\prime}})$ is a Banach space when endowed with the $\W^{2 , p^{\prime}}$-norm and $(L^*_2)_{p^{\prime}} \in \Isom(\dom((L^*_2)_{p^{\prime}}) , \L^{p^{\prime}} (\Omega ; \IC^d))$. Denote the dual operator from $\L^{p^{\prime}} (\Omega ; \IC^d)^{\prime}$ onto $\dom((L^*_2)_{p^{\prime}})^{\prime}$ by a $^\circ$, i.e.,
\begin{align*}
 \widetilde{\cL}_p := (L^*_2)_{p^{\prime}}^{\circ} \in \Isom(\L^{p^{\prime}} (\Omega ; \IC^d)^{\prime} , \dom((L^*_2)_{p^{\prime}})^{\prime})
\end{align*}
and define the \emph{extrapolation $\cL_p$ of $L_p$} on the ground space
$X_p^{-1} :=\dom((L_2^*)_{p^{\prime}})^{\prime}$ to be
\begin{equation}\label{eq:cLp}
 \cL_p : \dom(\cL_p) \subset X_p^{-1} \to X_p^{-1}, \quad \cL_p u := \widetilde{\cL}_p u \with \dom(\cL_p) := \L^{p^{\prime}} (\Omega ; \IC^d)^{\prime}.
\end{equation}
Observe that $\widetilde{\cL}_p$ is defined as the adjoint of the \textit{bounded} operator $(L_2^*)_{p^{\prime}} : \dom((L^*_2)_{p^{\prime}}) \to \L^{p^{\prime}} (\Omega ; \IC^d)$. This should be distinguished from the adjoint operator $(L^*_2)_{p^{\prime}}^{\prime} : \dom ((L^*_2)_{p^{\prime}}^{\prime}) \subset \L^{p^{\prime}} (\Omega ; \IC^d)^{\prime} \to \L^{p^{\prime}} (\Omega ; \IC^d)^{\prime}$, where $(L^*_2)_{p^{\prime}}$ is regarded as a closed and densely defined operator on $\L^{p^{\prime}} (\Omega ; \IC^d)$. The links between all these definitions are clarified in 
Appendix (see Lemma~\ref{Lem: Extrapolation lemma}).

\medbreak
The previous lemma allows us to define an extrapolation $\cL_p$ of the operator $L_p$ to the larger ground space $X_p^{-1} :=\dom((L_2^*)_{p^{\prime}})^{\prime}$, which can be regarded as a $\W^{-2 , p}$-space. In particular, Lemma~\ref{Lem: Extrapolation lemma}~\eqref{Item: Similarity of operators} allows us to write\footnote{We endow the product of two operators $A$ and $B$ with its maximal domain of definition, i.e., $\dom(A B) := \{ u \in \dom(B) : B u \in \dom(A) \}$.}
  $$\cL_p = T L_p T^{-1},$$ where $T := \widetilde{\cL}_p \Phi$  is an isomorphism from $\L^p (\Omega ; \IC^d)$ onto $X_p^{-1}$. This will enable us  to transport all kinds of functional analytic properties from $L_p$ to $\cL_p$. Finally, Lemma~\ref{Lem: Extrapolation lemma}~\eqref{Item: Recovering original operator} allows us to recover $L_p$ (modulo the canonical isomorphism $\Phi$) from $\cL_p$ as its part on $\L^{p^{\prime}} (\Omega ; \IC^d)^{\prime}$, so that $\cL_p$ can indeed be regarded as an extrapolation of $L_p$.
 This eventually leads to  the following proposition.
\begin{proposition}
Let $\mu , \mu^{\prime} \in \IR$ with $\mu > 0$ and $\mu + \mu^{\prime} > 0$, $1 < p < \infty$, and $\cL_p$ be the Lam\'e operator with coefficients $\mu$ and $z = \mu^{\prime}$ on $X^{-1}_p$. Then, $- \cL_p$ generates a bounded analytic semigroup on $X^{-1}_p$.
\end{proposition}

Having a bounded analytic semigroup on various function spaces at our disposal, we want to deduce the maximal $\L^q$-regularity of the Lam\'e operator on suitable intermediate spaces. For this purpose, we briefly introduce the setting of Da Prato and Grisvard established in~\cite{DaPrato-Grisvard}.

For $1 < p < \infty,$ define the spaces
 \begin{align*}
 X^k_p := \Phi \dom(L_p^k) \qquad (k \in \IN_0).
\end{align*}
 Endow $X^k_p$ with the norm
\begin{align*}
 \| \fu \|_{X^k_p} := \| L_p^k \Phi^{-1} \fu \|_{\L^p (\Omega ; \IC^d)} \qquad (\fu \in X_p^k).
\end{align*}
Observe that, by construction, all spaces $X_p^k$ (including $X_p^{-1}$) are complete.
\medbreak
 For $- 1 < s < 1$, $0 < t < 2$, and $1 \leq q \leq \infty$ define the following intermediate spaces via real interpolation:
\begin{align*}
 X^s_{p , q} := \big( X_p^{-1} , X_p^1 \big)_{(s + 1) / 2 , q} \qquad \text{and} \qquad Y^t_{p , q} := \big( X_p^0 , X_p^2 \big)_{t / 2 , q}.
\end{align*}
Note that for all of the parameters above, the following continuous inclusions hold
\begin{align}
\label{Eq: Embedding Y^t}
 X^s_{p , q} \hookrightarrow X^{-1}_p \qquad \text{and} \qquad Y^t_{p , q} \hookrightarrow X^0_p = \dom(\cL_p).
\end{align}
For some combinations of the parameters, the spaces $X^s_{p , q}$ and $Y^t_{p , q}$ are calculated as follows. To formulate the proposition, introduce, for $1 < p < \infty$, $1 \leq q \leq \infty$, and $s \in \IR,$
the space \begin{align*}
 \B^s_{p , q , D} (\Omega ; \IC^d) := \begin{cases} \{ f \in \B^s_{p , q} (\Omega ; \IC^d) : f|_{\partial \Omega} = 0 \}, &\text{if } s > 1 / p \\
 \B^s_{p , q} (\Omega ; \IC^d), &\text{if } s < 1 / p. \end{cases}
\end{align*}
Here, elements in the Besov space $\B^s_{p , q} (\Omega ; \IC^d)$ are defined to be restrictions to $\Omega$ of elements in $\B^s_{p , q} (\IR^d ; \IC^d)$ and the norm of $\B^s_{p , q} (\Omega ; \IC^d)$ is given by the corresponding quotient norm. Furthermore, if $\Omega$ is smooth enough, e.g., Lipschitz regular, then the following interpolation identity holds (see more details in~\cite[Thm.~2.13]{Triebel_Lipschitz}):
\begin{align*}
 \big( \B^{s_0}_{p , q_0} (\Omega ; \IC^d) , \B^{s_1}_{p , q_1} (\Omega ; \IC^d) \big)_{\theta , q} = \B^s_{p , q} (\Omega ; \IC^d),
\end{align*}
where
\begin{align*}
 \theta \in (0 , 1), \quad s_0 \neq s_1 \in \IR, \quad s = (1 - \theta) s_0 + \theta s_1, \quad p \in (1 , \infty), \quad \text{and} \quad q_0 , q_1 , q \in [1 , \infty].
\end{align*}

\begin{proposition}
\label{Prop: Identification ground spaces}
Let $1 < p < \infty$ and $1 \leq q \leq \infty$. Then, for $- 1 / p^{\prime} < 2 s < 2$ with $2 s \neq 1 / p$ it holds up to the identification by the isomorphism $\Phi$ that
\begin{align*}
 X^s_{p , q} = \B^{2 s}_{p , q , D} (\Omega ; \IC^d).
\end{align*}
Furthermore, for $0 < s < 1$ with $2 s \neq 1 / p$ it holds
\begin{align*}
 Y^s_{p , q} = \B^{2 s}_{p , q , D} (\Omega ; \IC^d).
\end{align*}
In the case $2 s = 1 / p$, it holds  that
\begin{align*}
 X^s_{p , q} \hookrightarrow \B^{2 s}_{p , q} (\Omega ; \IC^d) \quad \text{and} \quad Y^s_{p , q} \hookrightarrow \B^{2 s}_{p , q} (\Omega ; \IC^d).
\end{align*}
\end{proposition}

\begin{proof}
First, we consider the spaces $Y^s_{p , q}$. Notice that by~\cite[Prop.~6.6.7]{Haase} and the sectoriality of $L_p$ on $\L^p (\Omega ; \IC^d)$ it holds for $0 < s < 1$
\begin{align*}
 Y^s_{p , q} = \big( X_p^0 , X_p^2 \big)_{s / 2 , q} = \big( X_p^0 , X_p^1 \big)_{s , q}.
\end{align*}
Since, by definition of the spaces, $\Phi$ is an isomorphism
\begin{align*}
 \Phi : \L^p (\Omega ; \IC^d) \to X_p^0 \qquad \text{and} \qquad \Phi : \dom(L_p) \to X_p^1,
\end{align*}
it holds by virtue of~\cite[Thm.~5.2]{Amann}
whenever $2 s \neq 1 / p$ with equivalent norms that
\begin{align*}
 Y^s_{p,q}=\big( X_p^0 , X_p^1 \big)_{s , q} = \Phi \big(\L^p (\Omega ; \IC^d) , \dom(L_p) \big)_{s , q} = \Phi \B^{2 s}_{p , q , D} (\Omega ; \IC^d).
\end{align*}
If $2 s = 1 / p$, then $\dom(L_p) \subset \W^{2 , p} (\Omega ; \IC^d)$ implies that
\begin{align*}
 Y^{\frac{1}{2 p}}_{p,q} = \big( X_p^0 , X_p^1 \big)_{\frac{1}{2 p} , q} \subset \Phi \big(\L^p (\Omega ; \IC^d) , \W^{2 , p} (\Omega ; \IC^d) \big)_{\frac{1}{2 p}, q} = \Phi \B^{\frac{1}{p}}_{p , q} (\Omega ; \IC^d).
\end{align*}
\indent We turn to study the spaces $X^s_{p , q}$. As we already calculated $( X_p^0 , X_p^1)_{\theta , q}$ for $\theta \in (0 , 1)$, we concentrate first on $( X_p^{-1} , X_p^0)_{\theta , q}$ and the case $1 < q < \infty$. By the definitions of the spaces and the duality theorem~\cite[Sec.~1.11.2]{Triebel}, we find
\begin{align*}
 \big( X_p^{-1} , X_p^0 \big)_{\theta , q} = \big( \L^{p^{\prime}} (\Omega ; \IC^d) , \dom((L_2^*)_{p^{\prime}}) \big)_{1 - \theta , q^{\prime}}^{\prime} = \B^{2 (1 - \theta)}_{p^{\prime} , q^{\prime} , D} (\Omega ; \IC^d)^{\prime} = \B^{- 2 (1 - \theta)}_{p , q , D} (\Omega ; \IC^d).
\end{align*}
Notice that the following interpolation identities hold true, see~\cite[Thm.~5.2]{Amann},
\begin{align*}
 \big( X^0_p , X^1_p \big)_{\theta , q} = \Phi \B^{2 \theta}_{p , q} (\Omega ; \IC^d) \qquad (2 \theta < 1 / p)
\end{align*}
and
\begin{align*}
 \big( \L^{p^{\prime}} (\Omega ; \IC^d) , \dom((L_2^*)_{p^{\prime}}) \big)_{1 - \theta , q^{\prime}} = \B^{2 (1 - \theta)}_{p^{\prime} , q^{\prime}} (\Omega ; \IC^d) \qquad (2 (1 - \theta) < 1 / p^{\prime}).
\end{align*}
In particular,~\cite[Sec.~4.8.2]{Triebel} implies that
\begin{align*}
 \B^{2 (1 - \theta)}_{p^{\prime} , q^{\prime}} (\Omega ; \IC^d)^{\prime} = \B^{- 2 (1 - \theta)}_{p , q} (\Omega ; \IC^d) \qquad (2 (1 - \theta) < 1 / p^{\prime}).
\end{align*}
Since $\{ \B^s_{p , q} (\Omega ; \IC^d) \}_{- 1 / p^{\prime} < s < 1 / p}$ forms an interpolation family with respect to the real interpolation method~\cite[Sec.~4.3.1]{Triebel}, we find by~\cite{Wolff} (see also~\cite{Janson_Nilsson_Peetre}) and~\cite[Thm.~5.2]{Amann} modulo an identification with the canonical isomorphism $\Phi$ that
\begin{align*}
 X^s_{p , q} = \B^{2 s}_{p , q , D} (\Omega ; \IC^d) \qquad (- 1 / p^{\prime} < 2 s < 2 \text{ with } 2 s \neq 1 / p).
\end{align*}
The condition $q = 1$ or $q = \infty$ can now be added by the reiteration theorem.
\end{proof}

Having the scale $X^s_{p , q}$ of intermediate spaces at hand, we realize the Lam\'e operator 
${\bf L}_{p , q , s}$ on $X^s_{p , q}$ as the part of $\cL_p$ on this space, namely
\begin{align*}
 \dom({\bf L}_{p , q , s}) := \{ \fu \in \dom(\cL_p) \cap X^s_{p , q} : \cL_p \fu \in X^s_{p , q} \}.
\end{align*}
In Lemma~\ref{Lem: Domain of part}, it is shown that, for all $1 < p < \infty$, $1 \leq q \leq \infty$, and $- 1 < s < 1$ it holds with equivalent norms
\begin{equation}\label{eq:Ys1}
 \dom({\bf L}_{p , q , s}) = Y^{s + 1}_{p , q}.
\end{equation}

In general, if an operator generates a bounded analytic semigroup, its part onto a subspace need not generate a semigroup. However, as we already know that the domain of ${\bf L}_{p , q , s}$ 
is $Y^{s+1}_{p,q},$  this delivers right mapping properties of the resolvent of ${\bf L}_{p , q , s}$.

\begin{proposition}
\label{Prop: Generation proposition}
For all $1 < p < \infty$, $1 \leq q \leq \infty$, and $-1 < s < 1$ the operator $- {\bf L}_{p , q , s}$ with coefficients $\mu$ and $z = \mu^{\prime}$ generates a bounded analytic semigroup on $X^s_{p , q}$ with $0 \in \rho(L_p)$.
\end{proposition}
\begin{proof} According to Lemma~\ref{Lem: Extrapolation lemma},  $T := \widetilde{\cL}_p \Phi$ is an isomorphism between $\L^p (\Omega ; \IC^d)$ and $X^{-1}_p,$ and   $\cL_p = T L_p T^{-1}.$  Hence  $\rho(\cL_p) = \rho(L_p)$. 
 Furthermore, because $- L_p$ generates a bounded analytic semigroup, cf.~Proposition~\ref{Prop: Analyticity on Lp}, there exists some  $\theta \in (\pi/2 , \pi)$ and $C>0$ such that 
 $$ \Sec_{\theta} \subset \rho(- \L_p)\andf
  \| \lambda \Phi (\lambda + L_p)^{-1} \Phi^{-1} \|_{\Lop(X^0_p)} \leq C\quad\hbox{for all }\ 
  \lambda\in \Sec_\theta.$$
Notice that Lemma~\ref{Lem: Extrapolation lemma}~\eqref{Item: Recovering original operator} implies that $(\lambda + \cL_p)^{-1}|_{X^0_p} = \Phi (\lambda + L_p)^{-1} \Phi^{-1}$. 
  Thus, since $T : X_p^0 \to X^{-1}_p$ is an isomorphism, it holds
$$\begin{aligned}
 \| \lambda (\lambda + \cL_p)^{-1} \|_{\Lop(X^{-1}_p)}
 &= \| \lambda T\Phi^{-1}(\Phi(\lambda + L_p)^{-1}\Phi^{-1})\Phi T^{-1}\|_{\Lop(X^{-1}_p)}\\
  &\leq C \| \lambda \Phi (\lambda + L_p)^{-1} \Phi^{-1 }\|_{\Lop(X^0_p)} \leq C.
\end{aligned}$$
Then, by real interpolation we derive that for all $1 < p < \infty$, $1 \leq q \leq \infty$, and $-1 < s < 0$ there exists $C > 0$ such that for all $\lambda \in \Sec_{\theta}$ it holds
\begin{align}
\label{Eq: Interpolated resolvent estimate}
 \| \lambda (\lambda + \cL_p)^{-1}|_{X^s_{p , q}} \|_{\Lop(X^s_{p , q})} \leq C.
\end{align}
Finally, we prove that $\rho(\cL_p) \subset \rho({\bf L}_{p , q , s})$ and that $(\lambda + \cL_p)^{-1}|_{X^s_{p , q}} = (\lambda + {\bf L}_{p , q , s})^{-1}$ holds for $\lambda \in \rho(- \cL_p)$. \par
Let $\lambda \in \rho(- \cL_p)$. Clearly $\lambda + {\bf L}_{p , q , s}$ inherits the injectivity of $\lambda + \cL_p$. For the surjectivity, let $\ff \in X^s_{p , q}$. Since $\lambda \in \rho(- \cL_p)$, there exists $\fu \in \dom(\cL_p) = X^0_p$ such that $(\lambda + \cL_p) \fu = \ff$. Since $X^0_p \hookrightarrow X^s_{p , q}$, the definition of the part of an operator now implies that $\fu \in \dom({\bf L}_{p , q , s})$ and that $(\lambda + {\bf L}_{p , q , s}) \fu = \ff$. Consequently, this together with~\eqref{Eq: Interpolated resolvent estimate} implies that $- {\bf L}_{p , q , s}$ generates a bounded analytic semigroup on $X^s_{p , q}$. \par
In the case $0 < s < 1$ this follows immediately by the characterization in~\eqref{Eq: Characterization real interpolation spaces} and the fact that $- L_p$ generates a bounded analytic semigroup on $\L^p (\Omega ; \IC^d)$, see Proposition~\ref{Prop: Analyticity on Lp}. \par
The final case $s = 0$ follows by interpolation.
\end{proof}

Putting together all the previous results, it is now
possible to state maximal $\L^q$-regularity for the Lam\'e 
operator in Besov spaces, \emph{including the case $q=1$}.  
\begin{theorem} 
\label{Thm: Maximal regularity of Lame on Besov}
Let $\mu , \mu^{\prime} \in \IR$ with $\mu > 0$ and $\mu + \mu^{\prime} > 0$, $1 < p < \infty$, $1 \leq q \leq \infty$, $- 1 < s < 1$, and ${\bf L}_{p , q , s}$ be the Lam\'e operator with coefficients $\mu$ and $z = \mu^{\prime}$ on $X^s_{p , q}$. Then, ${\bf L}_{p , q , s}$ has maximal $\L^q$-regularity on the time interval $\IR_+$. In particular, if ${\bf L}_{p , q , s}^{\uparrow}$ denotes the
 lifted operator  to $\L^q (\IR_+ ; X^s_{p , q})$ (as in Section~\ref{Some background from semigroup theory}), then there exists a constant $C > 0$ such that the sum operator $\frac{\d}{\d t} + {\bf L}_{p , q , s}^{\uparrow}$ satisfies for all $K > 0$ and for all $f \in \L^q (\IR_+ ; X^s_{p , q})$
\begin{align}
\label{Eq: Sectoriality of the sum operator}
 \| \nabla^2 (\tfrac{\d}{\d t} + K + {\bf L}_{p , q , s}^{\uparrow})^{-1} f \|_{\L^q (\IR_+ ; \B^{2 s}_{p , q} (\Omega ; \IC^{d^3}))} \leq C \| f \|_{\L^q (\IR_+ ; X^s_{p , q})}.
\end{align}
\end{theorem}
\begin{proof} 
Fix $1 < p < \infty$ and $1 \leq q \leq \infty$. By virtue of Proposition~\ref{Prop: Generation proposition} we know for all $-1 < s_0 < 1$ that
\begin{align*}
 - {\bf L}_{p , q , s_0} \quad \text{generates a bounded analytic semigroup on} \quad X^{s_0}_{p , q} \quad \text{with} \quad 0 \in \rho({\bf L}_{p , q , s_0}).
\end{align*} 
Now, for $s \in (s_0 , \min\{ s_0 + 1 , 1 \})$ the discussion below Theorem~\ref{Thm: Da Prato-Grisvard} that leads to~\eqref{Eq: Da Prato-Grisvard estimate}, reveals that the part of ${\bf L}_{p , q , s_0}$ in $X^s_{p , q}$ has maximal $\L^q$-regularity on the time interval $\IR_+$. Since the part of ${\bf L}_{p , q , s_0}$ in $X^s_{p , q}$ is the operator ${\bf L}_{p , q , s}$ by Lemma~\ref{Lem: Domain of part}, this readily proves the first part of the theorem. \par
The estimate~\eqref{Eq: Sectoriality of the sum operator} follows by the boundedness of $\nabla^2 {\bf L}_{p , q , s}^{-1}$ from $X^s_{p , q}$ into $\B^{2 s}_{p , q} (\Omega ; \IC^3)$ which is established by combining Lemma~\ref{Lem: Domain of part} with Proposition~\ref{Prop: Identification ground spaces}. The estimate is then concluded by an application Theorem~\ref{Thm: Da Prato-Grisvard}.
\end{proof}

\begin{corollary}\label{c:lame} Let $0 < T \leq \infty.$ Let $1<p<\infty$ and $-1+1/p<s<1/p.$ For
any $u_0$ in $\B^s_{p,1}(\Omega;\IR^d)$ and 
$f\in \L^1(0,T;\B^s_{p,1}(\Omega;\IR^d)),$ system
\begin{align}\tag{(L)}\left\{\begin{aligned} \partial_tu-\mu\Delta u-\mu'\nabla\divergence u&=f && \hbox{in }\ (0,T)\times\Omega,\\
u|_{\partial\Omega}&=0 && \hbox{on }\ (0,T)\times\partial\Omega,\\
u|_{t=0}&=u_0 && \hbox{in }\ (0,T)\times\Omega,\end{aligned}\right.
\end{align}
admits a unique solution $u\in\cC_b([0,T];\B^s_{p,1}(\Omega;\IR^d))$
with 
\begin{align*}
 u \in \W^{1 , 1} (0 , T ; \B^s_{p , 1} (\Omega ; \IR^d)) \cap \L^1 (0 , T ; \B^{s + 2}_{p , 1} (\Omega ; \IR^d))
\end{align*}
and there exists a constant 
$C > 0$ depending only on $p,$ $s,$ $\mu'/\mu$, and $\Omega$ such that 
\begin{equation}\label{eq:inegu} \sup_{t\in[0,T]} \|u(t)\|_{\B^s_{p,1}}+\int_0^T \bigl(\|\partial_tu\|_{\B^s_{p,1}}
+\mu\|u\|_{\B^{s+2}_{p,1}}\bigr) \d t\leq 
C\biggl( \|u_0\|_{\B^s_{p,1}}+\int_0^T\|f\|_{\B^s_{p,1}}\,\d t\biggr)\cdotp\end{equation}
Furthermore, $C$ may be chosen uniformly  with respect to $\mu'/\mu$ whenever 
$\mu_*\leq \mu'/\mu\leq \mu^*$ for some constants  $\mu_*$ and $\mu^*$
such that $-1<\mu_*<\mu^*.$ 
\end{corollary}
\begin{proof} 
Performing the time rescaling 
$$u(t,x)=\wt u(\mu t,x)\andf f(t,x)=\mu\wt f(\mu t,x)$$
reduces the proof to the case $\mu=1.$
So we assume $\mu=1$ in what follows.
\medbreak
Now, if  $u_0=0,$ then the result is a mere reformulation of Theorem~\ref{Thm: Maximal regularity of Lame on Besov} with $q = 1.$ 
Indeed, from it, we get the maximal $\L^1$-regularity for ${\bf L}_{p , 1 , s},$  then using 
\eqref{eq:Ys1} and Proposition~\ref{Prop: Identification ground spaces} gives 
the desired bound for $\|u\|_{\B^{s+2}_{p,1}}.$  
The initial value $u_0$ can be added by virtue of Corollary~\ref{Cor: Initial values}, and the bound on $\|u(t)\|_{\B^s_{p,1}}$ follows from  the bound on $\partial_t u$ and the fundamental theorem of calculus.

Let us finally prove   that if  $\mu=1$ (with no loss of generality) and 
 $-1<\mu_*\leq\mu' \leq \mu^*,$  then the constant $C$ in~\eqref{eq:inegu} may be chosen 
independently of $\mu'$. Argue by contradiction, assuming that there exists a sequence $(\mu'_n)_{n\in\IN}$ in $[\mu_*,\mu^*]$
and a sequence $(u_{0,n},f_n)_{n\in\IN}$  such that  
$$\|u_{0,n}\|_{\B^s_{p,1}} +\int_0^{\infty} \|f_n\|_{\B^s_{p,1}}\d t=1$$
and the solution $u_n$ of $(L)$ with coefficients 
$\mu=1$ and $\mu'=\mu'_n,$  and data $(u_{0,n},f_n)$ satisfies
\begin{equation}\label{eq:inegun}\int_0^{\infty} \bigl(\|\partial_tu_n\|_{\B^s_{p,1}}
+ \|u_n\|_{\B^{s+2}_{p,1}}\bigr) \d t\geq n.\end{equation}
Up to subsequence, we have  $\mu'_n\to \bar\mu'\in[\mu_*,\mu^*].$  We observe that 
$$\partial_t u_n-\Delta u_n -\bar\mu'\nabla\divergence u_n=f_n+(\mu'_n-\bar\mu')\nabla\divergence u_n.$$
Hence applying Inequality~\eqref{eq:inegu} with coefficients $1$ and $\bar\mu',$ we get some constant
$C$ such that 
$$\int_0^{\infty} \bigl(\|\partial_tu_n\|_{\B^s_{p,1}}
\!+\!\|u_n\|_{\B^{s+2}_{p,1}}\bigr) \d t\leq C\biggl(\|u_{0,n}\|_{\B^s_{p,1}} +\int_0^{\infty}
\bigl(\|f_n\|_{\B^s_{p,1}} + |\mu'_n-\bar\mu'| \|\nabla\divergence u_n\|_{\B^s_{p,1}}\bigr)\d t\biggr)\cdotp$$
Given the definition of the data, we deduce (changing $C$ if need be) that 
$$\int_0^{\infty} \bigl(\|\partial_tu_n\|_{\B^s_{p,1}}
+\|u_n\|_{\B^{s+2}_{p,1}}\bigr) \d t\leq C\biggl(1 +  |\mu'_n-\bar\mu'| \int_0^{\infty} \|u_n\|_{\B^{s+2}_{p,1}} \, \d t \biggr)\cdotp$$
For $n$ large enough, the resulting inequality stands in contradiction with~\eqref{eq:inegun}.
  \end{proof}


\section{The linearized compressible Navier--Stokes system} 

In this section, we are concerned with  the full linearized compressible
Navier--Stokes system, in the case where the pressure function $P$ satisfies $P'(1)>0.$
We strive for a maximal $\L^q$-regularity result \emph{up to $q=1$}
on the whole time interval $\IR_+.$  The difficulty compared to the previous section is that  
we have to take into consideration the coupling between the density
equation which is of hyperbolic type and 
the velocity equation which is of parabolic type. 

As a first, let us observe that the  following  change of time scale  and velocity:
\begin{equation}\label{eq:rescaling}
(\rho,u)(t,x)\leadsto (\wt\rho, c \wt u)(c t,x)\with
c:=\sqrt{P'(1)}\end{equation}
reduces the study to the case  $P'(1)=1,$
so that  the linearization of  
 the   compressible Navier--Stokes system about $(\rho,u)=(1,0)$ 
 coincides with~\eqref{eq:LCNS}. 
  \medbreak
   Throughout this section,  we assume that $1 < p < \infty$ and that $- 1 / p^{\prime} < s <1/p$.
If $2 \leq p < \infty$, then we let $1 \leq q < \infty$ and if $1 < p < 2$, then we assume additionally that\footnote{Hence
we must have $p>2(d-1)/(d+2)$ owing to $s<1/p.$}
\begin{align*}
s>\frac dp-\frac d2-1\qquad\hbox{or}\qquad s\geq \frac dp-\frac d2-1\andf 1 \leq q \leq 2.
\end{align*}
Notice that these assumptions guarantee that functions in the space $\B^s_{p , q} (\Omega ; \IC^d)$ admit a well-defined trace and, owing to the boundedness of $\Omega$, that
\begin{align}
\label{Eq: Retract to L2-based spaces}
 \B^s_{p , q} (\Omega ; \IC^d) \hookrightarrow \W^{-1 , 2} (\Omega ; \IC^d) \qquad \text{and} \qquad \B^{s + 1}_{p , q} (\Omega) \hookrightarrow \L^2 (\Omega).
\end{align}
To define the second-order operator involved in~\eqref{eq:LCNS} in the context of the spaces $ \B^s_{p , q} (\Omega ; \IC^d),$ 
we set
\begin{align*}
 \cX^s_{p , q} &:= [\B^{s + 1}_{p , q} (\Omega) \cap \L^p_0 (\Omega)] \times \B^s_{p , q} (\Omega ; \IC^d) \\
 \dom(A_{p , q , s}) &:= [\B^{s + 1}_{p , q} (\Omega) \cap \L^p_0 (\Omega)] \times \B^{s + 2}_{p , q , D} (\Omega ; \IC^d),
\end{align*}
where $\L^p_0 (\Omega)$ denotes the space of $\L^p$-functions which are average free. \medbreak
Recall that ${\bf L}_{p , q , s}$ denotes the Lam\'e operator on $\B^s_{p , q} (\Omega ; \IC^d).$ 
Then, we put \begin{align}\label{Eq: Operator}
 A_{p , q , s} : \dom(A_{p , q , s}) \subset \cX^s_{p , q} \to \cX^s_{p , q}, \quad \begin{pmatrix} a \\ u \end{pmatrix} \mapsto \begin{pmatrix} \divergence u \\ {\bf L}_{p , q , s} u + \nabla a \end{pmatrix}\cdotp
\end{align}
The rest of the section is devoted to proving the following result which implies Theorem~\ref{thm:linear}.
\begin{theorem}
\label{Thm: Maximal regularity compressible}
Let $p$, $q$, and $s$ be chosen as above. Then $- A_{p , q , s}$ generates an exponentially stable analytic semigroup on $\cX^s_{p , q},$ and $A_{p , q , s}$ has maximal $\L^q$-regularity on the time interval $\IR_+$.
\end{theorem}
\begin{proof} The main steps are as follows. First, we show that
for each  $0 < T < \infty,$ the operator $A_{p , q , s}$ 
has maximal $\L^q$-regularity on the interval $(0 , T)$
(which, in light of~\cite[Thm.~4.3]{Dore},  implies that 
 operator $- A_{p , q , s}$ generates an analytic semigroup on $\cX^s_{p , q}$).
Next, we prove that $0$ is in the resolvent set of  $-A_{p , q , s}.$
In the third step -- the core of the proof -- we establish that 
the whole right complex half-plane is in $\rho(-A_{p , q , s}).$ 
By standard arguments, putting all those informations together
allows to conclude the proof (last step). 

\subsubsection*{First step: local-in-time maximal regularity}

We want to show that, for each  $0 < T < \infty,$ the operator $A_{p , q , s}$ 
has maximal $\L^q$-regularity on the interval $(0 , T)$.
 To proceed, we introduce, for some $K>0$ that will be chosen later on,  the auxiliary problem 
\begin{equation}\label{eq:auxiliary}
 \Bigl(\tfrac{\d}{\d t}+K\Bigr) \begin{pmatrix} \wt a \\ \wt u \end{pmatrix} + A_{p , q , s} \begin{pmatrix} \wt a \\ \wt u \end{pmatrix} = \begin{pmatrix} \wt f \\ \wt g \end{pmatrix}
\end{equation} 
for $(\wt f , \wt g) \in \L^q (\IR_+ ; \cX^s_{p , q}),$ supplemented with null initial data.
\medbreak
Clearly, $(\wt a,\wt u)$ satisfies~\eqref{eq:auxiliary} if and only if 
$(a,u)(t):=\e^{Kt}(\wt a,\wt u)(t)$ is a solution of 
\begin{equation}\label{eq:lcns}
 \tfrac{\d}{\d t} \begin{pmatrix} a \\ u \end{pmatrix} + A_{p , q , s} \begin{pmatrix} a \\ u \end{pmatrix} = \begin{pmatrix} f \\ g \end{pmatrix}
\end{equation} 
with null initial data and $(f,g)(t):=\e^{Kt}(\wt f,\wt g)(t).$ 
\medbreak
The operator  
$\frac{\d}{\d t}+K$ with domain $\W^{1 , q}_0 (\IR_+ ; \B^s_{p , q} (\Omega ; \IC^d))$ is invertible on  $\L^q (\IR_+; \B^s_{p , q} (\Omega ; \IC^d)),$ with inverse
given by
$$\bigl(\tfrac{\d}{\d t}+K\bigr)^{-1}\wt f: t\mapsto \int_0^t \e^{-K(t-\tau)} \wt f(\tau)\, \d\tau.$$
Furthermore, it holds
\begin{equation}\label{Eq: Inverse of time derivative}
 \bigl\| \bigl(\tfrac{\d}{\d t}+K\bigr)^{-1}\wt f \bigr\|_{\L^q (\IR_+ ; \B^s_{p , q} (\Omega ; \IC^d))} \leq 
 K^{-1} \| \wt f \|_{\L^q (\IR_+; \B^s_{p , q} (\Omega ; \IC^d))}.
\end{equation}
By abuse of notation, we will keep the same notation $\frac{\d}{\d t} +K$ to designate 
the time derivative plus $K$  on $\L^q (\IR_+ ; \B^{s + 1}_{p , q} (\Omega) \cap \L^p_0 (\Omega)).$ 
To solve the parabolic problem~\eqref{eq:auxiliary}, 
 define
\begin{align*}
 \wt a := (\tfrac{\d}{\d t}+K)^{-1} (\wt f - \divergence \wt u),
\end{align*}
where $\wt u$ is the unknown to be determined. Plugging this choice into the momentum equation delivers
\begin{align*}
 \bigl(\tfrac{\d}{\d t}+K\bigr)\wt u + {\bf L}_{p , q , s} \wt u - \bigl(\tfrac{\d}{\d t}+K\bigr)^{-1} \nabla \divergence \wt u = \wt g - \bigl(\tfrac{\d}{\d t}+K\bigr)^{-1} \nabla \wt f =: G.
\end{align*}
Notice that $G$ is a function in $\L^q (\IR_+ ; \B^s_{p , q} (\Omega ; \IC^d))$. To compute $\wt u$, introduce the new
 function $\wt v := (\tfrac{\d}{\d t} + K+ {\bf L}_{p , q , s}) \wt u$. Then,
\begin{align*}
 \bigl(\tfrac{\d}{\d t}+K\bigr) \wt u + {\bf L}_{p , q , s} \wt u - \bigl(\tfrac{\d}{\d t}+K\bigr)^{-1} \nabla \divergence \wt u 
 &= \wt v - \bigl(\tfrac{\d}{\d t}+K\bigr)^{-1} \nabla \divergence((\tfrac{\d}{\d t} +K+ {\bf L}_{p , q , s})^{-1} \wt v).
\end{align*}
Notice that by virtue of~\eqref{Eq: Inverse of time derivative}, 
Theorem~\ref{Thm: Maximal regularity of Lame on Besov} and Lemma~\ref{Lem: Domain of part}
\begin{align*}
 \bigl(\tfrac{\d}{\d t}+K\bigr)^{-1} \nabla \divergence(\tfrac{\d}{\d t} + K+ {\bf L}_{p , q , s})^{-1} : \L^q (\IR_+ ; \B^s_{p , q} (\Omega ; \IC^d)) \to \L^q (\IR_+; \B^s_{p , q} (\Omega ; \IC^d))
\end{align*}
is bounded and that there exists $C > 0$ (independent of $K$) such that
\begin{align*}
 \bigl\| \bigl(\tfrac{\d}{\d t}+K\bigr)^{-1} \nabla \divergence(\tfrac{\d}{\d t} +K+ {\bf L}_{p , q , s})^{-1} \|_{\Lop(\L^q (\IR_+; \B^s_{p , q} (\Omega ; \IC^d)))} \leq CK^{-1}.
\end{align*}
Thus, if taking $K>C$, then one may conclude that  the operator
\begin{align*}
 \Id -  (\tfrac{\d}{\d t}+K)^{-1} \nabla \divergence(\tfrac{\d}{\d t} + K+ {\bf L}_{p , q , s})^{-1} : \L^q (\IR_+; \B^s_{p , q} (\Omega ; \IC^d)) \to \L^q (\IR_+ ; \B^s_{p , q} (\Omega ; \IC^d))
\end{align*}
is invertible by a Neumann series argument.
This allows to express $\wt v$ in terms of $G,$ and to eventually get
$$ \wt u = \bigl(\tfrac{\d}{\d t} +K+ {\bf L}_{p , q , s}\bigr)^{-1} \bigl[ \Id -  \bigl(\tfrac{\d}{\d t}+K\bigr)^{-1} \nabla \divergence(\tfrac{\d}{\d t} +K+ {\bf L}_{p , q , s})^{-1} \bigr]^{-1} G.$$
Then, reverting to the original parabolic problem~\eqref{eq:lcns}, 
one can conclude the maximal $\L^q$-regularity of $A_{p , q , s}$ on  each interval $(0,T),$
with constant $C\e^{KT}$.

\subsubsection*{Second step: showing that  $0 \in \rho(A_{p , q , s})$}

To show surjectivity of $A_{p,q,s},$  we have to solve 
for all $(f , g) \in \cX^s_{p , q}$, the system
\begin{align}
\label{Eq: Resolvent problem0}
\left\{ \begin{aligned}
  \divergence u &= f && \text{in } \Omega \\
  {\bf L}_{p , q , s} u + \nabla a &= g && \text{in } \Omega \\
 u &= 0 && \text{on } \partial \Omega.
\end{aligned} \right.
\end{align}
Take $v \in \B^{s + 2}_{p , q , D} (\Omega ; \IC^d)$ such that $\divergence v = f$. The existence of  $v$ is guaranteed by interpolating the higher-order estimates in~\cite[Prop.~2.10]{Kozono_Sohr}.
\medbreak 
By considering $u = v + w$ and $h=g-  {\bf L}_{p , q , s}  v,$  the problem is thus reduced to \begin{align*}
\left\{ \begin{aligned}
 \divergence w &= 0 && \text{in } \Omega \\
 {\bf L}_{p , q , s} w + \nabla a &= h && \text{in } \Omega \\
 w &= 0 && \text{on } \partial \Omega.
\end{aligned} \right.
\end{align*}
Of course, since $\divergence w=0,$ we have ${\bf L}_{p , q , s} w=-\mu\Delta w,$ and 
we thus have only to consider the Stokes system with homogeneous boundary 
condition and source term in $\B^s_{p,q}(\Omega;\IC^d),$ which is 
standard and can also be derived by interpolating the result in~\cite[Prop.~2.10]{Kozono_Sohr}.
Finally, injectivity of $A_{p,q,s}$ is an obvious consequence of
the corresponding property for the Stokes system. 

\subsubsection*{Third step: showing that  $\IC_+^{\times} := \{ z \in \IC \setminus \{ 0 \} : \Re(z) \geq 0 \}$
is a subset of  $\rho (- A_{p , q , s})$}
Given $(f , g) \in \cX^s_{p , q}$ and $\lambda \in \IC,$ the resolvent problem for the operator $- A_{p , q , s}$ reads: 
\begin{align}
\label{Eq: Resolvent problem}
\left\{ \begin{aligned}
 \lambda a + \divergence u &= f && \text{in } \Omega \\
 \lambda u + {\bf L}_{p , q , s} u + \nabla a &= g && \text{in } \Omega \\
 u &= 0 && \text{on } \partial \Omega.
\end{aligned} \right.
\end{align}
As a first, we are going to show the result for a closed extension of $A_{p , q , s}$ on $\L^2_0 (\Omega) \times \W^{-1 , 2} (\Omega ; \IC^d)$. To this end, set
\begin{align*}
 \cX := \L^2_0 (\Omega) \times \W^{-1 , 2} (\Omega ; \IC^d) \qquad \text{and} \qquad \dom(\cA) := \L^2_0 (\Omega) \times \W^{1 , 2}_0 (\Omega ; \IC^d).
\end{align*}
With $\fa$ denoting the sesquilinear form defined in~\eqref{Eq: Sesquilinear form Lame}, define $\cA: \dom(\cA) \subset \cX \to \cX$ by
\begin{align*}
 \cA : \quad \begin{pmatrix} a \\ u \end{pmatrix} \mapsto \begin{pmatrix} \divergence u \\ \W^{1 , 2}_0 (\Omega ; \IC^d) \ni v \mapsto \fa (u , v) - \langle a , \divergence v \rangle_{\L^2} \end{pmatrix}\cdotp
\end{align*}

To investigate the resolvent problem for $\cA$ in the case $\lambda\not=0,$
 we  eliminate $a$ in the second equation 
of~\eqref{Eq: Resolvent problem}, getting
$$a=\lambda^{-1}(f-\divergence u)\andf \lambda u + {\bf L}_{p , q , s} u - \lambda^{-1}\nabla \divergence u= g-\lambda^{-1}\nabla f.$$
To determine $u,$ it is thus natural to consider the following  sesquilinear form:
$$ \fa_\lambda:\left\{\begin{aligned}
 &\W^{1 , 2}_0 (\Omega ; \IC^d) \times \W^{1 , 2}_0 (\Omega ; \IC^d) \longrightarrow \IC, \\ 
& (u , v) \longmapsto \lambda \int_{\Omega} u \cdot \overline{v} \; \d x + \mu \int_{\Omega} \curl u \cdot \overline{\curl v} \; \d x + (\mu + \mu^{\prime} + \lambda^{-1}) \int_{\Omega} \divergence u\, \overline{\divergence v} \; \d x.
\end{aligned}\right.$$
For all $\lambda\in\IC_+^{\times},$
 $\fa_\lambda$ is bounded  on the Hilbert space $\W^{1 , 2}_0 (\Omega ; \IC^d),$ and
  $\Re\lambda \geq 0$ implies that
\begin{align*}
 \Re \bigg( \lambda \int_{\Omega} \lvert u \rvert^2 \; \d x + \lambda^{-1} \int_{\Omega} \lvert \divergence u \rvert^2 \; \d x \bigg) \geq 0 \qquad (u \in \W^{1 , 2}_0 (\Omega ; \IC^d)).
\end{align*}
Consequently, employing~\cite[Lem.~3.1]{Mitrea_Monniaux} and $\sqrt{2} \lvert z + \alpha \rvert \geq \lvert z \rvert + \alpha$ whenever $\alpha \geq 0$ and $z \in \IC$ with $\Re(z) \geq 0$, we deduce that there exists $c > 0$ such that
\begin{align}
\label{Eq: Coercivity of a}
 \Re (\fa_\lambda (u , u)) \geq c \bigg\{\int_{\Omega} \lvert \nabla u \rvert^2 \; \d x 
 + \Re \biggl( \lambda \int_{\Omega} \lvert u \rvert^2 \; \d x + \lambda^{-1} \int_{\Omega} \lvert \divergence u \rvert^2 \; \d x \biggr)\bigg\}\cdotp
\end{align}
Omitting the second term on the right-hand side of~\eqref{Eq: Coercivity of a} and employing Poincar\'e's inequality yields a constant $C > 0$ such that
\begin{align}
\label{Eq: Coercivity with Poincare}
 \Re (\fa_\lambda (u , u)) \geq C \| u \|_{\W^{1 , 2}_0 (\Omega ; \IC^d)}^2 \qquad (u \in \W^{1 , 2}_0 (\Omega ; \IC^d),\  \lambda\in\IC^{\times}_+).
\end{align}
An application  Lax--Milgram's theorem then yields the following lemma.
\begin{lemma}
\label{Lem: Lax-Milgram for weak operator}
Let $\lambda \in \IC^{\times}_+$. For every $G \in \W^{-1 , 2} (\Omega ; \IC^d)$ there exists a unique $u \in \W^{1 , 2}_0 (\Omega ; \IC^d)$ such that
\begin{align*}
 \fa_\lambda(u , v) = \langle v , G \rangle_{\W^{1 , 2}_0 , \W^{-1 , 2}} \qquad (v \in \W^{1 , 2}_0 (\Omega ; \IC^d)).
\end{align*}
Furthermore, there exists $C > 0$ such that
\begin{align*}
 \| u \|_{\W^{1 , 2}_0 (\Omega ; \IC^d)} \leq C \| G \|_{\W^{-1 , 2} (\Omega ; \IC^d)} \qquad (G \in \W^{-1 , 2} (\Omega ; \IC^d)).
\end{align*}
\end{lemma}
The previous lemma opens the way to  prove that  
$\IC^{\times}_+ \subset \rho(- \cA)$.
Indeed, let  $u \in \W^{1 , 2}_0 (\Omega ; \IC^d)$ be the unique function provided by Lemma~\ref{Lem: Lax-Milgram for weak operator} that satisfies
\begin{align}
\label{Eq: Auxiliary variational problem}
 \fa_\lambda (u , v) = \langle v , G \rangle_{\W^{1 , 2}_0 , \W^{-1 , 2}} \qquad (v \in \W^{1 , 2}_0 (\Omega ; \IC^d)),\with G:=g-\lambda^{-1}\nabla f.
\end{align}
Then, remembering  $a := \lambda^{-1} (f - \divergence u) \in \L^2_0 (\Omega),$
relation~\eqref{Eq: Auxiliary variational problem} turns into
\begin{align*}
\langle v , g \rangle_{\W^{1 , 2}_0 , \W^{-1 , 2}} &+ \lambda^{-1} \int_{\Omega} f \, \overline{\divergence v} \; \d x \\
 &= \lambda \int_{\Omega} u \cdot \overline{v} \; \d x + \mu \int_{\Omega} \curl u \cdot \overline{\curl v} \; \d x + (\mu + \mu^{\prime} + \lambda^{-1}) \int_{\Omega} \divergence u\, \overline{\divergence v} \; \d x \\
 &= \lambda \int_{\Omega} u \cdot \overline{v} \; \d x + \mu \int_{\Omega} \curl u \cdot \overline{\curl v} \; \d x + (\mu + \mu^{\prime}) \int_{\Omega} \divergence u \,\overline{\divergence v} \; \d x \\
 &\hspace{5cm}+ \lambda^{-1} \int_{\Omega} f \,\overline{\divergence v} \; \d x - \int_{\Omega} a \,\overline{\divergence v} \; \d x.
\end{align*}
Consequently, $\lambda u - \mu \Delta u - \mu^{\prime} \nabla \divergence u + \nabla a = g$ holds in the sense of distributions. \medbreak
To show that $a$ and $u$ are unique, let $(f , g)\equiv(0,0)$. Eliminating $a$ by the relation $\lambda a = - \divergence u$ yields that $u \in \W^{1 , 2}_0 (\Omega ; \IC^d)$ must satisfy
\begin{align*}
 \fa_\lambda (u , u) = 0.
\end{align*}
By virtue of~\eqref{Eq: Coercivity with Poincare} this implies that $u = 0$ what in turn implies that $a = 0$. 
\medbreak
To conclude the proof of  $\lambda \in \rho(- \cA),$ it suffices to show the closedness of $\lambda + \cA$. 
For this purpose, assume that $(a_j , u_j) \in \dom(\cA)$ converges in $\cX$ to some element $(a , u)$  and that there exists $(f , g) \in \cX$ such that
\begin{align*}
\begin{aligned}
 \lambda a_j + \divergence u_j &=: f_j \to f && \text{in } \L^2_0 (\Omega) \qquad \text{and} \\
 \lambda u_j - \mu \Delta u_j - \mu^{\prime} \nabla \divergence u_j + \nabla a_j &=: g_j \to g && \text{in } \W^{-1 , 2} (\Omega ; \IC^d).
\end{aligned}
\end{align*}
 Eliminating again $a_j$ in the second equation, testing the respective equations for $u_j$ and $u_\ell$ by $u_j - u_\ell$, $j , \ell \in \IN$, and taking differences of the resulting equations yields
\begin{align*}
 \lvert \fa_\lambda(u_j - u_\ell , u_j - u_\ell) \rvert &\leq \| g_j - g_\ell \|_{\W^{-1 , 2} (\Omega ; \IC^d)} \| u_j - u_\ell \|_{\W^{1 , 2}_0 (\Omega ; \IC^d)} \\
 &\qquad+ \lvert \lambda^{-1} \rvert \| f_j - f_\ell \|_{\L^2 (\Omega)} \| \divergence u_j - \divergence u_\ell \|_{\L^2 (\Omega)}.
\end{align*}
By virtue of~\eqref{Eq: Coercivity with Poincare} and Young's inequality one obtains a constant $C > 0$ independent of $j$ and $\ell$ such that
\begin{align*}
 \| u_j - u_\ell \|_{\W^{1 , 2}_0 (\Omega ; \IC^d)} \leq C \big( \| g_j - g_\ell \|_{\W^{-1 , 2} (\Omega ; \IC^d)} + \| f_j - f_\ell \|_{\L^2 (\Omega)} \big)\cdotp
\end{align*}
Consequently, $u \in \W^{1 , 2}_0 (\Omega ; \IC^d).$ It follows that $(a , u) \in \dom(\cA)$ and that $(a , u)$ satisfies the equation $(\lambda + \cA) (a , u) = (f , g)$.
This completes the proof of 
\begin{equation}\label{eq: Resolvent of weak operator}
\IC_+^{\times}  \subset \rho (-\cA). 
\end{equation}  
It is now easy to show  the injectivity of $\lambda + A_{p , q , s }$ for $\lambda \in \IC_+^{\times}$.
Indeed, since
 $\cX^s_{p , q} \subset \cX$ (cf.~\eqref{Eq: Retract to L2-based spaces}) the operator $\cA$ is an extension of $A_{p , q , s}$. In particular, it holds $\dom(A_{p , q , s}) \subset \dom(\cA)$. Thus, $(\lambda + A_{p , q , s}) (a , u) = 0$ implies that $(\lambda + \cA) (a , u)$ = 0 and~\eqref{eq: Resolvent of weak operator} in turn implies that $(a , u) = 0$.
\medbreak
Let us finally show that 
the range of $\lambda + A_{p , q , s}$ is $\cX^s_{p , q}$ for all $\lambda \in \IC_+^{\times}$. Thus, let $(f , g) \in \cX^s_{p , q}$. Since $\cX^s_{p , q} \subset \cX$, ~\eqref{eq: Resolvent of weak operator} implies that there exists $(a , u) \in \dom(\cA)$ with
$$  (\lambda + \cA) \begin{pmatrix} a \\ u \end{pmatrix} = \begin{pmatrix} f \\ g \end{pmatrix}$$
 that is to say,
 \begin{align}
\label{Eq: Transition to complex coefficients}
 a = \lambda^{-1} (f - \divergence u) \quad \text{and} \quad - \mu \Delta u - (\lambda^{-1} + \mu^{\prime}) \nabla \divergence u = g - \lambda^{-1} \nabla f - \lambda u =: h.
\end{align}
Here, the second equation is fulfilled in $\W^{-1 , 2} (\Omega ; \IC^d)$. To prove the surjectivity of $\lambda + A_{p , q , s}$ it suffices to show $(a , u) \in \dom(A_{p , q , s})$, which follows once we derive $u \in \B^{s + 2}_{p , q , D} (\Omega ; \IC^d)$. \par

For this purpose, notice that by assumption it holds
\begin{align*}
 \mu + \Re(\mu^{\prime} + \lambda^{-1}) = \mu + \mu^{\prime} + \Re(\lambda) \lvert \lambda \rvert^{-2} > 0.
\end{align*}
Thus, the operator
\begin{align*}
 L^{\lambda} := - \mu \Delta - (\mu^{\prime} + \lambda^{-1}) \nabla \divergence \qquad (\lambda \in \IC^{\times}_+),
\end{align*}
belongs to the class of operators that was studied in 
the previous section.  Notice that $u \in \W^{1 , 2}_0 (\Omega ; \IC^d)$ implies that $u \in \B^s_{r_0 , q} (\Omega ; \IC^d)$ for all $1 < r_0 < \infty$ with 
$1 /r_0>(s-1)/d+1/ 2$. 
If $1/p>(s-1)/d+1/ 2,$ then one can take $r_0= p$ so that  the right-hand side $h$ defined in~\eqref{Eq: Transition to complex coefficients} lies in $\B^s_{p , q} (\Omega ; \IC^d).$ Then, 
by Lemma~\ref{Lem: Domain of part} and Proposition~\ref{Prop: Identification ground spaces}, it follows that $u \in \B^{s + 2}_{p , q , D} (\Omega ; \IC^d),$ and we are done. 

If $1/p\leq (s-1)/d+1/ 2,$ then any $r_0$ that satisfies the inequality above satisfies $1/p < 1/r_0$. Moreover, it is possible to choose $r_0$ large enough such that $s > -1 + 1/r_0,$ so that Lemma~\ref{Lem: Domain of part} together with  Proposition~\ref{Prop: Identification ground spaces}
guarantees that $u \in \B^{s + 2}_{r_0 , q , D} (\Omega ; \IC^d)$. Then, by Sobolev embedding, $h$ lies in a better space, which in turn implies that $u$ lies in a better space. Iterating this process delivers eventually $u \in \B^{s + 2}_{p , q , D} (\Omega ; \IC^d)$.

\subsubsection*{Last step: proving the global-in-time maximal regularity}

Step 1 tells us that  the operator $A_{p , q , s}$ has maximal $\L^q$-regularity on finite time intervals,
and generates an analytic semigroup. Hence, by virtue of Remark~\ref{Rem: Analyticity} there exists $\vartheta \in (\pi / 2 , \pi)$ and $\lambda_0 > 0$ such that  $[B(0 , \lambda_0)^c \cap \Sec_{\vartheta}] \subset \rho(- A_{p , q , s}),$ and  $C > 0$ such that for all $\lambda \in [B(0 , \lambda_0)^c \cap \Sec_{\vartheta}],$ 
it holds
\begin{align}
\label{Eq: Resolvent estimate in proof}
 \| \lambda (\lambda + A_{p , q , s})^{-1} \|_{\Lop(\cX^s_{p , q})} \leq C.
\end{align}
Moreover, by virtue of the second step and of the openness of the resolvent set, there exists $\eps > 0$ such that $B(0 , \eps) \subset \rho(- A_{p , q , s})$. 
Since $$D_{\eps,\lambda}:=\IC_+^{\times} \cap [\overline{B(0 , \lambda_0)} \setminus B(0 , \eps / 2)]$$ is compact and since $\IC_+^{\times} \subset \rho(- A_{p , q , s}),$ there exists $C > 0$ such that
 Inequality~\eqref{Eq: Resolvent estimate in proof} holds on $D_{\eps,\lambda}.$ Now, because the resolvent set is open
and the  boundary of $D_{\eps,\lambda}$ along the imaginary axis is compact,
one can eventually find some  $\theta \in (\pi / 2 , \vartheta)$ such that $\Sec_{\theta} \subset \rho(- A_{p , q , s})$ and there exists $C > 0$ such that~\eqref{Eq: Resolvent estimate in proof} holds for all $\lambda \in \Sec_{\theta}$, see also Figure~\ref{Fig: Sector}. It follows that $- A_{p , q , s}$ generates a bounded analytic semigroup. Moreover, since $0 \in \rho(A_{p , q , s})$ this semigroup is exponentially stable. Finally,~\cite[Thm.~5.2]{Dore} implies that $A_{p , q , s}$ has maximal $\L^q$-regularity on $\IR_+$.
\begin{figure}
\centering
 \includegraphics[width=0.55\textwidth]{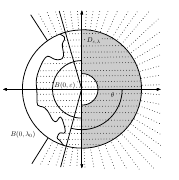}
 \caption{The ball $B(0 , \lambda_0)$ and half of the ball $B(0 , \eps)$ are depicted. The gray region visualizes the set $D_{\eps , \lambda}$. Due to the openness of the resolvent set (which is indicated by the dashed region), the spectrum must keep some distance to $D_{\eps , \lambda}$. One sees, that in this constellation one can find a sector $\Sec_{\theta}$ with $\theta \in (\pi / 2 , \vartheta)$ that is contained in the resolvent set.}
 \label{Fig: Sector}
\end{figure}
\end{proof}

For completeness, let us end the section proving Theorem~\ref{thm:linear}. 
As a start, we apply Theorem~\ref{Thm: Maximal regularity compressible}
with 
$q=1$ and notice that the last step of the proof ensures the existence of some $c>0$
depending only on $\mu,$  $\mu'$ and $\Omega$ so that 
$$\bigl\{z\in\IC: \Re (z)\geq -c\bigr\}\subset \rho(-A_{p,1,s}).$$
This implies that $A_{p,1,s}+\frac{c}{2}$ has maximal $\L^1$-regularity.
This yields Inequality~\eqref{eq:decay}. 
Of course,  Theorem~\ref{Thm: Maximal regularity compressible} 
directly yields that $(a,u)$ is in $\E_p.$ \par
To add non-zero initial data $(a_0 , u_0) \in \cX^s_{p , 1}$ in problem~\eqref{eq:LCNS} we cannot simply employ Corollary~\ref{Cor: Initial values}. The reason is that we would need to choose a ground space $\cX^t_{p , 1}$ for some $t$ slightly smaller than $s$. Then we would need to calculate the real interpolation space $(\cX^t_{p , 1} , \dom(A_{p , 1 , t}))_{\theta , 1}$. However, as the first components of $\cX^t_{p , 1}$ and $\dom(A_{p , 1 , t})$ are the same, the result of the real interpolation in this first component will be the very same space and thus we will not reach initial data in $\cX^s_{p , 1}$. \par
To circumvent this problem, consider the caloric extension
\begin{align*}
 \begin{pmatrix}
  a_c (t) \\ u_c (t)
 \end{pmatrix} := \begin{pmatrix} \e^{t \Delta_N} a_0 \\ \e^{t \Delta_D} u_0 \end{pmatrix} \cdotp
\end{align*}
Here, $\Delta_N$ denotes the Neumann Laplacian on $\B^{s + 1}_{p , 1} (\Omega) \cap \L^p_0 (\Omega)$ and $\Delta_D$ denotes the Dirichlet Laplacian on $\B^s_{p , 1} (\Omega ; \IC^d)$. Notice that both operators are invertible and that $\Delta_N$ generates a bounded analytic semigroup on $\L^p_0 (\Omega)$ while $\Delta_D$ generates a bounded analytic semigroup on $\W^{-1 , p} (\Omega ; \IC^d)$. An application of Corollary~\ref{Cor: Initial values} yields the existence of a constant $C > 0$ such that
\begin{align*}
 \| \partial_t a_c , a_c , \nabla^2 a_c \|_{\L^1 (\IR_+ ; \B^{s + 1}_{p , 1})} + \| \partial_t u_c , u_c , \nabla^2 u_c \|_{\L^1 (\IR_+ ; \B^s_{p , 1})} \leq C \| (a_0 , u_0) \|_{\cX^s_{p , 1}}.
\end{align*}
Notice that this together with the boundedness of the gradient operator between $\B^{s + 1}_{p , 1} (\Omega)$ and $\B^s_{p , 1} (\Omega ; \IC^d)$ implies that
\begin{align*}
 \| \nabla a_c \|_{\L^1 (\IR_+ ; \B^s_{p , 1})} \leq C \| a_c \|_{\L^1 (\IR_+ ; \B^{s + 1}_{p , 1})}.
\end{align*}
Now, let $(b , v) \in \E_p$ with $(b(0) , v(0)) = (0 , 0)$ solve
\begin{align*}
 \partial_t \begin{pmatrix}
  b \\ v
 \end{pmatrix} + A_{p , 1 , s} \begin{pmatrix} b \\ v \end{pmatrix} = - \partial_t \begin{pmatrix}
  a_c \\ u_c
 \end{pmatrix} - A_{p , 1 , s} \begin{pmatrix} a_c \\ u_c \end{pmatrix} \in \L^1 (\IR_+ ; \cX^s_{p , 1}).
\end{align*}
Then, for $a := b + a_c$ and $u := v + u_c$ one has $(a , u) \in \E_p$ and $(a , u)$ solve~\eqref{eq:LCNS} with $f$ and $g$ being zero and non-zero initial data.
\qed


\section{Global well-posedness for the compressible Navier--Stokes system}\label{s:global}

The fastest way  to solve  System~\eqref{eq:CNS} in the critical regularity 
setting is to  recast it in Lagrangian coordinates.
To this end,  let $X$ be the flow associated to $u,$ that is the solution to
\begin{equation}\label{eq:flow}
X(t,y)=y+\int_0^t u(\tau,X(\tau,y))\,\d\tau.
\end{equation}
The `Lagrangian' density and velocity are defined by 
\begin{equation}\label{eq:change}
\bar\rho(t,y):=\rho(t,X(t,y))\andf \bar u(t,y):=u(t,X(t,y)).
\end{equation}
With this notation,  relation~\eqref{eq:flow} becomes
\begin{equation}\label{eq:flow1}
X_\bu(t,y):=X(t,y)=y+\int_0^t \bar u(\tau,y)\,\d\tau,
\end{equation}
and thus 
\begin{equation}\label{eq:DX}
DX_\bu(t,y)=\Id+\int_0^t  D\bar u(\tau,y)\,\d\tau.
\end{equation}
The main interest of Lagrangian coordinates is that, whenever 
 $DX_\bu(t,y)$ is invertible,   the density is
entirely determined by $X_\bu$ and $\rho_0$ through the relation
\begin{equation}\label{eq:density}
\bar\rho(t,y)J_\bu(t,y)=\rho_0(y)\with J_\bu(t,y):= \det(DX_\bu(t,y)).
\end{equation}
Furthermore,  one can write
$$A_\bu(t,y):=(DX_\bu(t,y))^{-1}= J_\bu^{-1}(t,y)\,\adj(DX_\bu(t,y))$$
where $\adj(DX_\bu)$ (the adjugate matrix) stands for 
the transpose of the comatrix of $DX_\bu(t,y).$ 
Define the `twisted' deformation tensor and divergence operator by 
$$
D_A(z):=\frac12\Bigl(Dz\cdot A+{}^T\! A\cdot\nabla z\Bigr)
\andf \diva z:={}^T\!A:\nabla z=Dz:A,\quad \bigl(A\in \IR^d\times\IR^d\bigr)\cdotp$$
As shown in, e.g.,~\cite{D-Fourier}, 
in terms of the unknowns $\bar a:=\bar\rho-1$ and $\bar u,$ 
System~\eqref{eq:CNS}  translates into
\begin{equation}\label{eq:CNSlag}\left\{\begin{aligned}
&J_\bu\partial_t\bar  a+ (1+\bar a)D\bar u: \adj(DX_\bu)=0\quad&\hbox{in }\ \IR_+\times\Omega,\\
&\rho_0\partial_t\bar u-2\divergence\bigl(\mu(1+\bar a)\adj(DX_\bu)\cdot D_{A_\bu}(\bar u)\bigr)
-\nabla\bigl(\lambda(1+\bar a)\divergence_{A_\bu}\bar u\bigr)&\\&\hspace{5.5cm} +{}^T\!\!\adj(DX_\bu)\cdot\nabla(P(1+\bar a))=0&\quad\hbox{in }\ \IR_+\times\Omega,\\
&\bar u=0&\quad\hbox{on }\ \IR_+\times\partial\Omega,\\
&(\bar a,\bar u)|_{t=0}=(a_0,u_0)&\quad\hbox{in }\ \Omega.
\end{aligned}\right.\end{equation}
As pointed out in  the Appendix of~\cite{D-Fourier} (for $\IR^d$  but 
the proof in the bounded domain case is similar), in our functional 
framework, there exists $\eps>0$ such that whenever
\begin{equation}\label{eq:Du}
\int_0^T\|\nabla u\|_{\B^{d/p}_{p,1}(\Omega)}\,\d t\leq\eps,\end{equation}
 the Eulerian and Lagrangian formulations 
of the compressible Navier--Stokes equations are equivalent on $[0,T].$
\medbreak
The present section aims at proving a global existence result for small $(a_0,u_0)$  in the case $P'(1)>0.$ Note that,  after rescaling the time and velocity according to~\eqref{eq:rescaling}, 
System~\eqref{eq:CNSlag}  may be rewritten exactly as
\eqref{eq:lcns} with 
$$\begin{aligned}
f&:=(1-J_\bu)\partial_t\bar  a
+D\bar u: (\Id-\adj(DX_\bu)) - \bar a D\bar u:\adj(DX_\bu) ,\\
g&:=-a_0\partial_t\bar u+2\divergence\bigl(\wt\mu(\bar a)\adj(DX_\bu)\cdot D_{A_\bu}(\bar u)-\bar\mu D(\bar u)\bigr)
\\&\hspace{3cm}+\nabla\bigl((\wt\lambda(\bar a)\divergence_{\!A_\bu}\bar u-\bar\lambda\divergence\bar u)\bigr)
+(1-\Pi(\bar a))\nabla\bar a+\Pi(\bar a)(\Id-\adj(DX_\bu))\cdot\nabla\bar a.
\end{aligned}$$
Above, we denoted  $\wt\mu(z):=\mu(1+z),$
$\wt\lambda(z):=\lambda(1+z),$ $\Pi(z):=P'(1+z),$
$\bar\mu:=\mu(1)$ and $\bar\lambda:=\lambda(1).$ 
\medbreak
In the critical regularity setting, if we restrict ourselves to small perturbations of $(0,0),$
then one can expect  $f$ and $g$ (that contain only at least quadratic terms) 
to be even smaller. Hence, it looks reasonable to get  a global existence
result for~\eqref{eq:CNSlag} by taking advantage of our estimates for  the linearized system.
{}From the linear theory, we have the constraint $d/p-1 < 1/p$ (that is $p>d-1$)
and, when handling the nonlinear terms,  the additional conditions $p<2d$ and $d\geq2$ will  
  pop up. In the end, we will obtain the following result,
that is the counterpart of Theorem~\ref{Thm:global}, in Lagrangian coordinates.  Recall that $\E_p$ was defined by
\begin{align*}
 \E_p= \W^{1 , 1} (\IR_+ ; \B^{{d}/{p}}_{p,1} (\Omega) \times\B^{{d}/{p} - 1}_{p,1} (\Omega ; \IR^d)) \cap \L^1 (\IR_+ ; \B^{{d}/{p}}_{p,1} (\Omega)\times\B^{{d}/{p} + 1}_{p,1} (\Omega ; \IR^d))\cdotp
\end{align*}
\begin{proposition}\label{p:global}
Let the assumptions of Theorem~\ref{Thm:global} be in force.
Then, System~\eqref{eq:CNSlag} admits a unique global solution $(\bar a,\bar u)$
in the maximal regularity space $\E_p,$ and 
 there exist two positive constants  $c$ and $C$  depending only on the parameters of the system, 
on $p$, and on $\Omega,$ such that
\begin{eqnarray}\label{eq:globalbound}&\|\e^{ct}(\bar a,\bar u)\|_{\E_p}\leq C \bigl(\|a_0\|_{\B^{d/p}_{p,1}}
+ \|u_0\|_{\B^{d/p-1}_{p,1}}\bigr)\cdotp
\end{eqnarray}
\end{proposition}
\begin{proof}
Throughout, we use the short notation ${\bf A}$ for ${A_{p,1,d/p-1}}.$
The proof of existence is based on the fixed point theorem in the space  $\E_p^c$ defined by 
$$\E_p^c:=\bigl\{ (a,u): \IR_+\times\Omega \to \IR\times\IR^d\  \hbox{ s.t. }\ 
(\e^{tc}a, \e^{tc}u)\in\E_p\bigr\}$$
 for the map $\Phi:(\bb,\bv)\mapsto(\ba,\bu),$ where 
$(\ba,\bu)$ stands for the solution in $\E_p^c$ to the linear system
 \begin{equation}\label{eq:CNSlagbis}
 \tfrac{\d}{\d t} \begin{pmatrix} \ba \\ \bu \end{pmatrix} + {\bf A} \begin{pmatrix} \ba \\ \bu \end{pmatrix} = \begin{pmatrix} \bbf \\ \bg \end{pmatrix}
\end{equation} 
supplemented with initial data $(a_0,u_0)$ and  
$$\begin{aligned} \bbf&:=(1-J_{\bv})\partial_t\bb
+D\bv: (\Id-\adj(DX_\bv)) - \bb D\bv:\adj(DX_\bv) ,\\
\bg&:=-a_0\partial_t\bv
+2\divergence\bigl(\wt\mu(\bb)\adj(DX_\bv)\cdot D_{A_\bv}(\bv)-\bar\mu D(\bv)\bigr)\\
&\hspace{2cm}+\nabla\bigl((\wt\lambda(\bb)\divergence_{A_\bv}\bv-\bar\lambda\divergence\bv)\bigr)
+(1-\Pi(\bb))\nabla\bb+\Pi(\bb)(\Id-\adj(DX_\bv))\cdot\nabla\bb.
\end{aligned}$$
We claim that there exists some $R\in(0,1)$ such that, whenever $(\bb,\bv)$ belongs 
to the closed ball $\bar B_{\E_p^c}(0,R) := \{ (b,v) \in \E_p^c : \| (b,v) \|_{\E_p^c} \leq R \},$ System~\eqref{eq:CNSlagbis} 
admits a solution in $\bar B_{\E_p^c}(0,R).$ 
 Now, from Theorem~\ref{Thm: Maximal regularity compressible}, we gather that there exists some $c>0$ depending
 only on $\Omega,$ $p,$ $\mu$ and $\mu'$ such that
\begin{equation}\label{eq:gwp}
\|(\bar a,\bar u)\|_{\E_p^c} \lesssim \|(a_0,u_0)\|_{\cX^{d/p-1}_{p,1}}
+ \|\e^{ct}(\bbf,\bg)\|_{\L^1(\IR_+;{\cX^{d/p-1}_{p,1}})}.\end{equation}
Hence  our problem reduces to proving suitable estimates for $\bbf$ and $\bg.$
To this end, we need the following two results proved in Appendix:
\begin{proposition}\label{p:product} The numerical product is
continuous from $\B^s_{p,1}(\Omega) \times \B^{d/p}_{p,1}(\Omega)$ to $\B^s_{p,1}(\Omega)$
whenever  $-\min(d/p,d/p')<s\leq d/p.$\end{proposition}
\begin{proposition}\label{p:compo}
Let $K:\IR\to\IR$ be a smooth function vanishing at $0,$ and $p\in[1,\infty).$ 
 Then,  there exists $C > 0$  such that for all functions $z$
 belonging to $\B^{d/p}_{p,1}(\Omega),$ the function $K(z)$ belongs
 to  $\B^{d/p}_{p,1}(\Omega)$ and satisfies 
 $$ \|K(z)\|_{\B^{d/p}_{p,1}(\Omega)}\leq C (1+ \|z\|_{\B^{d/p}_{p,1}(\Omega)})^k
  \|z\|_{\B^{d/p}_{p,1}(\Omega)}\with k:=\lceil d/p\rceil.$$
  Furthermore (without assuming $K(0)=0$), for all pairs $(z_1,z_2)$ of functions
  in  $\B^{d/p}_{p,1}(\Omega),$  we have
  $$ \|K(z_2)-K(z_1)\|_{\B^{d/p}_{p,1}(\Omega)}\leq C(1+ \|z_1\|_{\B^{d/p}_{p,1}(\Omega)}
  + \|z_2\|_{\B^{d/p}_{p,1}(\Omega)})^k
  \|z_2-z_1\|_{\B^{d/p}_{p,1}(\Omega)}\with k:=\lceil d/p\rceil.$$
  \end{proposition}
For notational simplicity, we omit from now on the dependency on $\Omega$ in the norms.
Assume that $R$ has been chosen  so small as
\begin{equation}\label{eq:smallDv}
\|D\bv\|_{\L^1(\IR_+;\B^{d/p}_{p,1})}\leq\eps\ll1.
\end{equation}
In particular, owing to the embedding 
\begin{equation}\label{eq:embedLinfty}
\B^{d/p}_{p,1}(\Omega)\hookrightarrow \L^\infty(\Omega),
\end{equation}
the range of $\bb$ is included in  a small neighborhood of $0$ and the
functions $\wt\mu,$ $\wt\lambda$, and  $\Pi$ may thus be
extended smoothly to the whole $\IR$ without changing the value of $\bg$.
This allows to apply Proposition~\ref{p:compo} whenever it is needed.
\medbreak
Now,  decompose $\bbf$ into
$$\begin{aligned}\bbf &= (1-J_\bv)\partial_t\bb +D\bv:(\Id-\adj(DX_\bv))
-\bb\, D\bv:\adj(DX_\bv)\\
&=\bbf^1+\bbf^2+\bbf^3.
\end{aligned}$$
 Proposition~\ref{p:product} ensures that the space $\B^{d/p}_{p,1}$
 is stable under products.  Hence
$$\begin{aligned}\|\bbf^1\|_{\B^{d/p}_{p,1}}&\lesssim
 \|1-J_\bv\|_{\B^{d/p}_{p,1}} \|\partial_t\bar b\|_{\B^{d/p}_{p,1}},\\
\|\bbf^2\|_{\B^{d/p}_{p,1}}&\lesssim\|D\bv\|_{\B^{d/p}_{p,1}}
\|\Id-\adj(DX_\bv)\|_{\B^{d/p}_{p,1}},\\
\|\bbf^3\|_{\B^{d/p}_{p,1}}&\lesssim
\|\bb\|_{\B^{d/p}_{p,1}} \|D\bv\|_{\B^{d/p}_{p,1}}
\bigl(1+\|\Id-\adj(DX_\bv)\|_{\B^{d/p}_{p,1}}\bigr)\cdotp\end{aligned}$$
In order to bound  the right-hand sides (as well as  the terms
in $\bg$  below), we will use repeatedly the following inequality 
that is based on Neumann expansion arguments,~\eqref{eq:smallDv} and 
on the fact that $\B^{d/p}_{p,1}$ is stable under products
 (see details in the Appendix of~\cite{D-Fourier}  for the $\IR^d$ situation): 
 \begin{equation}\label{eq:J}
 \sup_{t\geq0} \Bigl(\|1-J_\bv(t)\|_{\B^{d/p}_{p,1}}
 +\|A_\bv(t)-\Id\|_{\B^{d/p}_{p,1}}+
\|\adj(DX_\bv(t))-\Id\|_{\B^{d/p}_{p,1}}\Bigr)\lesssim \|D\bv\|_{\L^1(\IR_+;\B^{d/p}_{p,1})}.
 \end{equation}
In the end, we  get
\begin{multline}\label{eq:f}
\|\e^{ct}\bbf\|_{\L^1(\IR_+;\B^{d/p}_{p,1})}\lesssim
\bigl(\|\e^{ct}\partial_t\bb\|_{\L^1(\IR_+;\B^{d/p}_{p,1})}\\
+(1+\|D\bv\|_{\L^\infty(\IR_+;\B^{d/p}_{p,1})})\|\e^{ct}\bb\|_{\L^\infty(\IR_+;\B^{d/p}_{p,1})}
+\|\e^{ct}D\bv\|_{\L^1(\IR_+;\B^{d/p}_{p,1})}\bigr)
\|D\bv\|_{\L^1(\IR_+;\B^{d/p}_{p,1})}.
\end{multline}
Next, we have to bound in $\L^1(\IR_+;\B^{d/p-1}_{p,1})$ the five terms 
constituting $\bg,$ namely
$$\begin{array}{ll}
\bg^1:= -a_0\partial_t\bv,
&\quad \bg^2:=2\divergence\bigl(\wt\mu(\bb)\adj(DX_\bv)\cdot D_{A_\bv}(\bv)-\bar\mu D(\bv)\bigr),\\[1.5ex]
\bg^3:= \nabla\bigl((\wt\lambda(\bb)\divergence_{A_\bv}\bv-\bar\lambda\divergence\bv)\bigr),
&\quad \bg^4:=(1-\Pi(\bb))\nabla\bb,\\[1.5ex]
\bg^5:=\Pi(\bb)(\Id-\adj(DX_\bv))\cdot\nabla\bb.&\end{array}$$ 
For $\bg^1,$  a direct application of Proposition~\ref{p:product} 
yields, provided $p<2d$ and $d\geq2,$ 
\begin{equation}\label{eq:g1}
\|\e^{ct}\bg^1\|_{\L^1(\IR_+;\B^{d/p-1}_{p,1})}\lesssim \|a_0\|_{\B^{d/p}_{p,1}}
\|\e^{ct}\partial_t\bv\|_{\L^1(\IR_+;\B^{d/p-1}_{p,1})}.\end{equation}
Similarly, combining Propositions~\ref{p:product} and~\ref{p:compo} yields
$$
\|\e^{ct}\bg^4\|_{\L^1(\IR_+;\B^{d/p-1}_{p,1})}\lesssim \|\e^{ct}\bb\|_{\L^\infty(\IR_+;\B^{d/p}_{p,1})}
\|\nabla\bb\|_{\L^1(\IR_+;\B^{d/p-1}_{p,1})}$$
and since (argue  by extension)
\begin{equation}\label{eq:deriv}
\nabla: \B^{s+1}_{p,1}(\Omega)\to\B^s_{p,1}(\Omega) \quad \text{is a bounded operator},
\end{equation}
one can conclude  that 
\begin{equation}\label{eq:g4}
\|\e^{ct}\bg^4\|_{\L^1(\IR_+;\B^{d/p-1}_{p,1})}\lesssim \|\e^{ct}\bb\|_{\L^\infty(\IR_+;\B^{d/p}_{p,1})}
\|\bb\|_{\L^1(\IR_+;\B^{d/p}_{p,1})}.\end{equation}
To handle $\bg^2$, we use the decomposition
$$
\bg^2=2\divergence\Bigl((\wt\mu(\bb)-\bar\mu)\adj(DX_\bv)\cdot D_{A_\bv}(\bv)
+\bar\mu(\adj(DX_\bv)-\Id)\cdot D_{A_\bv}(\bv)+ \bar\mu (D_{A_\bv}(\bv)-D(\bv))\Bigr)\cdotp
$$
{}From  the definition of $D_{A_\bv},$~\eqref{eq:smallDv} and~\eqref{eq:J}, we gather that
$$\|\e^{ct}(D_{A_\bv}(\bv)-D(\bv))\|_{\L^1(\IR_+;\B^{d/p}_{p,1})}
\lesssim \|\e^{ct}D\bv\|_{\L^1(\IR_+;\B^{d/p}_{p,1})} \|D\bv\|_{\L^1(\IR_+;\B^{d/p}_{p,1})}.$$
Hence, combining with Propositions~\ref{p:product} and~\ref{p:compo},~\eqref{eq:J} and~\eqref{eq:deriv}, as~\eqref{eq:smallDv} is fulfilled, we get 
\begin{equation}\label{eq:g2}
\|\e^{ct}\bg^2\|_{\L^1(\IR_+;\B^{d/p-1}_{p,1})}\lesssim
\bigl(1+\|\bb\|_{\L^\infty(\IR_+;\B^{d/p}_{p,1})}\bigr) \|\e^{ct}D\bv\|_{\L^1(\IR_+;\B^{d/p}_{p,1})}\|D\bv\|_{\L^1(\IR_+;\B^{d/p}_{p,1})}.
\end{equation}
Bounding $\bg^3$ is exactly the same. 
Finally, we have
$$\bg^5= \bigl(1+ (\Pi(\bb)-1)\bigr)(\Id-\adj(DX_\bv))\cdot\nabla\bb$$
and thus, combining Propositions~\ref{p:product} and~\ref{p:compo} with~\eqref{eq:deriv}, 
one ends up with 
\begin{equation}\label{eq:g5}
\|\e^{ct}\bg^5\|_{\L^1(\IR_+;\B^{d/p-1}_{p,1})}\lesssim 
\bigl(1+\|\bb\|_{\L^\infty(\IR_+;\B^{d/p}_{p,1})}\bigr)\|D\bv\|_{\L^1(\IR_+;\B^{d/p}_{p,1})}
\|\e^{ct}\bb\|_{\L^1(\IR_+;\B^{d/p}_{p,1})}.\end{equation}
Recall that the embedding $\W^{1 , 1} (\IR_+ ; \B^{{d}/{p}}_{p , 1}) \hookrightarrow \L^{\infty} (\IR_+ ; \B^{{d}/{p}}_{p , 1})$ allows to control the $\L^{\infty}$-norms of quantities involving $\bb$ by their norm in $\E_p$. Now, plugging Inequalities~\eqref{eq:f},~\eqref{eq:g1},~\eqref{eq:g4}, 
\eqref{eq:g2}, and~\eqref{eq:g5} in~\eqref{eq:gwp} and 
using the definition of  the  norm in  $\E_p$ yields 
$$\|\e^{ct}(\ba,\bu)\|_{\E_p} \leq C\Bigl(\|(a_0,u_0)\|_{\cX^{d/p-1}_{p,1}}
+ (1+\|(\bb,\bv)\|_{\E_p})\|(\bb,\bv)\|_{\E_p}\|\e^{ct}(\bb,\bv)\|_{\E_p}\Bigr)\cdotp$$
Remembering~\eqref{eq:smalldata} and 
 $(\bb,\bv)\in \bar B_{\E_p^c}(0,R)$  with $R\in(0,1),$ one thus gets up to a change of $C,$
$$\|(\ba,\bu)\|_{\E_p^c} \leq C(\alpha +R^2).$$
Therefore, choosing $R=2C\alpha$ and assuming that $4C\alpha\leq1,$
one can conclude that $(\ba,\bu)\in \bar B_{\E_p^c}(0,R).$
\medbreak
To complete the proof of existence of a fixed point for $\Phi,$ it is only a matter of
exhibiting its properties of contraction. So let us 
consider $(\bb_i,\bv_i)\in \bar B_{\E_p^c}(0,R)$ and $(\ba_i,\bu_i):=\Phi(\bb_i,\bv_i),$ $i=1,2.$
Denote $(\bbf_i,\bg_i),$ $i=1,2$ the right-hand sides of System \eqref{eq:CNSlagbis}
corresponding to $(\bb_i,\bv_i).$ 
Then, from Theorem~\ref{thm:linear}, we gather
\begin{equation}\label{eq:gwp1}
\|(\da,\du)\|_{\E_p^c} \lesssim  \|\e^{ct}(\df,\dg)\|_{\L^1(\IR_+;{\cX^{d/p-1}_{p,1}})},\end{equation}
where $\da:=\ba_2-\ba_1,$ $\du:=\bu_2-\bu_1,$ $\df:=\bbf_2-\bbf_1$, and  $\dg:=\bg_2-\bg_1.$
\medbreak
Let us use the short notation $\divergence_i:=\divergence_{v_i}$ and so on and also introduce $\db:=\bb_2-\bb_1$ and $\dv:=\bv_2-\bv_1$. We see that 
$$\begin{aligned}
\df^1&= (1-J_1)\partial_t\db  + (J_1-J_2)\partial_t\bb_2 ,\\
\df^2&= D\bv_1:\bigl(\adj(DX_1)-\adj(DX_2)\bigr)+ D\dv:(\Id-\adj(DX_2)),\\
\df^3&=\bb_1\bigl(D\bv_1:(\adj(DX_1)-\adj(DX_2))-D\dv:\adj(DX_2)\bigr) - \db\, D\bv_2:\adj(DX_2).
\end{aligned}$$
Since we have
$$A_2(t)-A_1(t) = \sum_{k=1}^\infty(-1)^k \sum_{j=0}^{k-1} \biggl(\int_0^t D\bv_2\,\d\tau\biggr)^j
\biggl(\int_0^t D\dv\,\d\tau\biggr)  \biggl(\int_0^t D\bv_1\,\d\tau\biggr)^{k-1-j}$$
and similar identities\footnote{More details may be found in the appendix of~\cite{D-Fourier}.}
  for $J_2(t)-J_1(t)$ and $\adj(DX_2(t))-\adj(DX_1(t)),$
we get thanks to the stability of $\B^{d/p}_{p,1}$ under multiplication and 
to~\eqref{eq:smallDv} (remember that $R$ is small) that for all $t\geq0,$
\begin{multline}\label{eq:dA}
\|A_2(t)-A_1(t)\|_{\B^{d/p}_{p,1}}
+\|\adj(DX_2(t))-\adj(DX_1(t))\|_{\B^{d/p}_{p,1}}\\
+\|J_2^{\pm1}(t)-J_1^{\pm1}(t)\|_{\B^{d/p}_{p,1}}\lesssim
 \|D\dv\|_{\L^1(\IR_+;\B^{d/p}_{p,1})}.\end{multline}
 Hence,  using 
once more the stability of $\B^{d/p}_{p,1}$ under multiplication
eventually yields
\begin{multline}\label{eq:df} \| \e^{ct} \df\|_{\L^1(\IR_+; \B^{d/p}_{p,1})}\lesssim  \bigl(\|D\bv_1\|_{\L^1(\IR_+;\B^{d/p}_{p,1})}+\|D\bv_2\|_{\L^1(\IR_+;\B^{d/p}_{p,1})} \\+ \| \partial_t \bb_2 \|_{\L^1 (\IR_+ ; \B^{{d}/{p}}_{p , 1})} + \| \bb_1 \|_{\L^{\infty} (\IR_+ ; \B^{{d}/{p}}_{p , 1})} \bigr) 
 \cdot \bigl(\| \e^{ct} D\dv\|_{\L^1(\IR_+;\B^{d/p}_{p,1})} + \| \e^{ct} \partial_t \db \|_{\L^1 (\IR_+ ; \B^{{d}/{p}}_{p , 1})} \bigr). \end{multline}
 We compute:
 $$\begin{aligned}
\dg^1&:= -a_0\partial_t\dv,\\
\dg^2&:=2\divergence\Bigl((\wt\mu(\bb_2)-\wt\mu(\bb_1))\adj(DX_2)\cdot D_{A_2}(\bv_2)\\
&+\wt\mu(\bb_1)\bigl(\bigl(\adj(DX_1)-\Id\bigr)\cdot\bigl(D_{A_2}(\bv_2)- D_{A_1}(\bv_1)\bigr)
+\bigl(\adj(DX_2)-\adj(DX_1)\bigr)\cdot D_{A_2}(v_2)\bigr)\\
&\qquad\qquad+(\wt\mu(\bb_1)-\bar\mu)\bigl(D_{A_2}(\bv_2)- D_{A_1}(\bv_1)\bigr)
+\bar\mu\bigl(D_{A_2}(\bv_2)- D_{A_1}(\bv_1)-D(\dv)\bigr)\Bigr),\\
\dg^3&:= \nabla\Bigl(\bigl(\wt\lambda(\bb_2)-\wt\lambda(\bb_1)\bigr)\divergence_{A_2}\bv_2
+\bigl(\wt\lambda(\bb_1)-\bar\lambda\bigr)
\bigl(\bigl(\divergence_{A_2}\bv_2-\divergence_{A_2}\bv_1\bigr)
\\&\qquad\qquad\qquad+\bigl(\divergence_{A_2}\bv_1-\divergence_{A_1}\bv_1\bigr)\bigr)
+\bar\lambda\bigl(\divergence_{A_2}\bv_2- \divergence_{A_1}\bv_1-\divergence\dv\bigr)\Bigr),\\
\dg^4&:=(1-\Pi(\bb_1))\nabla\db+\bigl(\Pi(\bb_1)-\Pi(\bb_2)\bigr)\nabla\bb_2,\\
\dg^5&:=\Pi(\bb_1)(\Id-\adj(DX_1))\cdot\nabla\db + \Pi(\bb_1)(\adj(DX_1)-\adj(DX_2))\cdot\nabla\bb_2\\
&\hspace{7cm}+ (\Pi(\bb_2)-\Pi(\bb_1))(\Id-\adj(DX_2))\cdot\nabla\bb_2.
\end{aligned} $$ 
It is straightforward that 
\begin{equation}\label{eq:dg1}
\|\e^{c t} \dg^1\|_{\B^{d/p-1}_{p,1}}\leq C\|a_0\|_{\B^{d/p}_{p,1}} \| \e^{ct} \partial_t\dv\|_{\B^{d/p-1}_{p,1}}.
\end{equation}
Next, from Propositions~\ref{p:product} and~\ref{p:compo},  ~\eqref{eq:deriv} and Inequality~\eqref{eq:dA},
we easily get for $i=2,3,4,5,$
$$
\begin{aligned}
\|\e^{ct} \dg^i\|_{\L^1 (\IR_+ ; \B^{d/p-1}_{p,1})} 
&\lesssim \Big(\|\bb_2, \nabla \bv_1 , \nabla\bv_2\|_{\L^1(\IR_+ ; \B^{d/p}_{p,1})} + \| \bb_1 \|_{\L^{\infty} (\IR_+ ; \B^{{d}/{p}}_{p , 1})} \Big) \| \e^{ct} \db\|_{\L^1(\IR_+ ; \B^{d/p}_{p,1})} \\
 &\quad + \Big( \|\bb_1 , \bb_2, \nabla \bv_1, \nabla \bv_2\|_{\L^1(\IR_+ ; \B^{d/p}_{p,1})} + \|\bb_1\|_{\L^\infty(\IR_+ ; \B^{d/p}_{p,1})} \Big) \| \e^{ct} \db \|_{\L^{\infty}(\IR_+ ; \B^{d/p}_{p,1})} \\
 &\quad + \Big( \|\bb_1 , \bb_2, \nabla \bv_1, \nabla \bv_2\|_{\L^1(\IR_+ ; \B^{d/p}_{p,1})} + \|\bb_1\|_{\L^\infty(\IR_+ ; \B^{d/p}_{p,1})} \Big) \| \e^{ct} \nabla\dv\|_{\L^1(\IR_+ ; \B^{d/p}_{p,1})}.
 \end{aligned}
$$
Note again, that the embedding $\W^{1 , 1} (\IR_+ ; \B^{{d}/{p}}_{p , 1}) \hookrightarrow \L^{\infty} (\IR_+ ; \B^{{d}/{p}}_{p , 1})$ allows to control the $\L^{\infty}$-norms of quantities involving $\bb_1$ or $\bb_2$ by their norm in $\E_p$. Altogether, we conclude that 
$$\|(\da,\du)\|_{\E_p^c} \leq C(R+\alpha)\|(\db,\dv)\|_{\E_p^c}.$$
Since we chose $R$ of order $\alpha,$ we see that, indeed, the map $\Phi$ 
is contracting provided $\alpha$ is small enough. 
Then, Banach fixed point theorem ensures that $\Phi$ admits
a fixed point in $\bar B_{\E_p^c}(0,R).$
Hence, we have a solution for~\eqref{eq:CNSlag} with the desired property. 
\medbreak
In order to prove the uniqueness, consider two solutions $(\ba_1,\bu_1)$ and $(\ba_2,\bu_2)$  in $\E_p^c$ of~\eqref{eq:CNSlag} supplemented with  the same data $(\rho_0,u_0).$ 
Then, we have $(\ba_i,\bu_i)=\Phi((\ba_i,\bu_i)),$ $i=1,2,$  and one can 
repeat the previous computation on  any interval $[0,T]$ such that 
$$\max\biggl(\int_0^T\|\nabla\bu_1\|_{\B^{d/p}_{p,1}}\,\d t,
\int_0^T\|\nabla\bu_2\|_{\B^{d/p}_{p,1}}\,\d t\biggr)\leq \eps\ll1.$$
On such an interval, we obtain  (with obvious notation)
 $$\|(\da,\du)\|_{\E_p(T)} \leq  C \bigl(\|(\ba_1,\bu_1)\|_{\E_p(T)}+\|(\da,\du)\|_{\E_p(T)}\bigr)
 \|(\da,\du)\|_{\E_p(T)}.$$
 Since the function $t\mapsto \|(\da,\du)\|_{\E_p(t)}$ is continuous and
 vanishes at $0$ and because one can assume with no loss of generality
 that $(\ba_1,\bu_1)$ is the small solution constructed just above, 
 we get uniqueness on $[0,T].$
 Then, using a standard bootstrap argument yields uniqueness for all time. 
 \end{proof}


\section{Local existence for general data with no vacuum}\label{s:local}

For achieving the  local well-posedness  of the compressible Navier--Stokes equations,
 there is no need to take
the linear coupling of the density and velocity equations into consideration, and 
the sign of $P'$ does not matter.
Actually, in the Lagrangian formulation~\eqref{eq:CNSlag},
 it is  enough to solve the velocity  equation, since $J_\bu\bar\rho=\rho_0$ and 
$J_\bu$~may be computed from $\bar u.$ 
The pressure may be seen as a source term,  and combining 
Corollary~\ref{c:lame} with $s=d/p-1$ and $q=1,$
with suitable nonlinear estimates allows  to solve~\eqref{eq:CNSlag} locally in the critical  regularity setting. 

Clearly, a basic perturbative method relying on our reference linear system with  constant coefficients  is bound to  fail if the density variations are too large.
However, since, in our  functional setting, $\rho_0$ has to be
uniformly continuous in $\Omega,$  one can expect that difficulty  to be challengeable
if using a suitable localization argument. 
\medbreak
Here, for expository purpose,  we first present the proof 
of the local well-posedness in the easier case where $\rho_0$ is close to some positive
constant. Then, we  explain what has to be modified to tackle the general 
case where one just assumes  that it is bounded away from $0.$

\subsection{The case of small variations of density}

Our goal here is to establish the following result that  implies Theorem~\ref{Thm:local}
in the case of small  density variations.
 \begin{proposition}\label{p:local} Let the assumptions of 
 Theorem~\ref{Thm:local} be in force, and assume in addition 
 that, for a small enough $\alpha>0,$ we have
 \begin{equation}\label{eq:smalldata2} 
\|a_0\|_{\B^{d/p}_{p,1}(\Omega)}\leq \alpha.
\end{equation}
Then, System~\eqref{eq:CNSlag} admits a unique  solution $(\bar a,\bar u)$
on some interval $[0,T],$  such that $1+\bar a := J_\bu^{-1}\rho_0$ is bounded
away from zero on $[0,T]\times\Omega$ and 
$$(\ba,\bu)\in \W^{1 , 1}(0,T;\B^{d/p}_{p,1} (\Omega) \times \B^{{d}/{p} - 1}_{p,1} (\Omega ; \IR^d)) \cap \L^1(0,T; \B^{{d}/{p}}_{p,1} (\Omega) \times \B^{d/p-1}_{p,1} (\Omega ; \IR^d)). $$
\end{proposition}
\begin{proof}
Throughout, we use the short notation ${\bf L}$ for ${\bf L}_{p,1,d/p-1}.$
Since the variations of density are small, 
one can look at  the velocity equation as follows: 
$$\displaylines{\quad
\partial_t\bar u+ {\bf L}\bar u=-a_0\partial_t \bar u
+2\divergence\bigl(\wt\mu(\bar a)\adj(DX_{\bar u})\cdot D_{A_{\bar u}}(\bar u)-\bar\mu D(\bar u)\bigr)\\
\hfill\cr\hfill+\nabla\bigl((\wt\lambda(\bar a)\divergence_{A_{\bar u}}\bar u-\bar\lambda\divergence\bar u)\bigr)
-{}^T\!\adj(DX_{\bar u})\cdot\nabla(P(1+\bar a))\quad}$$
with $\bar a$ given by 
$$\bar a= J^{-1}_\bu\rho_0-1= (J^{-1}_\bu-1)(1+a_0) + a_0.$$
To proceed, we introduce for $T>0,$ the space
$$\F_p(T):=\W^{1 , 1}(0,T;\B^{{d}/{p} - 1}_{p,1} (\Omega ; \IR^d)) \cap \L^1(0,T;\B^{d/p-1}_{p,1} (\Omega ; \IR^d))\cdotp$$
We consider the map 
$\Psi: \bv\mapsto \bu$ where  $\bu$ is the solution to 
$$\partial_t\bar u+ {\bf L}\bar u= \bh\quad\hbox{in }\ (0,T)\times\Omega\andf
u|_{t=0}=u_0\ \hbox{ in }\ \Omega,$$
with $\bh=\bh^1+\bh^2+\bh^3+\bh^4$ and 
$$\begin{aligned}
&\bh^1:=-a_0\partial_t \bv,\quad &&\bh^2:=2\divergence\bigl(\wt\mu(\bb)\adj(DX_\bv)\cdot D_{A_\bv}(\bv)-\bar\mu D(\bv)\bigr)\\
&\bh^3:=\nabla\bigl((\wt\lambda(\bb)\divergence_{A_{\bv}}\bv-\bar\lambda\divergence\bv)\bigr),\quad
&&\bh^4:=-{}^T\!\adj(DX_\bv)\cdot\nabla(P(1+\bb))\cdotp\end{aligned}$$
Above, the function $\bb$ is defined by 
\begin{equation}\label{def:bb}\bb= J^{-1}_\bv\rho_0-1= (J^{-1}_\bv-1)(1+a_0) + a_0.\end{equation}
We claim that there exists $\alpha>0$ in~\eqref{eq:smalldata2} such that 
for small enough  $R,T>0,$ the function $\Psi$ is a self-map on   $\bar B_{\F_p(T)}(u_L ,R),$ where $u_L:=\e^{-t{\bf L}}u_0$.
To justify  our claim,  we set $\wt v:=\bv-u_L$ and
look for $\bu$ under the form  $\bu:=u_L+\wt u$ with $\wt u$ satisfying
   $$\partial_t\wt u+ {\bf L}\wt u= \bh\quad\hbox{in }\ (0,T)\times\Omega\andf \wt u|_{t=0}=0\ \hbox{ in }\ \Omega.$$
Consequently,  Corollary~\ref{c:lame} yields some $C>0$ independent of $T>0$ such that
\begin{equation}\label{eq:wtu}
\|\wt u\|_{\F_p(T)} \leq C\|\bh\|_{\L^1(0,T;\B^{d/p-1}_{p,1})}
\andf \|u_L\|_{\F_p(T)} \leq C\|u_0\|_{\B^{d/p-1}_{p,1}}.\end{equation}
By Lebesgue's dominated convergence theorem, $\|\nabla u_L\|_{\L^1(0,T;\B^{d/p}_{p,1})}$
converges to $0$ as $T\to0.$ Hence, for any $R>0,$ one can find $T>0$ so that
\begin{equation}\label{eq:DuL}
\int_0^T\|\nabla u_L\|_{\B^{d/p}_{p,1}}\,\d t\leq \frac{R}{2}\cdotp
\end{equation}
Next, we have to bound $\bh^1$ to $\bh^4$ in $\L^1(0,T;\B^{d/p-1}_{p,1}).$
We shall use repeatedly Proposition~\ref{p:product} with $s\in\{d/p,d/p-1\}$ and Proposition
\ref{p:compo}, as well as the local-in-time version of~\eqref{eq:J}. 
First, it is obvious that 
$$\begin{aligned}
\|\bh^1\|_{\L^1(0,T;\B^{d/p-1}_{p,1})}&\lesssim
\|a_0\|_{\B^{d/p}_{p,1}}\|\partial_t\bv\|_{\L^1(0,T;\B^{d/p-1}_{p,1})}\\
&\lesssim \alpha R.\end{aligned}$$
In order to bound the next terms, we shall use the fact that, 
owing to the decomposition of $\bb$ in~\eqref{def:bb}, 
the product and composition results in Proposition~\ref{p:product} and~\ref{p:compo},
and the local-in-time version of~\eqref{eq:J}, we have
for all smooth functions $k$ vanishing at $0$ and $t\in[0,T],$ 
$$\begin{aligned}
\|k(\bb(t))\|_{\B^{d/p}_{p,1}}&\lesssim \|\bb(t)\|_{\B^{d/p}_{p,1}}\\
&\lesssim \|a_0\|_{\B^{d/p}_{p,1}}
+(1+\|a_0\|_{\B^{d/p}_{p,1}}) \|J^{-1}_\bv(t)-1\|_{\B^{d/p}_{p,1}}\\
&\lesssim \|a_0\|_{\B^{d/p}_{p,1}}
+(1+\|a_0\|_{\B^{d/p}_{p,1}}) \int_0^t\|\nabla\bv\|_{\B^{d/p}_{p,1}}\,\d\tau\\
&\lesssim \alpha+R.\end{aligned}$$
To bound $\bh^2,$ we use the decomposition
$$\displaylines{\wt\mu(\bb)\adj(DX_\bv)\cdot D_{A_\bv}(\bv)-\bar\mu D(\bv)
=(\wt\mu(\bb)-\bar\mu)\adj(DX_\bv)\cdot D_{A_\bv}(\bv)\hfill\cr\hfill
+\bar\mu(\adj(DX_\bv)-\Id)\cdot D_{A_\bv}(\bv)
+\bar\mu(D_{A_\bv}(\bv)- D(\bv)).}$$
Hence, using the aforementioned results and also~\eqref{eq:deriv}, we find that for all $t\in[0,T],$
$$\displaylines{\quad\|\bh^2\|_{\B^{d/p-1}_{p,1}}\lesssim
\|\bb\|_{\B^{d/p}_{p,1}}\|\adj(DX_\bv)\cdot D_{A_\bv}(\bv)\|_{\B^{d/p}_{p,1}}
\hfill\cr\hfill+\|(\adj(DX_\bv)-\Id)\cdot D_{A_\bv}(\bv)\|_{\B^{d/p}_{p,1}}+\|D_{A_\bv}(\bv)- D(\bv)\|_{\B^{d/p}_{p,1}},\quad}$$
whence we have
$$\begin{aligned}\|\bh^2\|_{\L^1(0,T;\B^{d/p-1}_{p,1})}&\lesssim 
\|\bb\|_{\L^\infty(0,T;\B^{d/p}_{p,1})}\|\nabla\bv\|_{\L^1(0,T;\B^{d/p}_{p,1})}
+\|\nabla\bv\|_{\L^1(0,T;\B^{d/p}_{p,1})}^2\\
&\lesssim R\,(\alpha + R).
\end{aligned}$$
Bounding $\bh^3$ is clearly the same. Finally,
 to handle  $\bh^4$ (that is, the pressure term), we assume with no loss of 
generality that  $P(1)=0,$ and use the decomposition
$$\bh^4=(\Id-{}^T\!\adj(DX_\bv))\cdot \nabla(P(1+\bb))
- \nabla(P(1+\bb)).$$
Hence 
$$\begin{aligned}\|\bh^4\|_{\L^1(0,T;\B^{d/p-1}_{p,1})}
&\lesssim\bigl(1+\|\Id-{}^T\!\adj(DX_\bv)\|_{\L^\infty(0,T;\B^{d/p}_{p,1})}\bigr)
\|P(1+\bb)\|_{\L^1(0,T;\B^{d/p}_{p,1})}
\\
&\lesssim \bigl(1+\|\nabla\bv\|_{\L^1(0,T;\B^{d/p}_{p,1})}\bigr)
\|\bb\|_{\L^1(0,T;\B^{d/p}_{p,1})}\\
&\lesssim T (\alpha +R).\end{aligned}$$
Reverting to~\eqref{eq:wtu}, we end up with
$$\|\wt u\|_{\F_p(T)}\leq C(\alpha+R)(T+R).$$
Consequently, if one takes $R=\alpha$ and assumes, in addition to~\eqref{eq:DuL}, 
that $T\leq\alpha,$ we obtain $$\|\wt u\|_{\F_p(T)}\leq 4C\alpha^2.$$
One can thus conclude that $\Psi$  is a self-map on $\bar B_{\F_p(T)}(u_L,R)$ provided
$8C\alpha\leq1.$ 
\medbreak
To complete the proof of existence of a fixed point for $\Psi,$ one has to
exhibit  its properties of contraction. Consider $\bv_i \in \bar B_{\F_p(T)}(u_L,R)$ and $\bu_i:=\Psi \bv_i,$ $i=1,2,$ with $R$ and $T$ as above.
Then, according to Corollary~\ref{c:lame}, we have
$$
\|\du\|_{\F_p(T)} \lesssim  \|\dh\|_{\L^1(0,T;\B^{d/p-1}_{p,1})},$$
where $\du:=\bu_2-\bu_1$, and so on. We see
that $\du$ fulfills (where $\dg^2$ and $\dg^3$ have been defined 
just above~\eqref{eq:dg1}): 
\begin{multline}\label{eq:dul}\partial_t\du+{\bf L}\du= 
-a_0\partial_t \dv+\dh^2+\dh^3\\
-{}^T\!\adj(DX_1)\cdot\nabla(P(1+\bb_2)-P(1+\bb_1))-{}^T\!(\adj(DX_2)-\adj(DX_1))\cdot\nabla(P(1+\bb_2)).
\end{multline}
Then, one has to perform  always the same type of computations as just above and 
in the previous section. The details are  omitted. One ends up with 
$$\|\du\|_{\F_p (T)} \leq CR\|\dv\|_{\F_p (T)},$$
which, provided $CR<1,$ allows to complete the proof of a fixed point
for $\Psi,$ and thus of a solution for~\eqref{eq:CNSlag}, in the desired regularity space. 

Proving uniqueness is similar as for  the global existence theorem, 
except that we now use~\eqref{eq:dul} with $\bv=\bu$ instead of the full system for $(\ba,\bu).$
In particular, there is no need to assume that the velocity of one of the solutions is small. 
Again, the details are left to the reader. \end{proof}


\subsection{The case of large  variations of density} 

This part is devoted to the proof of Theorem~\ref{Thm:local} in full generality. 
The main issue is  to adapt Corollary~\ref{c:lame} to the following system:
\begin{equation}\label{eq:lamebis}\left\{\begin{aligned} 
\rho\partial_tu-2\divergence(\mu D(u))-\nabla(\lambda \divergence u)= f&\quad\hbox{in }\ (0,T)\times\Omega,\\
u|_{\partial\Omega}=0&\quad\hbox{on }\ (0,T)\times\partial\Omega,\\
u|_{t=0}=u_0&\quad\hbox{in }\ \Omega,\end{aligned}\right.\end{equation}
where $\rho=\rho(x),$ $\lambda=\lambda(x)$, and $\mu=\mu(x)$ 
are given functions in $\B^{d/p}_{p,1}(\Omega),$ 
such that 
\begin{equation}\label{eq:ellipticity}
\inf_{x\in\Omega}\rho(x)>0,\quad \inf_{x\in\Omega}\mu(x)>0,\andf \inf_{x\in\Omega} (\lambda+2\mu)(x)>0.\end{equation}
\begin{proposition}\label{c:lamebis} Let $T>0.$  Let $1<p<\infty$ and $-1+1/p<s<1/p$
with $s\leq d/p-1.$ 
Take $u_0$ in $\B^s_{p,1}(\Omega;\IR^d)$  and $f$  in $\L^1(0,T;\B^s_{p,1}(\Omega;\IR^d)).$ Assuming~\eqref{eq:ellipticity}, 
System~\eqref{eq:lamebis}  admits 
 a unique solution $u\in\cC_b([0,T];\B^s_{p,1}(\Omega;\IR^d))$ in the space
 \begin{align*}
  u \in \W^{1 , 1} (0 , T ; \B^s_{p , 1} (\Omega ; \IR^d)) \cap \L^1(0,T;\B^{s+2}_{p,1}(\Omega;\IR^d))
 \end{align*}
and there exists a constant $C > 0$ depending only on $\rho,$ $\lambda,$ $\mu,$ 
$p,$ $s$ and $\Omega,$
such that 
\begin{equation}\label{eq:lamebis1}\sup_{t\in[0,T]} \|u(t)\|_{\B^s_{p,1}}+\int_0^T \bigl(\|\partial_tu\|_{\B^s_{p,1}}
+\|u\|_{\B^{s+2}_{p,1}}\bigr)\, \d t\leq 
C\biggl( \|u_0\|_{\B^s_{p,1}}+\int_0^T\|f\|_{\B^s_{p,1}}\,\d t\biggr)\cdotp\end{equation}\end{proposition}
\begin{proof}
The key idea is that  the embedding $\B^{d/p}_{p,1}(\Omega)\hookrightarrow\cC(\overline{\Omega})$
implies that the coefficients
of System~\eqref{eq:lamebis} are uniformly continuous on $\Omega,$ hence have
small variations on small  balls, so that one can take advantage of
Corollary~\ref{c:lame}, after localization of the system. 
\medbreak
To start with, as in~\cite{Danchin}, we introduce 
a covering $(B_k)_{1\leq k\leq K}$ of $\overline\Omega$ by balls of radius $\delta \in (0 , 1)$
and center $x_k \in \Omega,$ 
with finite multiplicity (independent of $\delta$),
and a partition of unity $(\phi_k)_{1\leq k\leq K}$ of smooth functions on $\IR^d$ such that:
\begin{itemize}
\item $\sum_{k=1}^K\phi_k\equiv1$ in $\Omega$;
\smallbreak
\item $\|\nabla^\alpha\phi_k\|_{\L^\infty(\IR^d)}\leq C_\alpha \delta^{-\alpha},\ \alpha\in\IN$;
\smallbreak\item  the support of $\phi_k$ is included in $B_k.$ 
\end{itemize}
This covering may be constructed from a smooth  function $\theta$ supported in the unit ball, such that
$$\sum_{k\in\IZ^d} \theta(x-k)=1\quad\hbox{on}\quad \IR^d.$$
It is just a matter of setting $\phi_k(x):=\theta(\delta^{-1}(x-\delta k))$ with $x_k=\delta k,$
then relabelling the family $(\phi_k),$ keeping only indices for which  ${\rm Supp}\, \phi_k\cap\Omega$ is nonempty. 
Clearly,  combining the bounds of $\nabla^\alpha\phi_k$ with the fact that 
 ${\rm Supp}\, \phi_k\subset B_k$ ensures that
 $$ \|\nabla^\alpha\phi_k\|_{\L^p(\IR^d)}\leq C'_\alpha \delta^{\frac dp-\alpha},\quad\alpha\in\IN$$ and thus,  by interpolation, 
 \begin{equation}\label{eq:boundphik}
 \|\phi_k\|_{\B^{d/p}_{p,1}(\IR^d)}\leq C \andf \|\nabla\phi_k\|_{\B^{d/p}_{p,1}(\IR^d)}\leq C\delta^{-1}. 
\end{equation}
We also need another two families  $(\check\phi_k)_{1\leq k\leq K}$ and   $(\wt\phi_k)_{1\leq k\leq K}$
such that  $\check\phi_k\equiv 1$ on 
 the support of $\phi_k$ and   $\wt\phi_k\equiv 1$ on 
 the support of $\check\phi_k,$  with $\check\phi_k$ and $\wt\phi_k$  supported in slightly larger balls  than $\phi_k,$
 and such that $\|\nabla^\alpha\check\phi_k\|_{\L^\infty}\leq C_\alpha \delta^{-\alpha}$
  and $\|\nabla^\alpha\wt\phi_k\|_{\L^\infty}\leq C_\alpha \delta^{-\alpha}$  hold.
\medbreak
Let   $\rho_k:=\rho(x_k),$ $u_k:=u\phi_k,$ $f_k= \rho_k f ,$
 $\lambda_k=\lambda(x_k)$, and $\mu_k=\mu(x_k).$ 
Then, we observe that $u_k$ satisfies:
\begin{equation}\label{eq:lamek}\left\{\begin{aligned} 
\rho_k\partial_tu_k- \mu_k\Delta u_k -(\lambda_k+\mu_k)\nabla \divergence u_k= F_k
&\quad\hbox{in }\ (0,T)\times\Omega,\\
u_k|_{\partial\Omega}=0&\quad\hbox{on }\ (0,T)\times\partial\Omega,\\
u_k|_{t=0}=u_{0,k}&\quad\hbox{in }\ \Omega,\end{aligned}\right.\end{equation}
with $u_{k , 0} := u_0 \phi_k$ and
$$\displaylines{F_k:= f_k+(\rho_k-\rho)\partial_t u_k+
2\divergence\bigl(\phi_k(\mu-\mu_k) D(u)\bigr) + \nabla\bigl(\phi_k(\lambda-\lambda_k)\divergence u\bigr)
\hfill\cr\hfill-2\mu D(u)\cdot\nabla\phi_k -\lambda \divergence u \,\nabla\phi_k
-\mu_k\divergence\bigl(u\otimes\nabla\phi_k +\nabla\phi_k\otimes u\bigr)
-\lambda_k\nabla(u\cdot\nabla\phi_k).}$$
Therefore, in light of Corollary~\ref{c:lame} and denoting $\wt\mu_k:=\mu_k/\rho_k,$ we have for all $t\in[0,T],$
\begin{equation}\label{eq:V1}
\|u_k(t)\|_{\B^{s}_{p,1}} +\int_0^t \bigl(\|\partial_t  u_k\|_{\B^{s}_{p,1}}
+ \wt\mu_k\|u_k\|_{\B^{s+2}_{p,1}}\bigr)\,\d\tau
\leq C\biggl(\|u_k(0)\|_{\B^{s}_{p,1}} +\rho_k^{-1}\int_0^t \|F_k\|_{\B^{s}_{p,1}}\,\d\tau\biggr)\cdotp
\end{equation}
Note that our ellipticity condition \eqref{eq:ellipticity} ensures that 
 the constant $C$  is independent of $k.$
\medbreak
Throughout, we fix some $\eps>0$ and take $\delta$ so that
for all $k\in\{1,\cdots,K\},$ 
\begin{equation}\label{eq:unifcont}
\max\bigl(\|1 - \rho/\rho_k\|_{\L^\infty(B_k)}, 
\mu_k^{-1}\|\mu-\mu_k\|_{\L^\infty(B_k)}, \mu_k^{-1}\|\lambda-\lambda_k\|_{\L^\infty(B_k)}\bigr)
\leq\eps.
\end{equation}
Actually, as we  have to perform estimates in Besov spaces, we need a stronger property, namely
\begin{equation}\label{eq:unifB}
\max\bigl(\|\wt\phi_k(1 - \rho/\rho_k)\|_{\B^{d/p}_{p,1}(\Omega)}, 
\mu_k^{-1}\|\wt\phi_k(\mu-\mu_k)\|_{\B^{d/p}_{p,1}(\Omega)},
 \mu_k^{-1}\|\wt\phi_k(\lambda-\lambda_k)\|_{\B^{d/p}_{p,1}(\Omega)}\bigr)
\leq\eps,
\end{equation}
which is proved at the end of the Appendix
\medbreak
Let us now estimate all the terms of $F_k.$ We have thanks to Proposition~\ref{p:product} and~\eqref{eq:unifB},
$$\begin{aligned}
\|(\rho_k-\rho)\partial_tu_k\|_{\B^{s}_{p,1}(\Omega)}&\leq C
\|\wt\phi_k(\rho_k-\rho)\|_{\B^{d/p}_{p,1}(\Omega)}
\|\partial_tu_k\|_{\B^{s}_{p,1}(\Omega)}\\&\leq C\eps\rho_k \|\partial_tu_k\|_{\B^{s}_{p,1}(\Omega)},
\end{aligned}
$$
and, using also~\eqref{eq:boundphik},  with the notation $\wt u_k:=\wt\phi_k u,$ 
$$
\begin{aligned}
\|\divergence\bigl(\phi_k(\mu-\mu_k) D(u)\bigr)\|_{\B^{s}_{p,1}(\Omega)}&\lesssim
\|\wt\phi_k(\mu-\mu_k)\|_{\B^{d/p}_{p,1}(\Omega)}
\|\phi_k\nabla u\|_{\B^{s+1}_{p,1}(\Omega)}\\
&\leq C\eps\mu_k\bigl( \|\nabla (\phi_k  u)\|_{\B^{s+1}_{p,1}(\Omega)} +  \| \nabla \phi_k\otimes \wt\phi_ku\|_{\B^{s+1}_{p,1}(\Omega)}\bigr)\\
&\leq C\eps\mu_k\bigl( \|u_k \|_{\B^{s+2}_{p,1}(\Omega)} +  \delta^{-1}\| \wt u_k\|_{\B^{s+1}_{p,1}(\Omega)}\bigr)\cdotp
\end{aligned}$$
The next term may be estimated in the same way.
 In order to estimate the term $\mu D(u)\cdot\nabla\phi_k,$ let us set  $\check u_k:=\check\phi_k u.$
 Applying Proposition~\ref{p:product} and  \eqref{eq:boundphik} yields
$$\begin{aligned}
\|\mu D(u)\cdot\nabla\phi_k\|_{\B^{s}_{p,1}(\Omega)}&\lesssim 
\|\mu\|_{\B^{d/p}_{p,1}(\Omega)} \|D(u)\cdot\nabla\phi_k\|_{\B^{s}_{p,1}(\Omega)}\\
&\lesssim \|\mu\|_{\B^{d/p}_{p,1}(\Omega)}\|\nabla\phi_k\|_{\B^{d/p}_{p,1}(\Omega)}
 \|\check\phi_k D(u)\|_{\B^{s}_{p,1}(\Omega)}\\
 &\lesssim \delta^{-1}\|\mu\|_{\B^{d/p}_{p,1}(\Omega)}
\bigl(\|\nabla(\check\phi_k u)\|_{\B^{s}_{p,1}(\Omega)}+    \|\wt\phi_k u\otimes \nabla \check\phi_k \|_{\B^{s}_{p,1}(\Omega)}\bigr)\\
&\lesssim  \delta^{-1}\|\mu\|_{\B^{d/p}_{p,1}(\Omega)}\bigl(\|\check u_k\|_{\B^{s+1}_{p,1}(\Omega)}+\delta^{-1}
\|\wt u_k\|_{\B^{s}_{p,1}(\Omega)}\bigr)\cdotp\end{aligned}$$
A similar estimate holds for $\lambda \divergence u \,\nabla\phi_k.$
Finally,
$$\begin{aligned}
\|\nabla(u\cdot\nabla\phi_k)\|_{\B^{s}_{p,1}(\Omega)}&\lesssim 
\|u\wt\phi_k\cdot\nabla\phi_k\|_{\B^{s+1}_{p,1}(\Omega)}\\
&\lesssim  \|\nabla\phi_k\|_{\B^{d/p}_{p,1}(\Omega)}
\| \wt u_k\|_{\B^{s+1}_{p,1}(\Omega)}\\
&\lesssim \delta^{-1}\|\wt u_k\|_{\B^{s+1}_{p,1}(\Omega)},
\end{aligned}$$
and the same holds for $\divergence\bigl(u\otimes\nabla\phi_k +\nabla\phi_k\otimes u\bigr)\cdotp$
 \medbreak
 Let us denote $\zeta^*:= 1+\|\lambda/\mu\|_{\L^\infty}.$ Then, atogether, reverting to~\eqref{eq:V1} and 
 assuming that $\eps$ has been chosen small enough (so 
 as to absorb the terms with  $\partial_tu_k$ and $\|u_k\|_{\B^{s+2}_{p,1}(\Omega)}$),  we end up  
  for all $k\in\{1,\cdots,K\}$ with
 \begin{multline}\label{eq:V2}
\|u_k(t)\|_{\B^{s}_{p,1}} +\int_0^t \bigl(\|\partial_t  u_k\|_{\B^{s}_{p,1}}
+ \wt\mu_k\|u_k\|_{\B^{s+2}_{p,1}}\bigr)\,\d\tau
\leq C\biggl(\|u_k(0)\|_{\B^{s}_{p,1}} +\int_0^t\|f_k\|_{\B^s_{p,1}}\,\d\tau\\ +
\wt\mu_k\delta^{-1}
\int_0^t \bigl(\zeta^*\|\wt u_k\|_{\B^{s+1}_{p,1}}+\mu_k^{-1}\|(\lambda,\mu)\|_{\B^{d/p}_{p,1}}\|\check u_k\|_{\B^{s+1}_{p,1}}
\bigr)\d\tau +  \rho_k^{-1}\delta^{-2}\int_0^t\|(\lambda,\mu)\|_{\B^{d/p}_{p,1}}\|\wt  u_k\|_{\B^{s}_{p,1}}\,\d\tau\biggr)\cdotp
\end{multline}
Let us introduce the notation:
$$\|z\|_{\B^{s,\psi}_{p,1}}:= \sum_{k=1}^K \|\psi_k z\|_{\B^{s}_{p,1}(\Omega)}\quad\hbox{for }\ \psi\in\{\phi,\check\phi,\wt\phi\}.$$
Then, summing up on $k\in\{1,\cdots,K\}$ in~\eqref{eq:V2}
and denoting $\wt\mu_*:=\inf_\Omega \mu/\rho,$ 
$\wt\mu^*:=\sup_\Omega \mu/\rho$ and $\rho_* :=\inf_\Omega \rho,$ we conclude that
 \begin{multline}\label{eq:V3}\|u(t)\|_{\B^{s,\phi}_{p,1}} +\int_0^t \bigl(\|\partial_t  u\|_{\B^{s,\phi}_{p,1}}
+ \wt\mu_*\|u\|_{\B^{s+2,\phi}_{p,1}}\bigr)\,\d\tau
\leq C\biggl(\|u_0\|_{\B^{s,\phi}_{p,1}} +\int_0^t\|f\|_{\B^{s,\phi}_{p,1}}\,\d\tau\\
+\delta^{-1}\wt\mu^*\zeta^*\int_0^t\|u\|_{\B^{s+1,\wt\phi}_{p,1}}\,\d\tau
+\delta^{-1}\rho_*^{-1}\int_0^t\|(\lambda,\mu)\|_{\B^{d/p}_{p,1}}
\bigl(\|u\|_{\B^{s+1,\check\phi}_{p,1}}
+\delta^{-1}\|u\|_{\B^{s,\wt\phi}_{p,1}}\bigr)\,\d\tau\biggr)\cdotp
\end{multline}
Since the properties of the support of the families $(\wt\phi_k)$ and $(\check\phi_k)$ guarantee that
$$\wt\phi_k=\sum_{k'\sim k} \wt\phi_k \phi_{k'}\andf \check\phi_k=\sum_{k'\sim k} \check\phi_k \phi_{k'},$$
we may write   for all  $-\min(d/p,d/p')<\sigma\leq d/p,$ owsing to Proposition \ref{p:product}
and  Inequality \eqref{eq:boundphik}, 
$$\|\wt u_k\|_{\B^\sigma_{p,1}}\leq  C \sum_{k'\sim k} \|\wt\phi_k\|_{\B^{d/p}_{p,1}}\| u_{k'}\|_{\B^\sigma_{p,1}}\leq  C
 \sum_{k'\sim k} \| u_{k'}\|_{\B^\sigma_{p,1}}.$$
 A similar property is true for $\check u_k.$ Hence 
 $$ \|u\|_{\B^{\sigma,\wt\phi}_{p,1}}\lesssim \|u\|_{\B^{\sigma,\phi}_{p,1}}\andf \|u\|_{\B^{\sigma,\check\phi}_{p,1}}\lesssim \|u\|_{\B^{\sigma,\phi}_{p,1}}.$$
 This means that $\wt\phi$ and $\check\phi$ may be replaced by $\phi$ in the right-hand side of \eqref{eq:V3}
 (up to a change of $C$ of course). 
 Now, the terms of \eqref{eq:V3} involving the index $s+1$ may be bounded by interpolation as follows  for all $A>0$ and $\eps>0$ :
 $$\begin{aligned}
 A\|u\|_{\B^{s+1,\phi}_{p,1}}&\leq C\sum_k A\|u_k\|_{\B^s_{p,1}}^{1/2}\|u_k\|_{\B^{s+2}_{p,1}}^{1/2}\\
 &\leq \eps\wt\mu_*\sum_k\|u_k\|_{\B^{s+2}_{p,1}} + C\eps^{-1}\wt\mu_*^{-1}A^2\sum_k \|u_k\|_{\B^s_{p,1}}\\
 &=\eps\wt\mu_*\|u\|_{\B^{s+2,\phi}_{p,1}} + C\eps^{-1}\wt\mu_*^{-1}A^2\|u\|_{\B^{s,\phi}_{p,1}},
 \end{aligned}$$
with $C$ independent of $A$ and $\eps.$
Hence, taking either $A=C\delta^{-1}\wt\mu^*\zeta^*$ 
or $A=C\delta^{-1}\rho_*^{-1}\|(\lambda,\mu)\|_{\B^{d/p}_{p,1}},$
 Inequality \eqref{eq:V3} entails (observing that the last term of it
 can be dominated by the other ones resulting from 
 the computations just above),  
$$\displaylines{
\|u(t)\|_{\B^{s,\phi}_{p,1}} +\int_0^t \bigl(\|\partial_t  u\|_{\B^{s,\phi}_{p,1}}
+ \wt\mu_*\|u\|_{\B^{s+2,\phi}_{p,1}}\bigr)\,\d\tau
\leq C\biggl(\|u_0\|_{\B^{s,\phi}_{p,1}} +\int_0^t\|f\|_{\B^{s,\phi}_{p,1}}\,\d\tau\hfill\cr\hfill
+\rho_*^{-1}\wt \mu_*^{-1}\delta^{-2}\bigl(\rho_*(\zeta^*)^2 (\wt\mu^*)^2
+\rho_*^{-1}\|(\lambda,\mu)\|_{\B^{d/p}_{p,1}}^2\bigr)
\int_0^t\|u\|_{\B^{s,\phi}_{p,1}}\,\d\tau\biggr)\cdotp}$$
Since,  $\rho_*\wt\mu^*\leq \mu^*,$ and thus $\rho_*\wt\mu^*\lesssim \|\mu\|_{\B^{d/p}_{p,1}},$  applying Gronwall lemma eventually leads to 
 \begin{multline}\label{eq:V4}\|u(t)\|_{\B^{s,\phi}_{p,1}} +\int_0^t \bigl(\|\partial_t  u\|_{\B^{s,\phi}_{p,1}}
+ \wt\mu_*\|u\|_{\B^{s+2,\phi}_{p,1}}\bigr)\,\d\tau\\
\leq C\biggl(\|u_0\|_{\B^{s,\phi}_{p,1}} +\int_0^t \|f\|_{\B^{s,\phi}_{p,1}}\,\d\tau\biggr)
\exp\Bigl(C\wt\mu_*^{-1}\delta^{-2}\rho_*^{-2}(\zeta^*)^2 \|(\lambda,\mu)\|_{\B^{d/p}_{p,1}}^2t\Bigr)\cdotp\end{multline}
Since the covering is finite,  the norms $\|\cdot\|_{\B^{\sigma,\phi}_{p,1}}$ are actually equivalent to the
Besov norms  $\|\cdot\|_{\B^{\sigma}_{p,1}(\Omega)}$ (with bounds depending on $K$ of course), which eventually ensures
the desired inequality~\eqref{eq:lamebis1}. 
\medbreak
In order to prove the existence of a solution to~\eqref{eq:lamebis}
in the space $\F^s_p(T)$ corresponding to the statement of Proposition~\ref{c:lamebis}, 
one may adapt the  continuity method used in~\cite[Thm.~2.2]{Danchin}. 

For all $\theta\in[0,1],$ we define the linear operator $\cL_\theta$ acting on time-dependent vector fields $u$ by:
$$
\cL_\theta u:=\rho_\theta\partial_tu-2\divergence(\mu_\theta D(u))-\nabla(\lambda_\theta\divergence u) $$
with $\rho_\theta:=(1-\theta) +\theta\rho,$ $\mu_\theta:=1-\theta + \theta \mu$ and $\lambda_\theta:=\theta\lambda.$
Note that the ellipticity condition \eqref{eq:ellipticity} is ensured uniformly for $\theta\in[0,1]$ and that 
the value of $\delta$ and of $C$ may be chosen independent of $\theta$ in Inequality \eqref{eq:V4}
(hence Inequality \eqref{eq:lamebis1}  corresponding to System \eqref{eq:lamebis} 
with coefficients $\rho_\theta,$ $\mu_\theta$ and $\lambda_\theta$ is uniform with respect to $\theta$ as well).  

We denote by $\cE$ the set of parameters $\theta\in[0,1]$ such that  for all data $u_0$ and $f$ satisfying the hypotheses
of Proposition \ref{c:lamebis},  System \eqref{eq:lamebis} 
with coefficients $\rho_\theta,$ $\mu_\theta$ and $\lambda_\theta$ has  a solution in $\F^s_p(T).$
Corollary \ref{c:lame} guarantees that  $0$ is in $\cE.$ 
Now consider any $\theta_0\in \cE$ and data $u_0,$ $f.$  
Solving 
$$\cL_\theta u=f,\qquad u|_{\partial\Omega}=0,\qquad u|_{t=0}=u_0$$  in $\F_p^s(T)$  amounts to finding 
a fixed point  in $\F_p^s(T)$ for the map $\Phi: v\mapsto u$ such that $u$ is a solution in  $\F_p^s(T)$ of 
\begin{equation}\label{eq:cL}\cL_{\theta_0} u = f+(\cL_{\theta_0}-\cL_\theta)v,\qquad u|_{\partial\Omega}=0,\qquad u|_{t=0}=u_0.\end{equation}
Obviously, we have
$$(\cL_{\theta_0}-\cL_\theta)v=(\theta-\theta_0)\Bigl( (1-\rho)\partial_tv
+2\divergence\bigl((\mu-1) D(v)\bigr)+\divergence\bigl(\lambda \divergence v\bigr)\Bigr)\cdotp$$
Hence, using Proposition \ref{p:product} eventually leads to 
$$ \|(\cL_{\theta_0}-\cL_\theta)v \|_{\B^s_{p,1}}\leq C|\theta-\theta_0|\bigl(\|\partial_tv\|_{\B^s_{p,1}}+\|v\|_{\B^{s+2}_{p,1}}\bigr)\cdotp$$
The constant $C$ depends of course of $\rho,$ $\lambda$ and $\mu,$ but is independent of $\theta$ and $\theta_0.$ 
Now, since $\theta_0\in\cE,$   equation \eqref{eq:cL} is solvable in $\F_p^s(T)$ and
estimate \eqref{eq:lamebis1} combined with the above computation gives us  
$$\begin{aligned}
\|\Phi(v)\|_{\F^s_p(T)} &\leq C\biggl(\|u_0\|_{\B^s_{p,1}}+\int_0^T \|f\|_{\B^s_{p,1}}\,dt+\int_0^T \|(\cL_{\theta_0}-\cL_\theta)v \|_{\B^s_{p,1}}\,\d t\biggr)\\
&\leq C\bigl(\|u_0\|_{\B^s_{p,1}}+|\theta-\theta_0| \|v\|_{\F^s_p(T)}\bigr)\cdotp\end{aligned}$$
The same computation leads, for all pair $(v_1,v_2)$ in $\F^s_p(T)$ to 
$$\|\Phi(v_2)-\Phi(v_1)\|_{\F^s_p(T)}  \leq C\,|\theta-\theta_0| \, \|v_2-v_1\|_{\F^s_p(T)}.$$
Hence, setting $\eps=1/2C,$ one can conclude by the contracting mapping argument that
$\Phi$ admits a fixed  point $u$ in $\F^s_p(T)$  whenever $|\theta-\theta_0|\leq\eps.$ 
Since $\eps$ is independent of $\theta_0,$ we deduce that $1$ is in the set $\cE,$ which completes the proof of existence.\end{proof}

\begin{proof}[Proof of Theorem ~\ref{Thm:local}] As in the previous parts, we shall rather prove
the result in Lagrangian coordinates.
Having  Proposition~\ref{c:lamebis} at hand, it  suffices to  modify the fixed 
point map $\Psi$ introduced a couple of pages ago accordingly. 
More precisely, we observe that we want the Lagrangian velocity $\bu$ to satisfy
$$\left\{\begin{aligned}
&{\bf L}_{\rho_0}\bu
= 2\divergence\bigl(\mu(\bar\rho_\bu)\adj(DX_\bu)\cdot D_{A_\bu}(\bu)-\mu_0 D(\bu)\bigr)
+\nabla\bigl(\lambda(\rho_\bu)\divergence_{A_\bu}\bu-\lambda_0\divergence\bu\bigr)\\&\hspace{10cm}-{}^T\!\adj(DX_\bu)\cdot\nabla(P(\bar\rho_\bu)),\\
&\bu|_{\partial\Omega}=0,\\
&\bu|_{t=0}=u_0\end{aligned}\right.
$$
with  $\lambda_0:=\lambda(\rho_0),$ $\mu_0:=\mu(\rho_0),$ ${\bf L}_{\rho_0}\bu:=\rho_0\partial_t\bu -2\divergence(\mu_0 D(\bu))-\nabla(\lambda_0\divergence\bu)$ and  $\bar\rho_\bu:=\rho_0 J^{-1}_\bu.$
\medbreak
Define $\Psi: \F_p(T)\to\F_p(T)$ to be the map $\bv\mapsto \bu$ with $\bu$ the solution in $\F_p(T)$ 
provided by Proposition~\ref{c:lamebis} that corresponds to the right-hand side of the above system 
with $\bv$ instead of $\bu.$ 
Denote by $u_L^{\rho_0}$ the solution to ${\bf L}_{\rho_0}u=0$ with initial data $u_0$ given  by Proposition~\ref{c:lamebis}. 

Then, by following the proof of Proposition~\ref{p:local}, it is not difficult to check that $\Psi$ satisfies the 
conditions of the contraction mapping theorem  on some ball $\bar B_{\F_p(T)}(u_L^{\rho_0},R)$ 
provided $R$ and $T$ are small enough.  In fact, the main changes are that 
the term corresponding to $\bar h^1$ is no longer present (hence we do not need to assume $\rho_0$ to be close to some constant)
and that one has to bound in $\B^{d/p}_{p,1}$ terms like $\mu(\bar\rho_\bv)-\mu_0.$ However, 
owing to Propositions \ref{p:product} and  \ref{p:compo}, and to Inequality \eqref{eq:J}, we may write
$$\begin{aligned}
\| \mu(\bar\rho_\bv(t))-\mu_0\|_{\B^{d/p}_{p,1}}&\lesssim \|\bar\rho_\bv(t)-\rho_0\|_{\B^{d/p}_{p,1}} \\
&\lesssim \|\rho_0\|_{\B^{d/p}_{p,1}}  \|J_\bv^{-1}(t)-1\|_{\B^{d/p}_{p,1}}\\
& \lesssim \|\rho_0\|_{\B^{d/p}_{p,1}}\int_0^t\|D\bv\|_{\B^{d/p}_{p,1}}\,\d\tau
\end{aligned}$$
hence the proof may be easily completed. The details are left to the reader.   
\end{proof}


\appendix 
\begin{appendix}
\section{Results on the Lam\'e operator}
As a first, for the convenience of the reader, we recall the proof of  regularity estimates 
in Sobolev spaces for the Lam\'e operator. 
\begin{proof}[Proof of Proposition~\ref{Prop: Regularity for reduced operator on Lp}]  The first step is to prove that there exists  a constant $C > 0$ such that all solutions 
$u \in \W^{k + 2 , p} (\Omega ; \IC^d)$ to the equation
\begin{align*}
\left\{ \begin{aligned}
 - \mu \Delta u - z \nabla \divergence u &= f && \text{in } \Omega \\
 u &= 0 && \text{on } \partial \Omega
\end{aligned} \right.
\end{align*}
for some $f \in \W^{k , p} (\Omega ; \IC^d)$ satisfy
\begin{equation}\label{eq: Application of ADN}
 \| u \|_{\W^{k + 2 , p} (\Omega ; \IC^d)} \leq C \big( \| f \|_{\W^{k , p} (\Omega ; \IC^d)} + \| u \|_{\L^p (\Omega ; \IC^d)} \big)\cdotp
\end{equation}
In dimension $d=1,$ the result readily follows by integration. 
In the multi-dimensional case, it is a consequence of the theory of Agmon, Douglis, and Nirenberg
 (more precisely~\cite[Thm.~10.5]{Agmon_Douglis_Nirenberg}). 
 To verify the assumptions therein, define the symbol of $L$ by
\begin{align*}
 \cS (\xi) := \mu |\xi|^2 \Id + z \xi \otimes \xi \qquad (\xi \in \IR^d).
\end{align*}
\begin{lemma}
\label{Lem: ADN ellipticity}
Let $\mu > 0$ and $z \in \IC$ with $\mu + \Re(z) > 0$. Let $\delta \in (0 , 1)$ be any number that satisfies $\delta \mu  + \Re(z) \geq 0$. Then, for each $\xi \in \IR^d$ the determinant of $\cS (\xi)$ satisfies
\begin{align*}
 \mu^d (1 - \delta)^d 2^{- \frac{d}{2}} \lvert \xi \rvert^{2 d} \leq \lvert \det(\cS (\xi)) \rvert \leq \big( \mu + \lvert z \rvert \big)^d \lvert \xi \rvert^{2d}.
\end{align*}
\end{lemma}

\begin{proof}
The result for $z=0$ being obvious, assume from now on 
that $z\not=0.$
Let $M$ denote the matrix $M := \xi \otimes \xi$. Because  $M$ is real and symmetric, 
 $\cS(\xi)$ is diagonalizable. Let $\eta \in \IC^d$ be a unit eigenvector to $\cS (\xi)$ with corresponding eigenvalue $\alpha \in \IC$. Then,
\begin{align*}
 \alpha \eta = \cS (\xi) \eta = \mu \lvert \xi \rvert^2 \eta + z M \eta, \quad\hbox{hence}\quad
 z^{-1} (\alpha - \mu \lvert \xi \rvert^2) \eta = M \eta.
\end{align*}
Hence, $\eta$ is an eigenvector to $M$ with corresponding eigenvalue $z^{-1} (\alpha - \mu \lvert \xi \rvert^2)$. Since $M$ is real and symmetric, $\eta$ and $z^{-1} (\alpha - \mu \lvert \xi \rvert^2)$ must be real. Thus, keeping in mind that $|\eta|=1,$  we get
\begin{align*}
 \alpha  = \mu \lvert \xi \rvert^2  + z M \eta \cdot \eta = \mu \lvert \xi \rvert^2  + z [\xi \cdot \eta]^2.
\end{align*}
Let $\delta \in (0 , 1)$ be such that $\delta \mu + \Re(z) \geq 0$ holds. This combined with $\mu > 0$ and some trigonometry yields
\begin{align*}
 \lvert \alpha \rvert = \Big\lvert \mu \big( \lvert \xi \rvert^2 - \delta [\xi \cdot \eta]^2  \big) + (\delta \mu + z) [\xi \cdot \eta]^2\Big\rvert &\geq \frac{1}{\sqrt{2}} \big( \mu (1 - \delta) \lvert \xi \rvert^2 + (\delta \mu + \Re(z)) [\xi \cdot \eta]^2  \big) \\
 &\geq \frac{\mu (1 - \delta)}{\sqrt{2}} \lvert \xi \rvert^2.
\end{align*}
Consequently, the determinant of $\cS (\xi)$ satisfies
\begin{align*}
 \lvert \det (\cS (\xi)) \rvert \geq \mu^d (1 - \delta)^d 2^{- \frac{d}{2}} \lvert \xi \rvert^{2 d}.
\end{align*}
The other inequality follows from
\begin{align*}
 \lvert \alpha \rvert \leq \mu \lvert \xi \rvert^2 + \bigl|z [\xi \cdot \eta]^2\bigr| \leq \big( \mu + \lvert z \rvert \big) \lvert \xi \rvert^2. &\qedhere
\end{align*}
\end{proof}

If $d \geq 3$, then Lemma~\ref{Lem: ADN ellipticity} implies that the operator $- \mu \Delta - z \nabla \divergence$ is elliptic in the sense of Agmon, Douglis, and Nirenberg, and 
we get~\eqref{eq: Application of ADN}. For $d=2,$ one needs to verify the following supplementary condition.
\begin{lemma}
Let $d = 2$ and let $\xi , \xi^{\prime} \in \IR^2$ be linearly independent. Then, $\det (\cS (\xi + \tau \xi^{\prime}))$ regarded as a polynomial in the complex variable $\tau$ has exactly two roots with positive and two roots with negative imaginary part.
\end{lemma}

\begin{proof}
The determinant of $\cS (\xi + \tau \xi^{\prime})$ is calculated as
\begin{align}
\label{Eq: Supplemenaty condition}
 \det (\cS (\xi + \tau \xi^{\prime})) = \mu (\mu + z) \big[ (\xi + \tau \xi^{\prime}) \cdot (\xi + \tau \xi^{\prime}) \big]^2.
\end{align}
Due to the assumptions on $\mu$ and $z$, the prefactor cannot be zero. If there would be a real root to the equation $\det(\cS(\xi + \tau \xi^{\prime})) = 0$, then $\xi$ and $\xi^{\prime}$ would have to be linearly dependent, a contradiction. Thus,~\eqref{Eq: Supplemenaty condition} determines a fourth order polynomial in $\tau$ with real coefficients and no real roots. Hence, there must be two roots with positive and with negative imaginary part.
\end{proof}
\smallbreak
Let us now go to the existence part of the proposition. 
Clearly,  the case $p = 2$ follows from 
Proposition~\ref{Prop: Higher regularity for reduced operator on L2}.
The case $p>2$ will be also a consequence of it. 
Indeed, then  $L_p$ is injective as it is the part of $L_2$ in $\L^p (\Omega ; \IC^d)$. 
Next, to prove the surjectivity of $L_p,$  let us first consider $p_1 \geq 2$ satisfying
 $1 / p_1 - 1 / 2 \leq 2 / d,$ and let $f \in \C^{\infty} (\overline{\Omega} ; \IC^d)$. By Proposition~\ref{Prop: Higher regularity for reduced operator on L2} there exists a unique $u \in \dom(L_2)$ with $L_2 u = f$ and $u \in \W^{k + 4 , 2} (\Omega ; \IC^d)$. By Sobolev's embedding theorem, we conclude that $u \in \W^{k + 2 , p_1} (\Omega ; \IC^d)$. Thus, by virtue of Inequality~\eqref{eq: Application of ADN} we discover that there exists a constant $C > 0$ such that
\begin{align*}
 \| u \|_{\W^{k + 2 , p_1} (\Omega ; \IC^d)} \leq C \big( \| f \|_{\W^{k , p_1} (\Omega ; \IC^d)} + \| u \|_{\L^{p_1} (\Omega ; \IC^d)} \big)\cdotp
\end{align*}
Moreover, by Sobolev's embedding theorem and again by Proposition~\ref{Prop: Higher regularity for reduced operator on L2} followed by H\"older's inequality together with the boundedness of $\Omega$, we derive
\begin{align}
\label{Eq: Bootstrapped elliptic regularity}
\begin{aligned}
 \| u \|_{\W^{k + 2 , p_1} (\Omega ; \IC^d)} &\leq C \big( \| f \|_{\W^{k , p_1} (\Omega ; \IC^d)} + \| u \|_{\W^{2 , 2} (\Omega ; \IC^d)} \big) \\
 &\leq C \big( \| f \|_{\W^{k , p_1} (\Omega ; \IC^d)} + \| f \|_{\L^2 (\Omega ; \IC^d)} \big) \\
 &\leq C \| f \|_{\W^{k , p_1} (\Omega ; \IC^d)}.
\end{aligned}
\end{align}
To proceed let $p_2 \geq p_1$ with $1 / p_1 - 1 / p_2 \leq 2 / d$. By Proposition~\ref{Prop: Higher regularity for reduced operator on L2}, we now find $u \in \W^{k + 6 , 2} (\Omega ; \IC^d) \hookrightarrow \W^{k + 2 , p_2} (\Omega ; \IC^d)$. Inequality~\eqref{eq: Application of ADN} followed by Sobolev's embedding theorem then provide the estimate
\begin{align*}
 \| u \|_{\W^{k + 2 , p_2} (\Omega ; \IC^d)} \leq C \big( \| f \|_{\W^{k , p_2} (\Omega ; \IC^d)} + \| u \|_{\W^{2 , p_1} (\Omega ; \IC^d)} \big)\cdotp
\end{align*}
Combining this with~\eqref{Eq: Bootstrapped elliptic regularity} in the case $k = 0$, H\"older's inequality, and the boundedness of $\Omega$ we conclude that
\begin{align*}
 \| u \|_{\W^{k + 2 , p_2} (\Omega ; \IC^d)} \leq C \| f \|_{\W^{k , p_2} (\Omega ; \IC^d)}.
\end{align*}
Bootstrapping this argument delivers the stated estimate of the proposition for all $p \geq 2$. By density, we get~\eqref{Eq: Elliptic regularity} for all $f \in \W^{k , p} (\Omega ; \IC^d).$ Taking  $k = 0$ gives the surjectivity of $L_p$. 
\medbreak
Let us next consider the case $1 < p < 2.$ Then, the invertibility of its adjoint (as 
according to  Lemma~\ref{Lem: Duality}, it is equal to $(L_2^*)_{p'}$, 
and $L_2^*$ is $L_2$  with $z$ replaced by $\overline{z}$), and standard annihilator relations imply that $L_p$ is injective and has dense range. \par
Next, for $f \in \L^2 (\Omega ; \IC^d) \hookrightarrow \L^p (\Omega ; \IC^d),$ let $u \in \dom(L_2)$ be such that $L_2 u = f$. In this case, we already know that $u \in \W^{2 , 2} (\Omega ; \IC^d) \hookrightarrow \W^{2 , p} (\Omega ; \IC^d)$ is valid, and Inequality~\eqref{eq: Application of ADN} implies
\begin{align}
\label{Eq: A priori estimate p < 2}
 \| u \|_{\W^{2 , p} (\Omega ; \IC^d)} \leq C \big( \| f \|_{\L^p (\Omega ; \IC^d)} + \| u \|_{\L^p (\Omega ; \IC^d)} \big).
\end{align}
As $u \in \dom(L_p),$ there exists by definition a sequence $(u_n)_{n \in \IN} \subset \dom(L_2)$ which converges in $\L^p (\Omega ; \IC^d)$ to $u$ and for which $f_n := L_2 u_n$ converges in $\L^p (\Omega ; \IC^d)$ to $f := L_p u$. Estimate~\eqref{Eq: A priori estimate p < 2} implies then that $(u_n)_{n \in \IN}$ converges in $\W^{2 , p} (\Omega ; \IC^d).$ Hence~\eqref{Eq: A priori estimate p < 2} is valid for all $u \in \dom(L_p)$. \par
One can show that~\eqref{Eq: Elliptic regularity 2} with $k=0$ is valid by contradiction. Assuming the contrary, we obtain the existence of a sequence $(u_n)_{n \in \IN} \subset \dom(L_p)$ with $f_n := L_p u_n$ such that for all $n \in \IN$
\begin{align*}
 \| u_n \|_{\W^{2 , p} (\Omega ; \IC^d)} = 1 \qquad \text{and} \qquad \| f_n \|_{\L^p (\Omega ; \IC^d)} \to 0 \quad \text{as} \quad n \to \infty.
\end{align*}
By compactness (and by going over to a subsequence), $(u_n)_{n \in \IN}$ converges in $\L^p (\Omega ; \IC^d)$ to some $u \in \L^p (\Omega ; \IC^d)$. The closedness of $L_p$ then implies $u \in \dom(L_p)$ and $L_p u = 0$. Since we already know that $L_p$ is injective, it follows that $u = 0$. Now,~\eqref{Eq: A priori estimate p < 2} gives a contradiction and thus we infer that~\eqref{Eq: Elliptic regularity 2} for $k = 0$ is valid. This estimate in turn implies that the range of $L_p$ is closed and since it is dense in $\L^p (\Omega ; \IC^d)$, we deduce that $0 \in \rho(L_p)$. \par
Next, let $f \in \dom(L_p)$ and $u \in \dom(L_p^2)$ with $L_p u = f$. By definition, there exists $(f_n)_{n \in \IN} \subset \dom(L_2)$ with $f_n \to f$ and $L_2 f_n \to L_p f$ in $\L^p (\Omega ; \IC^d)$ as $n \to \infty$. By~\eqref{Eq: A priori estimate p < 2} it holds
\begin{align*}
 \| f_n - f_m \|_{\W^{2 , p} (\Omega ; \IC^d)} \leq C \big( \| L_2 (f_n - f_m) \|_{\L^p (\Omega ; \IC^d)} + \| f_n - f_m \|_{\L^p (\Omega ; \IC^d)} \big)\cdotp
\end{align*}
Thus, $(f_n)_{n \in \IN}$ is a Cauchy sequence in $\W^{2 , p} (\Omega ; \IC^d)$. Define $u_n := L_2^{-1} f_n \in \dom(L_2^2)$ and observe that $u_n \to u$ in $\L^p (\Omega ; \IC^d)$ as $n \to \infty$. Since $\dom(L_2^2) \hookrightarrow \W^{4 , 2} (\Omega ; \IC^d) \hookrightarrow \W^{4 , p} (\Omega ; \IC^d)$, Inequality~\eqref{eq: Application of ADN} guarantees that
\begin{align*}
 \| u_n - u_m \|_{\W^{4 , p} (\Omega ; \IC^d)} \leq C \big( \| f_n - f_m \|_{\W^{2 , p} (\Omega ; \IC^d)} + \| u_n - u_m \|_{\L^p (\Omega ; \IC^d)} \big)\cdotp
\end{align*}
In the limit, this implies that $u \in \W^{4 , p} (\Omega ; \IC^d)$ and
\begin{align*}
 \| u \|_{\W^{4 , p} (\Omega ; \IC^d)} \leq C \big( \| f \|_{\W^{2 , p} (\Omega ; \IC^d)} + \| u \|_{\L^p (\Omega ; \IC^d)} \big)\cdotp
\end{align*}
As above,~\eqref{Eq: Elliptic regularity 2} for $k = 1$ follows from a contradiction argument. The case $k \geq 2$ follow the same strategy by iterating this argument.
\medbreak
Finally, using what we just proved in the case $k=0$ in the definition of $\dom(L_2)$  ensures 
that $\dom(L_p)\hookrightarrow \W^{2,p}(\Omega;\IC^d)\cap\W^{1,p}_0(\Omega;\IC^d),$
and the reverse embedding is obvious.
\end{proof}

The following lemma clarifies the relationships between $L_p,$ $\cL_p$ and $\wt\cL_p.$
\begin{lemma}
\label{Lem: Extrapolation lemma}
Let $1 < p < \infty$. Under the notations in~\eqref{eq:Fi} with $r=p',$ and~\eqref{eq:cLp}, 
the following statements hold true:
\begin{enumerate}
 \item \label{Item: Recovering extended operator} For all $u \in \dom(L_p)$ it holds 
 \begin{align*}
  \Phi^{-1} \widetilde{\cL}_p \Phi u = L_p u.
 \end{align*}
 \item \label{Item: Represenation of inverse} For all $f \in \L^p (\Omega ; \IC^d)$ it holds
 \begin{align*}
  \Phi^{-1} \cL_p^{-1} \Phi f = L_p^{-1} f.
 \end{align*}
 \item \label{Item: Similarity of operators} For    $T:=\widetilde{\cL}_p\Phi,$  we have that $T:\L^p (\Omega ; \IC^d) \to X^{-1}_p$ is an isomorphism and that
 \begin{align*}
   \cL_p= T L_p T^{-1}.
 \end{align*}
 \item \label{Item: Part is adjoint} If ${\bf L}_p$ denotes the part of $\cL_p$ in $\L^{p^{\prime}} (\Omega ; \IC^d)^{\prime}$, then it holds
\begin{align*}
 {\bf L}_p = (L_2^*)_{p^{\prime}}^{\prime}.
\end{align*}
 \item \label{Item: Recovering original operator} It holds 
 \begin{align*}
  \Phi^{-1} {\bf L}_p \Phi = L_p.
 \end{align*}
\end{enumerate}
\end{lemma}
\begin{proof}
\begin{enumerate}
 \item Let $u \in \dom(L_p)$. Then, by virtue of the definition of $\widetilde{\cL}_p$, the definition given in~\eqref{Eq: Lp adjoint}, and Lemma~\ref{Lem: Duality}, we have
\begin{align*}
 \Phi^{-1} \widetilde{\cL}_p \Phi u = \Phi^{-1} (L_2^*)_{p^{\prime}}^{\circ} \Phi u = (L^*_2)_{p^{\prime}}^* u = L_p u.
\end{align*}
 \item This is just a reformulation of~\eqref{Item: Recovering extended operator}.
 \item Notice that since $L_p$ maps into $\L^p (\Omega ; \IC^d)$ it holds $\dom(\widetilde{\cL}_p \Phi L_p \Phi^{-1} \widetilde{\cL}_p^{-1}) = \dom(L_p \Phi^{-1} \widetilde{\cL}_p^{-1})$. \par
Let $\fu \in \dom(L_p \Phi^{-1} \widetilde{\cL}_p^{-1})$. Since $L_p$ is invertible, there exists $f \in \L^p (\Omega ; \IC^d)$ such that
\begin{align*}
 L_p^{-1} f = \Phi^{-1} \widetilde{\cL}_p^{-1} \fu.
\end{align*}
Applying~\eqref{Item: Represenation of inverse} delivers $f = \Phi^{-1} \fu$ and it follows that $\fu \in \dom(\cL_p)$. Furthermore, another application of~\eqref{Item: Represenation of inverse} yields
\begin{align*}
TL_pT^{-1}u= \widetilde{\cL}_p \Phi L_p \Phi^{-1} \widetilde{\cL}_p^{-1} \fu = \widetilde{\cL}_p \fu = \cL_p \fu.
\end{align*}
\indent Conversely, let $\fu \in \dom(\cL_p) = \L^{p^{\prime}} (\Omega ; \IC^d)^{\prime}$. Then~\eqref{Item: Recovering extended operator} implies
\begin{align*}
 \fu = \widetilde{\cL}_p \Phi \Phi^{-1} \widetilde{\cL}_p^{-1} \Phi \Phi^{-1} \fu = \widetilde{\cL}_p \Phi L_p^{-1} \Phi^{-1} \fu.
\end{align*}
It follows that $\Phi^{-1} \widetilde{\cL}_p^{-1} \fu \in \dom(L_p)$ and thus that $\fu \in \dom(L_p \Phi^{-1} \widetilde{\cL}_p^{-1})$.
 \item Let $\fu \in \dom({\bf L}_p)$. Then by definition of the part of an operator, it holds $\fu \in \L^{p^{\prime}} (\Omega ; \IC^d)^{\prime}$ and $(L^*_2)_{p^{\prime}}^{\circ} \fu \in \L^{p^{\prime}} (\Omega ; \IC^d)^{\prime}$. In particular, there exists $w \in \L^{p^{\prime}} (\Omega ; \IC^d)^{\prime}$ such that for all $v \in \dom((L_2^*)_{p^{\prime}})$ it holds
\begin{align*}
 \langle (L^*_2)_{p^{\prime}}^{\circ} \fu , v \rangle_{\dom((L_2^*)_{p^{\prime}})^{\prime} , \dom((L_2^*)_{p^{\prime}})} = \langle w , v \rangle_{(\L^{p^{\prime}})^{\prime} , \L^{p^{\prime}}}.
\end{align*}
Consequently, 
\begin{align*}
 \langle \fu , (L_2^*)_{p^{\prime}} v \rangle_{(\L^{p^{\prime}})^{\prime} , \L^{p^{\prime}}} = \langle (L^*_2)_{p^{\prime}}^{\circ} \fu , v \rangle_{\dom((L_2^*)_{p^{\prime}})^{\prime} , \dom((L_2^*)_{p^{\prime}})} = \langle w , v \rangle_{(\L^{p^{\prime}})^{\prime} , \L^{p^{\prime}}}. 
\end{align*}
This implies that $\fu \in \dom((L_2^*)_{p^{\prime}}^{\prime})$ and that $(L_2^*)_{p^{\prime}}^{\prime} \fu = w$. \par
Conversely, let $\fu \in \dom((L_2^*)_{p^{\prime}}^{\prime})$. By definition, it holds $\fu \in \L^{p^{\prime}} (\Omega ; \IC^d)^{\prime}$ and there exists $w \in \L^{p^{\prime}} (\Omega ; \IC^d)^{\prime}$ such that for all $v \in \dom((L_2^*)_{p^{\prime}})$ it holds
\begin{align*}
 \langle \fu , (L_2^*)_{p^{\prime}} v \rangle_{(\L^{p^{\prime}})^{\prime} , \L^{p^{\prime}}} = \langle w , v \rangle_{(\L^{p^{\prime}})^{\prime} , \L^{p^{\prime}}}.
\end{align*}
Thus,
\begin{align*}
 \langle (L^*_2)_{p^{\prime}}^{\circ} \fu , v \rangle_{\dom((L_2^*)_{p^{\prime}})^{\prime} , \dom((L_2^*)_{p^{\prime}})} = \langle \fu , (L_2^*)_{p^{\prime}} v \rangle_{(\L^{p^{\prime}})^{\prime} , \L^{p^{\prime}}} = \langle w , v \rangle_{(\L^{p^{\prime}})^{\prime} , \L^{p^{\prime}}}. 
\end{align*}
It follows that $(L^*_2)_{p^{\prime}}^{\circ} \fu \in \L^{p^{\prime}} (\Omega ; \IC^d)^{\prime}$ and thus $u \in \dom({\bf L}_p)$.
 \item This readily follows by combining~\eqref{Item: Part is adjoint} with~\eqref{Eq: Lp adjoint} and Lemma~\ref{Lem: Duality}. \qedhere
\end{enumerate}
\end{proof}

\begin{lemma}
\label{Lem: Domain of part}
For all $1 < p < \infty$, $1 \leq q \leq \infty$, and $- 1 < s < 1$ it holds with equivalent norms
that $\dom({\bf L}_{p , q , s}) = Y^{s + 1}_{p , q}.$

Furthermore, if $\theta\in(0,1)$ and $s+\theta<1$, then  the part of ${\bf L}_{p , q , s}$ on $X^{s + \theta}_{p , q} \simeq \B_{p , q}^{2 (s  + \theta)} (\Omega ; \IC^d)$ coincides with ${\bf L}_{p , q , s + \theta}$. 
\end{lemma}
\begin{proof}
First of all, recall that $\cL_p$ is invertible and that its inverse is a bounded operator
\begin{align}
\label{Eq: First interpolation endpoint}
 \cL_p^{-1} : X_p^{-1} \to X_p^0.
\end{align}
If $\ff \in X^1_p$, then $\ff$ can be written as $\ff = \Phi f$ for some $f \in \dom(L_p)$. Now, Lemma~\ref{Lem: Extrapolation lemma}~\eqref{Item: Represenation of inverse} implies
\begin{align*}
 \cL_p^{-1} \ff = \Phi L_p^{-1} \Phi^{-1} \ff = \Phi L_p^{-1} f.
\end{align*}
By virtue of Proposition~\ref{Prop: Regularity for reduced operator on Lp}, we thus have
\begin{align*}
 \| \cL_p^{-1} \ff \|_{X^2_p} = \| L_p^2 L_p^{-1} f \|_{\L^p (\Omega ; \IC^d)} \leq C \| L_p^{-1} f \|_{\W^{4 , p} (\Omega ; \IC^d)} \leq C \| f \|_{\W^{2 , p} (\Omega ; \IC^d)} \leq C \| \ff \|_{X^1_p}.
\end{align*}
It follows that $\cL_p^{-1}$ gives rise to a bounded operator
\begin{align}
\label{Eq: Second interpolation endpoint}
 \cL_p^{-1} : X_p^1 \to X_p^2.
\end{align}
Interpolating~\eqref{Eq: First interpolation endpoint} and~\eqref{Eq: Second interpolation endpoint} reveals that for all $- 1 < s < 1$, $1 < p < \infty$, and $1 \leq q \leq \infty$ the operator $\cL_p^{-1}$ is a bounded operator 
\begin{align}
\label{Eq: Interpolated mapping properties}
 \cL_p^{-1} : X^s_{p , q} \to Y^{s + 1}_{p , q}.
\end{align}
\indent Let $\fu \in \dom({\bf L}_{p , q , s})$. Then, by~\eqref{Eq: Interpolated mapping properties}
\begin{align*}
 \fu = \cL_p^{-1} \cL_p \fu \in Y^{s + 1}_{p , q}.
\end{align*}
Moreover, since $\cL_p \fu = {\bf L}_{p , q , s} \fu$, the boundedness stated in~\eqref{Eq: Interpolated mapping properties} implies that there exists $C > 0$ such that
\begin{align*}
 \| \fu \|_{Y^{s + 1}_{p , q}} \leq C \| {\bf L}_{p , q , s} \fu \|_{X^s_{p , q}}.
\end{align*}
\indent Conversely, let $\fu \in Y^{s + 1}_{p , q}$. Since $\dom (\cL_p) = \L_{p^{\prime}} (\Omega ; \IC^d)^{\prime}$ and since $Y^{s + 1}_{p , q} \hookrightarrow X^0_p = \L_{p^{\prime}} (\Omega ; \IC^d)^{\prime}$ (cf.~\eqref{Eq: Embedding Y^t}), we have $\fu \in \dom(\cL_p)$. By~\eqref{Eq: Interpolated mapping properties}, we find $\cL_p \fu \in X^s_{p , q}$ and the only information we need, to conclude that $\fu \in \dom({\bf L}_{p , q , s})$, is that $\fu \in X^s_{p , q}$. This, however, follows by Proposition~\ref{Prop: Identification ground spaces}. Finally, the inequality follows from 
\begin{align*}
 \| {\bf L}_{p , q , s} \fu \|_{X^s_{p , q}} = \| \cL_p \fu \|_{X^s_{p , q}} \leq C \| \fu \|_{Y^{s + 1}_{p , q}}. 
\end{align*}
Finally,  to prove the second part of the lemma, we use that the domain of the part of ${\bf L}_{p , q , s}$ on $X^{s + \theta}_{p , q}$ is by definition given as
\begin{align*}
 \big\{ \fu \in \dom({\bf L}_{p , q , s}) \cap X^{s + \theta}_{p , q} &: {\bf L}_{p , q , s} \fu \in X^{s + \theta}_{p , q} \big\} \\
 &= \big\{ \fu \in \dom(\cL_p) \cap X^{s}_{p , q} \cap X^{s + \theta}_{p , q} : \cL_p \fu \in X^{s}_{p , q} \cap X^{s + \theta}_{p , q} \big\} \\
 &= \big\{ \fu \in \dom(\cL_p) \cap X^{s + \theta}_{p , q} : \cL_p \fu \in X^{s + \theta}_{p , q} \big\} \\
 &= \dom({\bf L}_{p , q , s + \theta}). \qedhere
\end{align*}
\end{proof}


\section{Some results for Besov spaces in domains}


\begin{proof}[Proof of Proposition~\ref{p:product}] Consider two real valued functions 
$u\in \B^s_{p,1}(\Omega) $ and $v\in\B^{d/p}_{p,1}(\Omega).$
We want to prove that $uv$ lies in $\B^s_{p,1}(\Omega),$
if $-\min(d/p,d/p')<s\leq d/p.$ 
The result  is classical for $\Omega=\IR^d$
and the general domain case follows from the definition of Besov spaces by restriction
given in Section \ref{s:lame}. 
Indeed,  if $u\in \B^s_{p,1}(\Omega)$ and $v\in \B^{d/p}_{p,1}(\Omega),$ then 
for any extension $\wt u\in \B^s_{p,1}(\IR^d)$ and $\wt v\in \B^{d/p}_{p,1}(\IR^d)$
of $u$ and $v$ on $\IR^d,$ we may write 
$$\|\wt u\,\wt v\|_{\B^s_{p,1}(\IR^d)}\lesssim \|\wt u\|_{\B^s_{p,1}(\IR^d)}
\|\wt v\|_{\B^{d/p}_{p,1}(\IR^d)}.$$
As $\wt u\,\wt v$ is  an extension of $uv$ on $\IR^d,$ taking the infimum on 
all extensions gives the result. 
\end{proof}

\begin{proof}[Proof of Proposition~\ref{p:compo}] 
Looking at the proof of~\cite[Prop.~1.7]{D-Chambery} and using the embedding
of ${\B^{d/p}_{p,1}(\IR^d)}$ in $\L^\infty(\IR^d),$ we see that in the $\IR^d$ case, 
we do have the result with the estimate  
$$\|K(z)\|_{\B^{d/p}_{p,1}(\IR^d)}\leq C (1+ \|z\|_{\B^{d/p}_{p,1}(\IR^d)})^k
  \|z\|_{\B^{d/p}_{p,1}(\IR^d)}\with k:=\lceil d/p\rceil.$$
  The result in a general domain then follows, considering 
  all the extensions $\wt z\in \B^{d/p}_{p,1}(\IR^d)$ of $z\in \B^{d/p}_{p,1}(\Omega),$ 
  then taking the infimum. 
  
  The second part of the proposition follows from  the first part, the following formula:
  $$K(z_2)-K(z_1)=K'(0)(z_2-z_1)+\int_0^1\bigl(K'(z_1+\tau(z_2-z_1))-K'(0)\bigr)(z_2-z_1)\,\d\tau$$
    and Proposition~\ref{p:product}.
    \end{proof}

Property~\eqref{eq:unifB} is a consequence of the following proposition.
    \begin{proposition} Let $f$ be in $\B^{d/p}_{p,1}(\IR^d)$ for some $1\leq p\leq\infty.$ 
    Let $\psi$ be a smooth function, supported in the unit ball of $\IR^d.$ 
    Denote $\psi_{\delta, x_0}:=\psi(\delta^{-1}(\cdot-x_0))$ for $\delta>0$ and $x_0\in\IR^d.$     
    Then, 
    $$\lim_{\delta\to0} \|\psi_{\delta, x_0}\: (f - f(x_0))\|_{\B^{d/p}_{p,1}}=0\quad\hbox{uniformly with respect to } x_0.$$
    \end{proposition}
    \begin{proof}
    Let us first establish the result for $g$ a smooth function with bounded derivatives at all order. Let without loss of generality $\delta \in (0 , 1)$.
    We first notice, owing to the mean value theorem and the fact that  $\psi_{\delta,x_0}$ is supported in a ball of radius $\delta$ that
    $$\|\nabla^\alpha\psi_{\delta, x_0} (g-g(x_0))\|_{\L^p}\leq C    \|\nabla g\|_{\L^\infty} \delta^{1+\frac dp-\alpha}\quad \hbox{for all } \ \alpha\in\IN.$$
    Next, we see that for any couple $(\beta,\gamma)$ of integers with $\gamma\geq1,$
    $$    \|\nabla^\beta\psi_{\delta, x_0}\:\nabla^\gamma(g-g(x_0))\|_{\L^p} \leq C   \|\nabla^\gamma g\|_{\L^\infty} \delta^{\frac dp-\beta}.$$
    Consequently, in light of Leibniz formula, for all integer $\alpha$
    there exists a constant $C_\alpha > 0$ depending only on $g$ and such that for all $x_0\in\IR^d$ and $\delta\in(0,1),$ 
   $$\|\nabla^\alpha(\psi_{\delta, x_0} \,(g-g(x_0)))\|_{\L^p}\leq C_\alpha \delta^{\frac dp+1-\alpha}.$$ 
    If $\alpha \in \IN$ is such that $d/p < \alpha \leq d/p + 1$, the exponent of $\delta$ in the previous inequality is non-negative. Moreover, combining the derived estimates with the following interpolation inequality 
    $$    \|h\|_{\B^{d/p}_{p,1}}\leq C\|h\|_{\L^p}^{1-\frac d{p\alpha}}\|h\|_{\W^{\alpha,p}}^{\frac d{p\alpha}}$$
    and the assumption that $\delta < 1$ yields that there exists  a constant $C_g > 0$ depending only on $g,$ $p$ and $d$ such that 
    \begin{equation}\label{eq:psi}\|\psi_{\delta, x_0} \, (g-g(x_0))\|_{\B^{d/p}_{p,1}}\leq C_g\delta^{(1 + \frac{d}{p})(1 - \frac{d}{p \alpha})}\quad\hbox{for all }\ \delta\in(0,1)\andf x_0\in\IR^d.
    \end{equation}
    Let us now prove the proposition for a general function $f$ in $\B^{d/p}_{p,1}.$ Fix some $\eps>0$ and 
    take $g$ smooth with bounded derivatives at all order such that $\|f-g\|_{\B^{d/p}_{p,1}}\leq\eps.$
    We have 
    $$
    \|\psi_{\delta, x_0} \, (f-f(x_0))\|_{\B^{d/p}_{p,1}}\leq
    \|\psi_{\delta, x_0} \, (g-g(x_0))\|_{\B^{d/p}_{p,1}} +   \|\psi_{\delta, x_0} \,(f-g)\|_{\B^{d/p}_{p,1}} +|f(x_0)-g(x_0)|
    \| \psi_{\delta,x_0} \|_{\B^{d/p}_{p,1}}.$$
    Using Proposition~\ref{p:product}, Inequality~\eqref{eq:psi} and the embedding $\B^{d/p}_{p,1}\hookrightarrow \L^\infty,$ we thus have
    $$\|\psi_{\delta, x_0} \, (f-f(x_0))\|_{\B^{d/p}_{p,1}}\leq C_g\delta^{(1 + \frac{d}{p})(1 - \frac{d}{p \alpha})} +  C\|\psi_{\delta, x_0}\|_{\B^{d/p}_{p,1}}\|f-g\|_{\B^{d/p}_{p,1}}.$$
        Using the invariance (up to an harmless constant) of the norm in $\B^{d/p}_{p,1}(\IR^d)$ by translation and dilation, 
        and the definition of $g,$ we end up with 
         $$    \|\psi_{\delta, x_0} \, (f-f(x_0))\|_{\B^{d/p}_{p,1}}\leq  C_g\delta^{(1 + \frac{d}{p})(1 - \frac{d}{p \alpha})} +  C\eps,$$
         which ensures    $$    \|\psi_{\delta, x_0} \, (f-f(x_0))\|_{\B^{d/p}_{p,1}}\leq 2 C\eps$$ provided $\delta$ is small enough.     
  \end{proof}
\end{appendix}


\end{document}